\documentclass[11pt, twoside]{amsart}
\usepackage{shadethm,framed}
\usepackage[utf8]{inputenc}
\usepackage[english]{babel}
\usepackage[papersize={21cm,29.7cm},left=2.5cm,right=2.5cm,top=3cm,bottom=3cm]{geometry}
\usepackage{enumerate}
\usepackage{listings}
\usepackage{amsthm, amsfonts, amssymb, amsmath}
\usepackage{graphicx}
\usepackage{xcolor}
\usepackage{color}
\usepackage{framed}
\usepackage{wallpaper}
\usepackage{tikz,tikz-cd}
\usepackage{pgf,tikz,pgfplots}
\usepackage{tcolorbox}
\usepackage{scalefnt}
\usepackage{mdframed}
\usepackage{booktabs}
\usepackage{caption}
\usepackage{makecell}
\usepackage{float}
\usepackage{calrsfs}
\usepackage{transparent}
\usepackage{adjustbox}
\usepgfplotslibrary{patchplots}
\usetikzlibrary{patterns, positioning, arrows,decorations.markings}

\definecolor{blou}{rgb}{0.66, 0.44, 0.98}
\usepackage[hidelinks,urlcolor=blue]{hyperref}
\hypersetup{
    colorlinks,
    linkcolor={blue!50!black},
    citecolor=blou,
    urlcolor={blue!80!black}
}

\usepackage{pgfplots}
\pgfplotsset{width=7cm, compat=1.10}
\usepgfplotslibrary{fillbetween}
\pgfmathdeclarefunction{poly}{0}{\pgfmathparse{sin(x)+x}}

\tikzcdset{arrow style=tikz,diagrams={>={Straight Barb[scale=0.8]},all arrow/.append style = -{Latex[scale=0.8]}}}

\tcbset{
    frame code={}
    center title,
    left=0pt,
    right=0pt,
    top=0pt,
    bottom=0pt,
    colback=red!70,
    colframe=white,
    width=\dimexpr\textwidth\relax,
    enlarge left by=0mm,
    boxsep=5pt,
    arc=0pt,outer arc=0pt,
    }

\newtheoremstyle{pourdef}
  {10pt}
  {10pt}
  {}
  {}
  {\bf}
  {.~}
  { }
  {}
\newtheoremstyle{pourth}{10pt}{10pt}{\em}{}{\sc}{.~}{ }{}
\newtheoremstyle{pourpp}{10pt}{10pt}{\em}{}{\bf \em}{.~}{ }{}
\newtheoremstyle{pourrk}{10pt}{10pt}{}{}{\em}{.~}{ }{}
\newtheoremstyle{pourlm}{10pt}{10pt}{\em}{}{\bf \em}{.~}{ }{}
\newtheoremstyle{pourco}{10pt}{10pt}{\em}{}{\bf \em}{.~}{ }{}

\definecolor{black}{cmyk}{1,1,1,1}
\definecolor{colordef}{rgb}{0.2,0.5,0.07} 
\definecolor{colorprop}{rgb}{0.2,0.1,0.5}

\definecolor{color1}{rgb}{0.6,0.4,0.8}
\definecolor{color1}{rgb}{0.9,0.6,0.4}
\definecolor{color1}{rgb}{0.36, 0.54, 0.66}

\definecolor{color2}{rgb}{0.2,0.1,0.5}
\definecolor{color2}{rgb}{0.91, 0.84, 0.42}
\definecolor{color2}{rgb}{0.87, 0.36, 0.51}
\definecolor{color2}{rgb}{0.4, 0.69, 0.2}
\definecolor{color2}{rgb}{0.84, 0.23, 0.24}

\definecolor{color3}{rgb}{1.0, 0.13, 0.32}
\definecolor{color3}{rgb}{0.54, 0.17, 0.89}

\definecolor{babypink}{rgb}{0.96, 0.76, 0.76}

\definecolor{babyblue}{rgb}{0.76, 0.76, 0.95}

\makeatletter
\renewenvironment{proof}[1][\proofname]{\par
  \normalfont
  \trivlist
  \item[\hskip\labelsep\itshape
    #1.]\ignorespaces
}{%
  \endtrivlist
}
\makeatother

\newtheorem{lem}{Lemma}[section]
\newtheorem{cor}[lem]{Corollary}

\newtheorem{thm}[lem]{Theorem}
\newtheorem*{reptheorem}{Theorem}

\newtheorem{fact}[lem]{Fact}

\newtheorem{prop}[lem]{Proposition}
\theoremstyle{definition}
\newtheorem{definition}[lem]{Definition}
\newtheorem{ex}[lem]{Example}

\newtheorem{rmk}[lem]{Remark}

\newtheorem{questions}[lem]{Questions}

\newtheorem{notation}[lem]{Notation}

\numberwithin{equation}{section}

\newcommand{\compact}{\mathsf{K}}
\newcommand{\LP}{\mathfrak{p}}

\newcommand{\SB}{\mathbf{Sb}}
\newcommand{\Ein}{\operatorname{Ein}}

\newcommand{\Lag}{\operatorname{Lag}}
\def\O{\Omega}
\newcommand{\Diams}{\mathbf{D}}
\newcommand{\I}{\mathbf I}

\newcommand{\id}{\operatorname{id}}
\newcommand{\J}{\mathbf J}
\newcommand{\Fl}{\mathcal{F}}
\newcommand{\Phot}{\Lambda}
\newcommand{\Affstd}{\mathbb{A}}

\newcommand{\Cartinv}{\sigma_{\g}}
\newcommand{\hyp}{\operatorname{Z}}

\newcommand{\uu}{\mathfrak{u}}
\newcommand{\g}{\mathfrak{g}}

\newcommand{\aaa}{\mathfrak{a}}
\newcommand{\oppinv}{\mathsf{i}}
\newcommand{\Mat}{\operatorname{Mat}}
\newcommand{\tr}{\mathsf{t}}

\newcommand{\FS}{\operatorname{\Delta}}
\newcommand{\Kf}{\mathbb{K}}
\newcommand{\Cf}{\mathbb{C}}
\newcommand{\Rf}{\mathbb{R}}
\newcommand{\Hf}{\mathbb{H}}

\def\b{\mathbf{b}}

\newcommand{\Ad}{\operatorname{Ad}}
\newcommand{\ad}{\operatorname{ad}}
\newcommand{\leng}{\mathsf{len}_{\O}}
\newcommand{\plongsl}{\operatorname{j}}

\newcommand{\E}{\operatorname{E}}
\newcommand{\F}{\operatorname{F}}
\newcommand{\He}{\operatorname{H}}

\newcommand{\e}{\varepsilon}

\newcommand{\Aut}{\mathsf{Aut}}

\newcommand{\HTT}{\operatorname{HTT}}

\newcommand{\SL}{\operatorname{SL}}
\newcommand{\PO}{\operatorname{PO}}
\newcommand{\GL}{\operatorname{GL}}
\newcommand{\PGL}{\operatorname{PGL}}
\newcommand{\SU}{\operatorname{SU}}
\newcommand{\Sp}{\operatorname{Sp}}
\newcommand{\spp}{\mathfrak{sp}}

\newcommand{\sll}{\mathfrak{sl}}

\newcommand{\SO}{\operatorname{SO}}
\newcommand{\soo}{\mathfrak{so}}
\renewcommand{\qed}{\hfill\square}
\DeclareMathOperator{\Ker}{ker}
\newcommand{\Proj}{\mathbf{p}}

\newcommand{\borc}{\partial_i \mathcal{C}}

\newcommand{\suu}{\mathfrak{su}}
\newcommand{\eeee}{\mathfrak{e}}
\newcommand{\opp}{-}

\newcommand{\std}{\mathsf{std}}
\newcommand{\Gambor}{\partial_{\infty}\Gamma}
\newcommand{\limorb}{\Lambda_{\O}^{\operatorname{orb}}(\Gamma)}

\newcommand{\Od}{\O^{**}_0}

\newcommand{\indx}{\mathsf{idx}}
\newcommand{\inver}{\mathsf{s}}
\newcommand{\typ}{\mathsf{typ}}
\newcommand{\pos}{\mathsf{pos}}

\makeatletter
\newcommand{\tpitchfork}{%
  \vbox{
    \baselineskip\z@skip
    \lineskip-.52ex
    \lineskiplimit\maxdimen
    \m@th
    \ialign{##\crcr\hidewidth\smash{$-$}\hidewidth\crcr$\pitchfork$\crcr}
  }%
}
\makeatother

\title[Transverse groups preserving proper domains]{Transverse groups preserving proper domains in flag manifolds}
\author{Blandine Galiay}
\date{}

\begin{document}

\begin{abstract}
Given a semisimple Lie group~$G$ and a self-opposite flag manifold~$\Fl$ of~$G$, we establish a necessary condition for an infinite subgroup~$H$ of~$G$ to preserve a proper domain in~$\Fl$. In the case where $G$ is a Hermitian Lie group of tube type, we introduce and study a notion of causal convexity in the Shilov boundary~$\SB(\g)$ of the symmetric space of~$G$, inspired by the one already existing in conformal Lorentzian geometry. We show that subgroups~$H$ of~$G$ that are transverse with respect to a parabolic subgroup of~$G$ defining~$\SB(\g)$ and that preserve a proper domain in~$\SB(\g)$ satisfy a geometric property with respect to this causal convexity, close to the strong projective convex cocompactness defined by Danciger--Guéritaud--Kassel. This result highlights the spatial nature of the dynamics of~$H$. We construct Zariski-dense examples of such transverse subgroups.  
\end{abstract}

\maketitle

\section{Introduction}

A manifold $M$ admits a \emph{$(G,X)$-structure}, where $G$ is a real Lie group and $X$ is a $G$-homogeneous space, if there exists an atlas of charts on $M$ with values in $X$ whose changes of charts are given by elements of $G$. The manifold $M$ endowed with this structure is called a \emph{$(G,X)$-manifold}. In this paper, we investigate $(G,X)$-structures, where~$G$ is a real semisimple Lie group and~$X$ is a \emph{flag manifold of~$G$}. A \emph{flag manifold} of~$G$ is the quotient~$G/P$ of~$G$ by a \emph{parabolic subgroup}~$P$ of $G$, i.e.\ the stabilizer of a point $x$ on the visual boundary of the Riemannian symmetric space~$\mathbb{X}_G$ associated with~$G$. If there exists a bi-infinite geodesic in~$\mathbb{X}_G$ connecting $x$ to another point $y \in \partial \mathbb{X}_G$, then the stabilizer~$P^\opp$ of~$y$ in~$G$ is a \emph{parabolic subgroup opposite} to~$P$, and the manifold~$G/P^\opp$ is the flag manifold \emph{opposite} to~$G/P$. When~$P$ and~$P^\opp$ are conjugate in~$G$, we say that~$P$ and~$G/P$ are \emph{self-opposite}, in which case~$G/P \simeq G/P^\opp$.

\subsection{Quotients of proper domains}\label{sect_intro_falg_mfds} We are particularly interested in~$(G, G/P)$-manifolds~$M$ that are quotients of the form~$\Omega / \Gamma$, where~$\Gamma \leq G$ preserves a \emph{proper domain}~$\Omega \subset G/P$, i.e.\ a nonempty connected open subset of~$G/P$ whose closure avoids a subset of~$G/P$ of the form  
\vspace{-3pt}
\begin{equation}\label{eq_hyp_Zz}
    \hyp_z = \{ x \in G/P \mid x \text{ is not transverse to } z \}, \quad \text{where } z \in G/P^\opp,
\vspace{-3pt}
\end{equation}
called a \emph{maximal proper Schubert subvariety}. These algebraic subvarieties of~$G/P$ play a key role in this paper, through the notion of \emph{dual convexity} introduced by Zimmer~\cite{zimmer2018proper}: a domain~$\O \subset G/P$ is said to be \emph{dually convex} if for every point~$x \in \partial \O$, there exists~$\xi \in G/P^\opp$ such that~$ x \in \hyp_\xi$ and~$\O \cap \hyp_\xi = \emptyset$. In other words, the hypersurface~$\hyp_\xi$ is a supporting hypersurface to~$\O$ at~$x$. This definition generalizes the classical notion of convexity in real projective space, through its dual characterization. 

The case where~$\Omega/\Gamma$ is compact has been the subject of extensive study. When~$G = \PGL(n, \mathbb{R})$ and~$X = \mathbb{P}(\mathbb{R}^n)$, this is part of the rich theory of \emph{divisible convex sets}, which has yielded many examples; see e.g.\ \cite{vinberg1965structure, goldman1990convex, benoist2000automorphismes, benoist2003convexes, benoist2005convexes, benoist2006convexes, cooper2015convex, choi2020convex, zimmer2020higher, blayac2024boundary, islam2019rank}. When~$X$ is a flag manifold other than the real projective space or the conformal sphere, one expects a strong rigidity on the domain~$\O$, as conjectured by Limbeek--Zimmer \cite{van2019rigidity} and proved in certain cases in \cite{zimmer2013rigidity, zimmer2018proper, van2019rigidity, galiay2024rigidity, chalumeau2024rigidity}. In this paper, we are interested in more flexible situations: rather than requiring~$\Gamma$ to act cocompactly on~$\O$, we instead ask that~$\Gamma$ be a~\emph{$P$-transverse} group, in the sense of Section~\ref{sect_anosov_geometric} below.

\subsection{Quotients of symmetric spaces and rigidity} 
Another well-studied example of~$(G,X)$-manifolds is that of \emph{locally symmetric spaces}; these are quotients of the form~$\mathbb{X}_G / \Gamma$, where~$\mathbb{X}_G$ is the symmetric space associated with a semisimple Lie group~$G$, and~$\Gamma \leq G$ is a discrete subgroup. 

For instance, a hyperbolic~$n$-manifold~$M$ is \emph{convex cocompact} (resp.\ \emph{geometrically finite}) if its fundamental group~$\Gamma \leq \PO(n, 1)$ preserves and acts cocompactly (resp.\ with finite-volume fundamental domain) on a closed convex subset of the real hyperbolic space~$\mathbb{X}_{\PO(n, 1)} = \mathbb{H}^n$. In this case, the manifold~$M$ can be identified with the quotient~$\mathbb{H}^n / \Gamma$ and its convex core is compact (resp.\ the union of a compact subset and a finite number of \emph{cusps}). Many examples of noncompact convex cocompact (resp.\ geometrically finite) hyperbolic manifolds exist (see e.g.\ \cite{kassel2018geometric}). However, any Zariski-dense discrete subgroup of a real simple Lie group $G$ of real rank $r \geq 2$ acting cocompactly on a closed (geodesically) convex subset of $\mathbb{X}_G$ is a lattice in $G$ \cite{quint2005groupes, kleiner2006rigidity}. Thus, the intuitive geometric generalization of convex cocompactness becomes rigid in higher rank.

\subsection{Anosov representations and geometric structures}\label{sect_anosov_geometric}
By contrast, a generalization of convex cocompactness based on the dynamical properties of convex cocompact discrete subgroups of $\PO(n,1)$ has been developed over the past twenty years, through \emph{Anosov representations} (see Section~\ref{sect_anosov} for a definition). These were first introduced by Labourie \cite{labourie2006anosov} in his study of Hitchin representations of surface groups and were later further investigated by Guichard--Wienhard \cite{guichard2012anosov} and many other authors. These representations are discrete, with finite kernel, and structurally stable, making them a key concept in recent developments in higher Teichmüller theory \cite{guichard2012anosov, wienhard2018invitation, labourie2021positivity, beyrer2023hp, beyrer2024positivity} and geometric structures \cite{guichard2012anosov, kapovich2017dynamics, danciger2018convex, DGKproj}. The property of being Anosov depends on the choice of a conjugacy class of parabolic subgroups $P$ of $G$; a representation is said to be \emph{$P$-Anosov} if it is Anosov with respect to $P$. If $ G = \PO(n,1)$, then $P$ is necessarily the unique (up to conjugation) proper parabolic subgroup of $G$, and $P$-Anosov representations are exactly the convex cocompact representations.

The images of $P$-Anosov representations are called \emph{$P$-Anosov subgroups of $G$}, and they lie in the wide family of \emph{$P$-transverse subgroups} (see e.g.\ \cite{canary2023patterson, kapovich2017anosov} and Section~\ref{sect_divergent_groups}). This family is characterized by weaker dynamical properties than Anosov representations and includes, for instance, relatively Anosov subgroups in the sense of \cite{zhu2022relatively, kapovich2023relativizing}. The question of whether such groups with strong dynamical properties provide examples of $(G, X)$-manifolds $M$ where $X$ is a $G$-homogeneous manifold, has been investigated by numerous authors, in particular when $X$ is a flag manifold of $G$ and $M$ is a quotient $\O / \Gamma$, where $\O$ is an open subset of $X$, and $\Gamma$ is a $P$-transverse subgroup of $G$ preserving $\O$:
\begin{enumerate}
    \item In \cite{frances2005lorentzian, guichard2012anosov}, for some semisimple Lie groups $G$ and parabolic subgroups $P, P'$ of $G$, examples of compact $(G, G/P)$-manifolds are built as quotients $\O/\Gamma$ of a well-chosen open subset $\O \subset G/P'$ preserved by a $P$-Anosov subgroup $\Gamma \leq G$. A general approach is developed in \cite{kapovich2017anosov}. These open sets are in general not proper.
    
    \item Danciger--Guéritaud--Kassel \cite{danciger2018convex, DGKproj} and Zimmer \cite{zimmer2021projective} introduce a notion of convex cocompactness in the projective space: if $\Gamma$ preserves a properly convex open subset $\O$ of $\mathbb{P}(\mathbb{R}^n)$ and acts cocompactly on its convex core $\mathcal{C} \subset \O$, and the ideal boundary of $\mathcal{C}$ contains no projective segment, then $\Gamma$ is said to be \emph{strongly convex cocompact in $\mathbb{P}(\mathbb{R}^n)$}. A discrete subgroup $\Gamma \leq \PGL(n, \mathbb{R})$ is strongly convex cocompact if and only if it is $P_1$-Anosov and preserves a proper domain in $\mathbb{P}(\mathbb{R}^n)$ \cite{DGKproj} (see also \cite{zimmer2021projective} for the Zariski-dense case).

    \item Crampon--Marquis and Cooper--Long--Tillman define a \emph{geometrically finite projective manifold} as the quotient $\O / \Gamma$ of a strictly convex domain $\O \subset \mathbb{P}(\mathbb{R}^n)$ by a discrete subgroup $\Gamma \leq \PGL(n, \mathbb{R})$, such that the convex core of $\O / \Gamma$ is the union of a compact set and a finite number of ends, called \emph{cusps} \cite{blayac2012finitude, cooper2015convex}. Fléchelles--Islam--Zhu prove that a discrete subgroup $\Gamma \leq \PGL(n, \mathbb{R})$ preserving a proper domain in $\mathbb{P}(\mathbb{R}^n)$ is relatively $P_1$-Anosov if and only if there exists a round (i.e.\ strictly convex and admitting a unique supporting hyperplane at any boundary point) properly convex $\Gamma$-invariant domain $\O \subset \mathbb{P}(\mathbb{R}^n)$ on which the action of $\Gamma$ is geometrically finite (see \cite{flechelles2024geometric}).
\end{enumerate}

In fact, by \cite{canary2023patterson}, every discrete $P$-transverse subgroup $\Gamma \leq G$ preserves a proper domain $\O$ in a certain real projective space, and thus gives rise to a projective manifold $\O / \Gamma$. Point (2) (resp.\ (3)) above implies that if $\Gamma$ is moreover $P$-Anosov (resp.\ relatively $P$-Anosov), this manifold can be asked to be convex cocompact (resp.\ geometrically finite). However, this manifold is \emph{a priori} modelled on the real projective space, not on $G/P$. It is therefore natural to seek to construct $(G, G/P)$-manifolds of the form $\O' / \Gamma$, where $\O'$ is a  domain of $G/P$. We may also ask if this domain~$\O'$ can be taken to be proper, as for Poins (2) and (3). Thus the following questions naturally arises:
\begin{questions}\label{question_P_transverse} 
Let $G$ be a real semisimple Lie group and~$P$ a self-opposite proper parabolic subgroup of~$G$, and $H \leq G$ a discrete $P$-transverse subgroup.  
\begin{enumerate}
    \item Under what conditions does $H$ preserve a proper domain in $G/P$?  
    \item If $H$ is discrete and preserves a proper domain $\O \subset G/P$, what additional geometric conditions on the action of $H$ on $\O$ are necessary to ensure that $H$ is $P$-Anosov?  
\end{enumerate}
\end{questions}

\subsection{Content of the paper}\label{sect_content_paper}
This paper is motivated by Questions~\ref{question_P_transverse}. We begin by focusing on Point~(1), and observe that the property of preserving a proper domain is in fact quite restrictive (see Theorem~\ref{prop_restrictions_maslov_index} below). A natural approach to addressing Point~(2) is to draw inspiration from the notion of convex cocompactness introduced in~\cite{DGKproj}. For this, an appropriate notion of convexity in flag manifolds is required. We define such a notion in the so-called \emph{causal flag manifolds}, which generalizes the concept of \emph{causal convexity} from Lorentzian geometry. While the condition of preserving a proper domain in a causal flag manifold is restrictive, we show that once a group preserves a proper domain~$\O$, the property of acting cocompactly on a closed convex subset of~$\O$ is, by contrast, quite weak (see Theorem~\ref{thm_equiv_anosov} and Remark~\ref{rmk_convex_envelop_projective} below).

\subsubsection{Causal flag manifolds} A substantial part of this paper focuses on transverse groups preserving proper domains in \emph{causal flag manifold}. These flag manifolds appear naturally in several contexts, such as Euclidean Jordan algebras and complex analysis (see e.g.\ \cite{faraut1994analysis}), or $\Theta$-positivity and higher Teichmüller Theory \cite{guichard2018positivity}; and their structure is well understood (see e.g.\ \cite{kaneyuki1988sylvester, kaneyuki2006causal, neeb2025open}). In order to state the results of this paper in the rest of the introduction, let us recall the definition and basic structure of these spaces.

Let $G$ be a Hermitian simple Lie group of tube type --- denoted by $\HTT$ --- and of real rank $r \geq 1$ (the complete list of the corresponding Lie algebras is given in Table~\ref{table_shilov_bnds}). Let $\FS$ denote the set of simple restricted roots of $G$, and let $\alpha_r \in \FS$ be the unique long root. Then the flag manifold $\Fl(\g, \{\alpha_r\})$, defined by $\{\alpha_r\}$ (see Section~\ref{sect_flag_mfds}), is the \emph{Shilov boundary} of the symmetric space $\mathbb{X}_G$ of $G$, denoted by~$\SB(G)$. It admits a \emph{causal structure}, i.e.\ there exists (up to taking an index-two subgroup of $G$) a $G$-equivariant smooth family $(c_x)_{x \in \SB(G)}$ of properly convex open cones in the tangent bundle $T(\SB(G))$. This property is specific to Shilov boundaries of Hermitian symmetric spaces of tube type (see \cite{neeb2025open}), thus we also call them the \emph{causal flag manifolds}.

\subsubsection{Topological Restrictions}

Section~\ref{sect_topo_restrictions} is devoted to the study of Question~\ref{question_P_transverse}.(1). In \cite[Prop.\ 1.2]{benoist2000automorphismes}, Benoist provides a necessary and sufficient condition for a strongly irreducible subgroup of~$\PGL(n, \mathbb{R})$ to preserve a properly convex open subset of~$\mathbb{P}(\mathbb{R}^n)$. The proof relies on some fundamental properties of convex sets in projective space that no longer hold in the more general setting of flag manifolds~$G/P$. An analogue of his necessary condition can nevertheless be recovered, as expressed in Proposition~\ref{prop_restrictions_maslov_index_2} below.

Let~$G$ be a real semisimple Lie group and~$P \leq G$ a self-opposite parabolic subgroup. There exists an involution~$\inver : ((G/P) \smallsetminus \hyp_{P^\opp}) \to ((G/P) \smallsetminus \hyp_{P^\opp})$ which acts as~$-\operatorname{id}$ on the affine chart~$((G/P) \smallsetminus \hyp_{P^\opp})$ based at~$P$, and permutes the connected components of~$(G/P) \smallsetminus (\hyp_{P} \cup \hyp_{P^\opp})$.

If~$x \in G/P$ is transverse to both~$P$ and~$P^\opp$, the \emph{type}~$\typ(P, x, P^\opp)$ of the triple~$(P, x, P^\opp)$ is defined as the orbit under~$P \cap P^\opp$ of the connected component of~$(G/P) \smallsetminus (\hyp_{P} \cup \hyp_{P^\opp})$ containing~$x$. This notion of type extends in a~$G$-invariant way to any triple~$(a, b, c)$ of pairwise transverse points in~$G/P$, and encodes their relative position. When~$G$ is a Hermitian simple Lie group of tube type and~$G/P = \SB(\g)$, this type is described by the classical \emph{Maslov index}~$\indx(a, b, c)$ of the triple~$(a, b, c)$ (see~\cite{lion1980weil} and Section~\ref{sect_maslov}). We prove the following:

\begin{thm}[see Proposition~\ref{prop_restrictions_maslov_index_2} and Corollary~\ref{cor_restrictions_maslov_index}]\label{prop_restrictions_maslov_index} 
Let~$G$ be a real semisimple Lie group and~$P \leq G$ a self-opposite parabolic subgroup. Let~$H \leq G$ be a subgroup preserving a proper domain~$\O \subset G/P$ such that the limit set~$\Lambda_P(H)$ contains at least three pairwise transverse points. Then there exists an~$\inver$-invariant connected component~$\mathcal{O} \subset (G/P) \smallsetminus (\hyp_P \cup \hyp_{P^\opp})$ such that for any triple~$(a, b, c) \in \Lambda_P(H)^3$ of pairwise transverse points, we have~$\typ(a, b, c) = [\mathcal{O}]$.

In the case where~$G$ is a Hermitian simple Lie group of tube type and real rank~$r \geq 2$, with~$G/P = \SB(\g)$, then~$r$ is even, and~$\indx(x, y, z) = 0$ for every triple of pairwise distinct points~$(x, y, z) \in \Lambda_P(H)^3$.
\end{thm}

\begin{rmk}
Theorem~\ref{prop_restrictions_maslov_index} is closely related to the notion of \emph{Property I} (``I" for ``fully involutive") introduced in~\cite{dey2024restrictions}. A flag manifold~$G/P$ is said to satisfy \emph{Property I} if no connected component of~$(G/P) \smallsetminus (\hyp_P \cup \hyp_{P^\opp})$ is invariant under~$\inver$; see Section~\ref{sect_stable_connected components} for further details. According to Theorem~\ref{prop_restrictions_maslov_index}, if there exists a subgroup~$H \leq G$ preserving a proper domain in~$G/P$, then~$G/P$ does \emph{not} satisfy Property I. Question~\ref{question_P_transverse}.(1) is therefore closely related to the following question, asked by Dey--Greenberg--Riestenberg and studied by Dey~\cite{dey2022borel}, Dey--Greenberg--Riestenberg~\cite{dey2024restrictions}, and Kineider--Troubat~\cite{kineider2024connected} (see Remark~\ref{rmk_prop_I_anosov}): \emph{which self-opposite flag manifolds~$G/P$ satisfy Property I?}
\end{rmk}

\subsubsection{Examples}\label{sect_intro_exemples}  
In Section~\ref{sect_examples}, we construct Zariski-dense $P$-Anosov subgroups of Hermitian Lie groups~$G$ that preserve a proper domain in~$\SB(\g) = G/P$:

\begin{thm}\label{prop_existence_CC_subgroups_1}
Let~$r = 2p$, with~$p \in \mathbb{N}_{>0}$. If~$G$ is a Hermitian simple Lie group of tube type and of real rank~$r$, and~$P \leq G$ is a parabolic subgroup such that~$G/P = \SB(\g)$, then there exist Zariski-dense $P$-Anosov surface groups in~$G$ that preserve a proper domain in~$\SB(\g)$.
\end{thm}
Here, the condition that~$r$ is even is necessary, as in the odd case no examples exist by Theorem~\ref{prop_restrictions_maslov_index}.

The strategy of the proof of Theorem~\ref{prop_existence_CC_subgroups_1} is, first, to construct examples preserving diamonds (where \emph{diamonds} are specific proper domains in~$\SB(\g)$, defined in next Section~\ref{sect_intro_convexity}). Then, we can slightly deform these examples into Zariski-dense ones, by \cite{kim2015flexibility}. After small deformation, they still preserve a proper domain, as ensured by the following openness property:

\begin{prop}[see Corollary~\ref{cor_preserving_proper_domain_open}]\label{prop_preserving_proper_domain_open}
    Let~$G$ be a noncompact real semisimple Lie group and~$P$ be a parabolic subgroup of~$G$. Let~$\Gamma$ be a Gromov-hyperbolic group. Then the property of preserving a proper domain in~$G/P$ is open in the set of~$P$-Anosov representations of~$\Gamma$ into~$G$.
\end{prop}

The examples constructed in the proof of Theorem~\ref{prop_existence_CC_subgroups_1} are themselves subject to dynamic and topological constraints; see Proposition~\ref{prop_restrictions_L} and Example~\ref{ex_cas_SP_rigidity}. If the Lie algebra of~$G$ is~$\spp(2r, \mathbb{R})$, with~$r$ a multiple of~$4$,~$\uu(r,r)$ or~$\soo^*(4r)$ with~$r$ even, then there exist Zariski-dense~$P$-Anosov subgroups of~$G$ preserving a proper domain in~$\SB(\g)$ that are neither virtually free nor surface groups
(see Example~\ref{ex_cas_SP_rigidity}.(2)).
In the case where~$G = \SO(n, 2)$, other example appear, as seen in Proposition~\ref{cor_cohomological_dimension}.

\subsubsection{Causal convexity in causal flag manifolds}\label{sect_intro_convexity} As mentioned above, in this paper we define a natural notion of convexity, specific to causal flag manifolds. If a point~$x$ lies in an affine chart~$\Affstd$ of~$\SB(\g)$ (using the notation from Section~\ref{sect_transversality}), causality allows us to define the \emph{future} and the \emph{past} of~$x$ (see Section~\ref{sect_causalité}). For any point~$y$ in the future of~$x$, the intersection of the past of~$y$ with the future of~$x$ defines a \emph{diamond}, denoted~$\Diams_\Affstd(x,y)$. From another perspective, there exist exactly two diamonds whose endpoints are~$x$ and~$y$: these are the only two proper connected components of~$\SB(\g) \smallsetminus (\hyp_x \cup \hyp_y)$ (see Definition~\ref{def_diamant}). Diamonds induce a notion of \emph{causal convexity} analogous to that already present in spacetime geometry~\cite{Sanchez2008}. A connected subset~$X \subset \SB(\g)$ is said to be \emph{causally convex} if it is contained in an affine chart~$\Affstd$, and for all~$x, y \in X$ such that~$y$ lies in the future of~$x$ in~$\Affstd$, the closed diamond~$\Diams^c_\Affstd(x,y)$ is contained in~$X$ (see Definition~\ref{def_causal_convexity}). In Proposition~\ref{lem_convex_hul_indép_carte_aff}, we show that this definition is independent of the choice of affine chart~$\Affstd$ containing~$X$. We then compare it to the \emph{dual convexity} mentioned in Section~\ref{sect_intro_falg_mfds} above:

\begin{prop}\label{prop_dual_implique_causall} 
Let~$G$ be a Hermitian simple Lie group of tube type, and let~$\O \subset \SB(\g)$ be a dually convex domain strictly contained in~$\SB(\g)$. Then~$\O$ is contained in at least one affine chart and is causally convex. 
\end{prop}

Our proof is based on an analysis of the open orbits of the automorphism group of the diamond (see Lemma~\ref{lem_inclusion_Ois}), initiated in work of Kaneyuki \cite{kaneyuki1988sylvester}.
In the case where~$G = \SO(n, 2)$ for some~$n \geq 2$, Proposition~\ref{prop_dual_implique_causall} admits a more direct proof, given in~\cite[Prop.\ 6.2]{chalumeau2024rigidity}.

\subsubsection{Convexity and transverse groups}  
In Section~\ref{sect_convex_transverse}, originally motivated by Question~\ref{question_P_transverse}.(2), we study the geometry of manifolds of the form~$\O / \Gamma$, where~$\Gamma$ is a discrete $P$-transverse subgroup of a Hermitian simple Lie group~$G$ of tube type, preserving a proper domain~$\O \subset G/P = \SB(\g)$ (see Theorem~\ref{thm_equiv_anosov}). Given a proper domain~$\O \subset \SB(\g)$ preserved by a discrete subgroup~$\Gamma \leq G$, the \emph{full orbital limit set}~$\limorb$ of~$(\O, \Gamma)$ is the set of all accumulation points of all $\Gamma$-orbits of points in~$\O$ (see~\cite{DGKproj}). A \emph{convex core} of~$(\O, \Gamma)$ is a closed nonempty connected causally convex and $\Gamma$-invariant subset~$\mathcal{C}$ of~$\O$ such that the ideal boundary~$\partial_i \mathcal{C} := \overline{\mathcal{C}} \smallsetminus \mathcal{C}$ contains~$\limorb$. If~$\partial_i \mathcal{C} = \partial \O$, then~$\O$ is a diamond; this rigidity is proved in \cite{galiay2024rigidity}. Conversely, requiring the action of~$\Gamma$ to admit a convex core with transverse ideal boundary is a very flexible condition:

\begin{thm}\label{thm_equiv_anosov}   
Let~$G$ be a~$\HTT$ Lie group and~$\Gamma \leq G$ a discrete subgroup. Let~$P \leq G$ be a parabolic subgroup such that~$G/P = \SB(\g)$. The following are equivalent:
\begin{enumerate}
    \item The group~$\Gamma$ is finitely generated, $P$-transverse, preserves a proper domain~$\O \subset \SB(\g)$, and~$\Lambda_P(\Gamma)$ contains at least three points.
    \item There exists a proper, \emph{causally convex}, $\Gamma$-invariant domain~$\O \subset \SB(\g)$ such that~$\Gamma$ acts cocompactly on a convex core~$\mathcal{C}$ of~$(\O, \Gamma)$ whose ideal boundary is transverse and contains at least three points.
    \item There exists a proper, \emph{dually convex}, $\Gamma$-invariant domain~$\O' \subset \SB(\g)$ such that~$\Gamma$ acts cocompactly on a convex core~$\mathcal{C}'$ of~$(\O', \Gamma)$ whose ideal boundary is transverse and contains at least three points.
\end{enumerate}

If these conditions hold, then~$\partial_i \mathcal{C} = \Lambda_P(\Gamma) = \limorb = \Lambda_{\O'}^{\operatorname{orb}}(\Gamma) = \partial_i \mathcal{C}'$.
\end{thm}

Theorem~\ref{thm_equiv_anosov} shows that a seemingly natural definition of convex cocompactness in~$\SB(\g)$ — namely, Point~(2) — does not distinguish $P$-Anosov subgroups (with~$G/P = \SB(\g)$) from other finitely generated $P$-transverse discrete subgroups of~$G$ that preserve a proper domain. This phenomenon stems from the intrinsically “timelike” nature of the convexity considered (causal convexity) and the “spatial” (i.e.\ of Maslov index 0) nature of the dynamical behavior of~$\Gamma$, as already noted in Theorem~\ref{prop_restrictions_maslov_index}; see Remark~\ref{rmk_intuition_main_thm} and Example~\ref{ex_illustration_einstein}.

\subsection{Outline of the paper} In Section~\ref{sect_prelminaries}, we recall some background material and establish a few elementary results concerning the notions used throughout the paper; in particular, we review the definition and properties of causal flag manifolds. In Section~\ref{sect_proj_geom_causal}, we recall the definition of \emph{diamonds} (Definition~\ref{def_diamant}), introduce the notion of causal convexity (Definitions~\ref{def_causal_convexity} and~\ref{def_causally_convex_indep}), and prove Proposition~\ref{prop_dual_implique_causall}. In Section~\ref{sect_topo_restrictions}, we establish Theorem~\ref{prop_restrictions_maslov_index}, through Proposition~\ref{prop_restrictions_maslov_index_2} and Corollary~\ref{cor_restrictions_maslov_index}. In Section~\ref{sect_groups_acting_cocompactly}, we prove the implication ``$(2) \Rightarrow \Gamma$ finitely generated'' from Theorem~\ref{thm_equiv_anosov}. We also establish preliminary results on groups preserving proper domains~$\O$ and acting cocompactly on a convex core with transverse boundary in~$\O$ (Lemmas~\ref{lem_p_conic_shilov} and~\ref{lem_strongly_convex_cocompact_shilov}), in preparation for the proof of the remaining of direction ``$(2) \Rightarrow (1)$'' and the equivalence ``$(2) \Leftrightarrow (3)$'' in Theorem~\ref{thm_equiv_anosov}. In Section~\ref{sect_convex_transverse}, we complete the proof of Theorem~\ref{thm_equiv_anosov}. In Section~\ref{sect_examples}, given a Lie group~$\HTT$~$G$ and a parabolic subgroup~$P \leq G$ such that~$G/P = \SB(\g)$, we construct examples of Zariski-dense $P$-Anosov subgroups of~$G$ preserving proper domains in~$\SB(\g)$, thereby proving Theorem~\ref{prop_existence_CC_subgroups_1}. We also prove Proposition~\ref{prop_preserving_proper_domain_open} (through Corollary~\ref{cor_preserving_proper_domain_open}). In Section~\ref{sect_anosov_einstein}, we investigate further examples in the case where $G = \SO(n,2)$ (corresponding to $\SB(\g) = \Ein^{n-1,1}$). In the Appendix (Section~\ref{sect_complement}), we discuss a generalization of results obtained in Section~\ref{sect_groups_acting_cocompactly}; the results of the section are independent from the rest of the paper.

\subsection*{Aknowledgements} I would like to thank my PhD advisor, Fanny Kassel, for her valuable corrections and feedback on this paper. I also thank her for encouraging me to explore the topic of convex cocompactness in flag manifolds, and for our insightful discussions on this subject. I am grateful to Yosuke Morita for his corrections, his help in the development of Lemma~\ref{lem_inclusion_Ois}, and more broadly, for our inspiring discussions on~$\HTT$ Lie groups. I would also like to thank Olivier Guichard and Andrew Zimmer, whose comments on my PhD thesis significantly contributed to improving the quality of this paper. I also thank Clarence Kineider, Roméo Troubat, and Max Riestenberg for sharing their insights on Property~I. Finally, I am very grateful to Rym Smaï for her helpful feedback on Section~\ref{sect_anosov_einstein}.

\section{Preliminaries}\label{sect_prelminaries}

In this section, we provide some basic reminders on objects that will be used throughout this paper and establish some elementary lemmas.

\subsection{Real projective space} Given a finite-dimensional real vector space~$V$, we will denote by~$[v]$ the projection in~$\mathbb{P}(V)$ of a vector~$v \in V \smallsetminus \{ 0  \}$. We denote by~$V^*$ the space of all linear forms on~$V$. The space~$\mathbb{P}(V^*)$ can be identified with the space of projective hyperplanes of~$\mathbb{P}(V)$, via the map 
\begin{equation}\label{indent_vector_space}
 [f] \mapsto \mathbb{P}(\Ker(f)). 
\end{equation}

\subsection{Preliminaries in Lie theory}\label{sect_prelim_lie_theory}
We fix a noncompact real semisimple Lie group~$G$, and denote by~$\g$ its Lie algebra.

\subsubsection{$\mathfrak{sl}_2$-triples} \label{sect_sl2-triples} \mbox{ } A triple~$\tr = (e,h,f)$ of elements of~$\g$ satisfying the relations $[h,e] = 2e$, $[h,f] = -2f$ and $[e,f] = h$ is called an \emph{$\mathfrak{sl}_2$-triple}. There is a Lie algebra embedding~$\plongsl_\tr: \mathfrak{sl}_2(\mathbb{R}) \hookrightarrow \g$ such that~$\plongsl_\tr(\E) = e$,~$\plongsl_\tr (\He) = h$ and~$\plongsl_\tr(\F) = f$, where
\begin{equation*}
    \E = \begin{pmatrix}
        0 & 1 \\ 0 & 0
    \end{pmatrix}; \quad \He = \begin{pmatrix}
        1 & 0 \\ 0 & -1
    \end{pmatrix}; \quad \F = \begin{pmatrix}
        0 & 0 \\ -1 & 0
    \end{pmatrix}.
\end{equation*}

\subsubsection{Cartan decomposition}\label{sect_cartan_decomp} Let~$B$ be the Killing form on~$\g$. Let~$K \leq G$ be a maximal compact subgroup and~$\mathfrak{h}$ be the~$B$-orthogonal of the Lie algebra~$\mathfrak{k}$ of~$K$ in~$\g$. Then one has~$\g = \mathfrak{k} \oplus \mathfrak{h}$. The \emph{Cartan involution }of~$\g$ (with respect to~$K$) is then the Lie algebra automorphism~$ \Cartinv: \g \rightarrow \g$ defined by~$ (\Cartinv)|_{\mathfrak{k}} = \operatorname{id}_{\mathfrak{k}}$ and~$ (\Cartinv)|_{\mathfrak{h}} = -\operatorname{id}_{\mathfrak{h}}$. It induces a Lie group automorphism of~$G$, still denoted by~$\Cartinv$ and called the \emph{Cartan involution of }$G$.

\subsubsection{Restricted root system.} \label{sect_real_lie_alg} Let~$\aaa \subset \mathfrak{h}$ be a maximal abelian subspace, and~$\g_0$ the centralizer of~$\aaa$ in~$\g$. We denote by~$\aaa^*$ the space of all linear forms on~$\aaa$. For~$\alpha \in \aaa^*$, we define
\begin{equation*}
    \g_{\alpha} := \{ X \in \g \mid [H, X] = \alpha(H)X \quad \forall H \in \aaa\}.
\end{equation*}
One has~$[\g_{\alpha}, \g_{\beta}] \subset \g_{\alpha + \beta}$ for any~$\alpha, \beta \in \aaa^*$. If~$\alpha \in \aaa^* \smallsetminus \{0\}$ satisfies~$\g_{\alpha} \ne \{0 \}$, then we say that~$\alpha$ is \emph{a restricted root} of~$(\g, \aaa)$. We denote by~$\Sigma = \Sigma(\g, \aaa)$ the set of all restricted roots of~$(\g, \aaa)$. One has~$\g = \g_0 \oplus \bigoplus_{\alpha \in \Sigma} \g_{\alpha}$. We fix a \emph{fundamental system}~$\FS = \{ \alpha_1, \dots , \alpha_N\} \subset \Sigma$, i.e.\ a family of restricted roots such that any root of~$\g$ can be uniquely written as~$\alpha = \sum_{i=1}^N n_i \alpha_i$, where the~$n_i$ all have same sign for~$1 \leq i \leq N$. The elements of~$\FS$ are called \emph{simple restricted roots}. From now on, whenever we fix a noncompact real semisimple Lie group~$G$ (resp.\ a real semisimple Lie algebra of noncompact type~$\g$), it will always implicitly be endowed with a fixed set~$\FS$ of simple restricted roots.

The choice of a fundamental system determines a set of \emph{positive restricted roots}~$\Sigma^+$, i.e.\ those roots~$\alpha$ where the~$n_i$ are all nonnegative.

For any~$\alpha \in \Sigma$ and~$X \in \g_{\alpha}\smallsetminus \{0\}$, there exists a unique scalar multiple~$X'$ of~$X$ such that~$(X',  [\Cartinv(X'), X'],  \Cartinv(-X'))$ is an~$\mathfrak{sl}_2$-triple. The element~$[ \Cartinv(X'), X']$ does not depend on the choice of~$X \in \g_{\alpha}$, and is denoted by~$h_{\alpha}$. The family~$(h_{\alpha})_{\alpha \in \FS}$, whose elements are called the \emph{coroots of~$\g$}, forms a basis of~$\aaa$, whose dual basis in~$\aaa^*$ is denoted by~$(\omega_{\alpha})_{\alpha \in \FS}$.

The \emph{closed positive Weyl chamber associated with~$\FS$} is 
\begin{equation*}
    \overline{\aaa}^+ = \{X \in \aaa \mid \alpha(X) \geq 0 \quad \forall \alpha \in \FS\}.
\end{equation*}
The \emph{Cartan decomposition} states that for all~$g \in G$, there exist~$k,\ell \in K$, and a unique~$\mu(g) \in \overline{\aaa}^+$ such that~$g = k \exp(\mu(g))\ell$. This defines the \emph{Cartan projection}~$\mu: G \rightarrow \overline{\aaa}^+$. The \emph{Lyapunov projection} is then the map~$\lambda: G \rightarrow \overline{\aaa}^+$ satisfying~$\lambda(g) = \lim_{k \rightarrow + \infty} \mu(g^k)/k$ for all~$g \in G$.

 \begin{rmk}\label{rmk_lyapunov_projection}
Note that the formula we gave is not the actual definition of the Lyapunov projection, but a characterization (see e.g.\ \cite{benoist1997proprietes}). We will only use this formula and the fact that, in the case where~$G = \SL(n, \mathbb{R})$, the Lyapunov projection of an element~$g \in G$ is a diagonal matrix whose entries are the logarithms of the moduli of the complex eigenvalues~$\lambda_1(g), \dots, \lambda_n(g)$ of~$g$, ordered so that that~$|\lambda_1(g)| \geq \dots \geq |\lambda_n(g)|$.
 \end{rmk}

\subsubsection{The restricted Weyl group}\label{sect_restricted_weyl_grouup}
The \emph{restricted Weyl group}~$W$ of~$G$ is the quotient~$N_K(\aaa)/Z_K(\aaa)$ of the normalizer of~$\aaa$ in~$K$ (for the adjoint action) by the centralizer of~$\aaa$ in~$K$. ~For its natural embedding in~$\GL(\aaa)$, it is a finite group generated by the~$B$-orthogonal reflexions in~$\aaa$ with respect to the kernels of the simple restricted roots. By duality with respect to~$B$ (which induces a scalar product on~$\aaa$), the action of~$W$ on~$\aaa$ induces an action on~$\aaa^*$ preserving~$\Sigma$. There exists a unique~$w_0 \in W$, called the \emph{longest element}, such that~$w_0 \cdot \Sigma^+ = -\Sigma^+$. The map~$\oppinv: \aaa^* \rightarrow \aaa^*$ defined as~$ \oppinv = - w_0$ is called the \emph{opposition involution}, and satisfies~$\oppinv(\FS) = \FS$.

\subsubsection{Parabolic subgroups} \label{sect_parab_subgroups} Let~$\Theta \subset \FS$ be a subset of the simple restricted roots. \emph{The standard parabolic subgroup}~$P_{\Theta}^+$ (resp.\ the \emph{standard opposite parabolic subgroup}~$P_{\Theta}^{\opp}$) is defined as the normalizer in~$G$ of the Lie algebra
\begin{equation}\label{eq_lie_u}
    \uu_{\Theta}^+ := \bigoplus_{\alpha \in \Sigma_{\Theta}^+} \g_{\alpha} \quad \Bigg{(}\text{resp.\ }  \uu_{\Theta}^- := \bigoplus_{\alpha \in \Sigma_{\Theta}^+} \g_{-\alpha} \Bigg{)},
\end{equation}
where~$\Sigma_{\Theta}^+ := \Sigma^+ \smallsetminus \operatorname{Span}(\FS \smallsetminus \Theta)$. By “standard”, we mean with respect to the above choices. More generally, a \emph{parabolic subgroup of type~$\Theta$} of~$G$ is a conjugate of~$P_{\Theta}^+$ in~$G$.

The Lie algebra of~$P_\Theta^+$ (resp.\ $P_\Theta^\opp$) is denoted by~$\LP_\Theta^+$ (resp.\ $\LP_\Theta^\opp$). For any representative~$k_0 \in N_K(\aaa)$ of~$w_0$, one has~$k_0 P_{\Theta}^\opp k_0 = P_{\oppinv(\Theta)}^+$. 

The \emph{Levi subgroup} associated with~$\Theta$ is the reductive Lie group defined as the intersection~$L_{\Theta} := P_{\Theta}^+ \cap P_{\Theta}^{\opp}$. Its Lie algebra is
\begin{equation}\label{eq_decomposition_l}
    \mathfrak{l} = \g_0 \oplus \bigoplus_{\alpha \in \Sigma^+ \cap \operatorname{Span}(\FS \smallsetminus \Theta)} (\g_\alpha \oplus \g_{-\alpha}).
\end{equation}
In particular, it contains~$\aaa$.

The unipotent radical of~$P_{\Theta}^+$ (resp.\ $P_{\Theta}^{\opp}$) is~$U_{\Theta}^+ := \exp (\uu_{\Theta}^+)$ (resp.\ $U_{\Theta}^- := \exp (\uu_{\Theta}^-)$). One then has~$P_{\Theta}^+ = U_{\Theta}^+ \rtimes L_{\Theta}$ (resp.\ $P_{\Theta}^{\opp} = U_{\Theta}^- \rtimes L_{\Theta}$).

\subsubsection{Flag manifolds}\label{sect_flag_mfds} 
The \emph{flag manifold} associated with~$\Theta$ is the quotient space~$\Fl(\g, \Theta) :=  \g /\mathfrak{p}_{\Theta}^+$. The flag manifold \emph{opposite} to~$\Fl(\g, \Theta)$ is~$\Fl(\g, \Theta)^{\opp} := \Fl(\g, \oppinv(\Theta))$. Since~$G$ acts transitively on~$\Fl(\g, \Theta)$ (resp.\ $\Fl(\g, \Theta)^\opp$) via the adjoint action, we have the~$\Ad(G)$-equivariant identification 
\begin{equation}\label{ident_flag_mfds}
    G/P_\Theta^+ \simeq \Fl(\g, \Theta) \quad \text{(resp.\ }G/P_\Theta^\opp \simeq \Fl(\g, \Theta)^\opp\text{)}.
\end{equation}
We will simply denote by~$g \cdot x$ the action of an element~$g \in G$ on~$x \in \Fl(\g, \Theta)$ (instead of~$\Ad(g) \cdot x$).

\subsubsection{Automorphisms of flag manifolds}\label{sect_automorphism_group} The group of all Lie algebra automorphisms of~$\g$ is called the \emph{automorphism group} of~$\g$ and denoted by~$\Aut(\g)$. It is a Lie group with Lie algebra~$\g$. When~$G$ is semisimple, the map~$\Ad: G \rightarrow \Aut(\g)$ has finite kernel.

In general, the group~$\Aut(\g)$ does not act on~$\Fl(\g, \Theta)$. However, it admits a finite-index subgroup that does: indeed, any~$g \in \Aut(\g)$ induces an automorphism~$\psi_g$ of the fundamental system~$\Delta$. This defines a group homomorphism~$\Aut(\g) \rightarrow \Aut(\Delta)$. For~$\Theta \subset \FS$, we denote by~$\Aut_\Theta(\g)$ the subgroup of~$\Aut(\g)$ of all Lie algebra automorphisms~$g$ such that~$\psi_g$ fixes~$\Theta$. It acts naturally on~$\Fl(\g, \Theta)$, and contains~$\Ker \psi$, which itself contains~$\Ad(G)$. We will denote this action by~$g \cdot x$ for~$x \in \Aut_\Theta(\g)$ and~$x \in \Fl(\g, \Theta)$. 

Note that~$\ker(\Ad)$ acts trivially on~$\Fl(\g, \Theta)$. Since we will always consider the action of~$G$ on~$\Fl(\g, \Theta)$, we may thus identify it with its image~$\Ad(G)$, and we will always be able to assume that~$G \subset \Aut_\Theta(\g)$.

\subsubsection{Transversality}\label{sect_transversality}
The action of~$G$ on~$\Fl(\g, \Theta) \times \Fl(\g, \Theta)^{\opp}$ by left translations has exactly one open orbit, which is the orbit of~$(\LP_{\Theta}^+, \LP^{\opp}_{\Theta})$ and is dense. Two elements~$x \in \Fl(\g, \Theta)$ and~$y \in \Fl(\g, \Theta)^{\opp}$ are said to be \emph{transverse} if~$(x,y) \in G \cdot (\LP_\Theta^+, \LP_\Theta^\opp)$.

Given a point~$y \in \Fl(\g, \Theta)^{\opp}$ (resp.\ $x \in \Fl(\g, \Theta)$), we let~$\hyp_{y}$ (resp.\ $\hyp_{x}$) be the set of all elements of~$\Fl(\g, \Theta)$ (resp.\ $\Fl(\g, \Theta)^\opp$) that are not transverse to~$y$ (resp.\ to~$x$):
\begin{equation}\label{eq_def_hyp_x}
    \begin{split}
        \hyp_y &:= \{z \in \Fl(\g, \Theta) \mid (z,y) \notin G \cdot (\LP_\Theta^+, \LP_\Theta^\opp)\}; \\
        \hyp_x &:= \{z' \in \Fl(\g, \Theta) \mid (x,z') \notin G \cdot (\LP_\Theta^+, \LP_\Theta^\opp)\}.
    \end{split}
\end{equation}

The set~$\hyp_y$ (resp.\ $\hyp_x$) is an algebraic subvariety of~$\Fl(\g, \Theta)$ (resp.\ of~$\Fl(\g, \Theta)^\opp$), called a \emph{maximal proper Schubert subvariety}. The space
$$\Affstd_y := \Fl(\g, \Theta) \smallsetminus \hyp_y$$
is called an \emph{affine chart} (or more classically a \emph{big Schubert cell}) and is an open dense subset of~$\Fl(\g, \Theta)$. The affine chart
\begin{equation}\label{eq_standard_chart}
    \Affstd_\std := \Fl(\g, \Theta) \smallsetminus \hyp_{\LP_{\Theta}^{\opp}}
\end{equation}
is called the \emph{standard affine chart}. We have the bijection
\begin{equation}\label{eq_A_egal_exp}
   \varphi_{\mathsf{std}}: \begin{cases}
        \uu_{\Theta}^- &\overset{\sim}{\longrightarrow} \Affstd_\std \\
        X &\longmapsto \exp(X) \cdot \LP_{\Theta}^+.
    \end{cases}
\end{equation}

\subsubsection{Self-opposite flag manifolds} \label{sect_self_opposite} If~$\oppinv(\Theta) = \Theta$, then we say that~$\Fl(\g, \Theta)$ is \emph{self-opposite}. For any representative~$k_0 \in N_K(\aaa)$ of the longest element~$w_0 \in W$, one has~$\LP_{\Theta}^+ = k_0 \cdot \LP_{\Theta}^{\opp}$, so~$\Fl(\g, \Theta)= \Fl(\g, \Theta)^{\opp}$. A subset~$F \subset \Fl(\g, \Theta)$ will be said to be \emph{transverse} if any two distinct points~$x,y \in F$ are transverse.

\subsection{Divergent groups}\label{sect_divergent_groups} Divergent groups are groups with strong dynamical properties for their action on flag manifolds. 

A sequence~$(g_k) \in G^{\mathbb{N}}$ is \emph{$\Theta$-divergent} if~$\alpha(\mu(g_k)) \rightarrow + \infty$ for every~$\alpha \in \Theta$. It is \emph{$\Theta$-contracting} if there exists~$(x, \xi) \in \Fl(\g, \Theta)\times \Fl(\g, \Theta)^\opp$ such that~$g_k \cdot y \rightarrow x$ uniformly on compact subsets of~$\Fl(\g, \Theta) \smallsetminus \hyp_{\xi}$; the pair~$(x, \xi)$ is then uniquely defined by~$(g_k)$, and we say that~$(g_k)$ is~$\Theta$-contracting with respect to~$(x, \xi)$, and that~$x$ is \emph{the~$\Theta$-limit} of~$(g_k)$. Let us recall the following fact, which is an immediate consequence of the Cartan decomposition of~$G$ (see e.g.\ \cite{kapovich2017anosov}):

\begin{fact}\label{fact_KAK_divergent} Let~$(g_k) \in G^{\mathbb{N}}$.
\begin{enumerate}
    \item Assume that that there exist an open subset~$\mathcal{U} \subset \Fl(\g, \Theta)$ and a point~$x \in \Fl(\g, \Theta)$ such that~$g_k \cdot \mathcal{U} \rightarrow \{ x \}$ for the Hausdorff topology. Then~$(g_k)$ is~$\Theta$-contracting with~$\Theta$-limit~$x$.
    \item A sequence~$(g_k)$ is~$\Theta$-divergent if and only if every subsequence of~$(g_k)$ admits a~$\Theta$-contracting subsequence.
\end{enumerate}
\end{fact}

Thus, given a~$\Theta$-divergent sequence~$(g_k) \in G^\mathbb{N}$, one can define the \emph{$\Theta$-limit set of~$(g_k)$} as the set, denoted by~$\Lambda_\Theta(g_k)$, of all~$\Theta$-limits of~$\Theta$-contracting subsequences of~$(g_k)$. Given a subgroup~$H \leq G$, we denote by~$H^\mathbb{N}_\Theta$ the set of~$\Theta$-divergent sequences of elements of~$H$, and following \cite{gueritaud2017anosov, kapovich2017dynamics}, we define the~\emph{$\Theta$-limit set of~$H$} as
\begin{equation*}
    \Lambda_\Theta(H) = \bigcup_{(g_k) \in H^\mathbb{N}_\Theta} \Lambda_\Theta(g_k).
\end{equation*}

 We say that~$H$ is \emph{$\Theta$-divergent} if every sequence of distinct elements of~$H$ is~$\Theta$-divergent. If~$\Theta \subset \Delta$ is self-opposite, the subgroup~$H$ is said to be \emph{$\Theta$-transverse} if it is~$\Theta$-divergent and its~$\Theta$-limit set~$\Lambda_\Theta(H)$ is transverse, in the sense of Section~\ref{sect_self_opposite}.

\subsection{Anosov representations}\label{sect_anosov}

Anosov representations are generalizations to any reductive Lie groups~$G$ of rank-one convex-cocompact representations of Gromov-hyperbolic groups. The original definition is essentially dynamical \cite{labourie2006anosov}, but in this paper, we use as a definition the following characterization of~$\Theta$-Anosov representations proven in \cite{gueritaud2017anosov}:

\begin{definition}\label{def_anosov_rep}
    Let~$\Gamma$ be a discrete word-hyperbolic group and let~$\rho: \Gamma \rightarrow G$ be a representation.  We say that~$\rho$ is \emph{$\Theta$-Anosov} if the following properties are satisfied: 
\begin{enumerate}
    \item For every sequence of elements~$(g_k) \in \Gamma^{\mathbb{N}}$ diverging in~$\Gamma$, the sequence~$(\rho(g_k))$ is~$\Theta$-divergent.
    \item There exist continuous,~$\rho$-equivariant maps~$\xi_\rho: \Gambor \rightarrow \Fl(\g, \Theta)$ and~$ \xi_\rho^\opp: \Gambor \rightarrow \Fl(\g, \Theta)^\opp$ which are:
    \begin{enumerate}
        \item \emph{transverse}, i.e.\ for all~$x,y \in \Gambor$, we have:
        \begin{equation*}
            x \ne y \Longrightarrow \xi_\rho(x) \text{ and } \xi_\rho^\opp (y) \text{ are transverse};
        \end{equation*}
        \item \emph{dynamics-preserving}, that is, for any infinite-order element~$g \in \Gamma$ with attracting fixed point~$a \in \Gambor$, the sequence~$(\rho(g)^k)_{k \in \mathbb{N}}$ is~$\Theta$-contracting (resp.\ $\oppinv(\Theta)$-contracting) with~$\Theta$-limit~$\xi_\rho(a)$ (resp.\ with~$\oppinv(\Theta)$-limit~$\xi_\rho^\opp(a)$).
    \end{enumerate}
\end{enumerate}
\end{definition}
The maps~$\xi_\rho$ and~$\xi_\rho^\opp$ are uniquely defined by~$\rho$ and are respectively called the \emph{limit map} and the \emph{dual limit map} of~$\rho$. If~$\Theta$ is self-opposite, then~$\xi_{\rho} = \xi_{\rho}^\opp$. In this paper we will denote by~$\operatorname{Hom}_{\Theta\text{-}\mathsf{An}}(\Gamma, G)$ the space of all~$\Theta$-Anosov representations of~$\Gamma$ into~$G$. 

Anosov representations are discrete with finite kernel, and they are \emph{structurally stable}, i.e.\ the set \( \operatorname{Hom}_{\Theta\text{-}\mathsf{An}}(\Gamma, G) \) is open in the set \( \operatorname{Hom}(\Gamma, G) \) of all representations of \( \Gamma \) into \( G \) \cite{labourie2006anosov}.

We will say that a subgroup~$\Gamma$ of~$G$ is~\emph{$\Theta$-Anosov} if it is the image of a~$\Theta$-Anosov representation. Any~$\Theta$-Anosov subgroup is in particular~$\Theta$-divergent, and even~$\Theta$-transverse if~$\Theta$ is self-opposite.

\subsection{Irreducible representations of semisimple Lie groups. } \label{sect_irred_semi} A pair~$(V, \rho)$ is a finite-dimensional real linear (resp.\ projective) representation of~$G$, if~$\rho$ is a group morphism~$G \rightarrow \GL(V)$ (resp.\ $G \rightarrow \PGL(V)$). We will often use the following result, which allows some arguments in general flag manifolds to reduce to arguments in real projective space:

\begin{fact}[{\cite[Prop.\ 3.3]{gueritaud2017anosov}}]\label{prop_ggkw} Let~$G$ be a noncompact real semisimple Lie grop and~$\Theta$ be a subset of the simple restricted roots of~$G$. There exists an irreducible linear (resp.\ projective) representation~$(V, \rho)$ of~$G$, and a decomposition~$V = \ell \oplus H$, where~$\ell \subset V$ is a vector line and~$H \subset V$ is a hyperplane, such that:
    
\begin{enumerate}
    \item The stabilizer of~$\ell$ (resp.\ $H$) in~$G$ is~$P_{\Theta}$ (resp.\ $P_{\Theta}^{\opp}$).
    \item Considering the two~$\rho$-equivariant embeddings
    \begin{align*}
        &\iota_\rho: \Fl(\g, \Theta) \longrightarrow \mathbb{P}(V); g \cdot \LP_{\Theta}^+ \longmapsto \rho(g) \cdot \ell, \\
        &\iota^{\opp}_\rho: \Fl(\g, \Theta)^{\opp} \longrightarrow \mathbb{P}(V^*); g \cdot \LP_{\Theta}^\opp \longmapsto \rho(g) \cdot H
    \end{align*}
    induced by~$\rho$, two elements~$(x,\xi) \in \Fl(\g, \Theta) \times \Fl(\g, \Theta)^{\opp}$ are transverse if and only if their images~$\iota_\rho(x)$ and~$\iota_{\rho}^{\opp}(\xi)$ are.
    \item For any sequence~$(g_k) \in G^\mathbb{N}$, the sequence~$(g_k)$ is~$\Theta$-divergent (resp.\ $\Theta$-contracting) if and only if the sequence~$(\rho(g_k)) \in \PGL(V)^\mathbb{N}$ is~$\{\alpha_1\}$-divergent (resp.\ $\{\alpha_1\}$-contracting). If~$(g_k)$ is~$\Theta$-contracting with~$\Theta$-limit~$x \in \Fl(\g, \Theta)$, then the~$\{\alpha_1\}$-limit of~$(\rho(g_k))$ is~$\iota_\rho(x)$.
    \item Let~$\Gamma$ be a word hyperbolic group and let~$\rho' : \Gamma \rightarrow G$ be a representation. The representation~$\rho'$ is~$\Theta$-Anosov if and only if the representation~$\rho \circ \rho': \Gamma \rightarrow \PGL(V)$ is~$\{\alpha_1\}$-Anosov.
\end{enumerate}
\end{fact}

Note that for~$\nu \in V \smallsetminus \{0 \}$ and~$f \in V^* \smallsetminus \{ 0 \}$, the transversality of~$[\nu] \in \mathbb{P}(V)$ and~$[f] \in \mathbb{P}(V^*)$ is equivalent to~$f(\nu) \ne 0$. With the identification~\eqref{indent_vector_space} and the notations of Fact~\ref{prop_ggkw}, two elements~$(x,\xi) \in \Fl(\g, \Theta) \times \Fl(\g, \Theta)^{\opp}$ are transverse if and only if one has~$\iota_\rho(x) \notin \iota^{\opp}_\rho(\xi)$.

In the notation of Fact~\ref{prop_ggkw}, we have the following elementary property:
\begin{fact}\label{fact_generat}(see e.g.\ \cite{zimmer2018proper}) Let~$\O \subset \Fl(\g, \Theta)$ be a subset with nonempty interior. There exist~$\xi_1, \cdots , \xi_n \in \O$ such that~$\mathbb{P}(V) = \iota_\rho (\xi_1) \oplus \cdots \oplus \iota_\rho (\xi_n)$. 
\end{fact}

\subsubsection{Continuity of the orbit for Anosov representations}\label{sect_continuity_anosov} The map~$\rho \mapsto \xi_\rho$, associating to a~$\Theta$-Anosov representation its limit map, is continuous on~$\operatorname{Hom}_{\Theta\text{-}\mathsf{An}}(\Gamma, G)$ \cite{guichard2012anosov}. In the next lemma, using Fact~\ref{prop_ggkw}, we extend this continuity property to some orbits of~$\rho(\Gamma)$ in~$\Fl(\g, \Theta)$ in the case where the hypersurfaces~$\hyp_{\xi_\rho^\opp(\eta)}$, with~$\eta \in \partial_\infty \Gamma$, do not cover~$\Fl(\g, \Theta)$. 

\begin{lem}\label{lem_convergence_orbit_anosov_hausdorff} Let~$G$ be a noncompact real semisimple Lie group and~$\Theta$ be a subset of the simple restricted roots of~$G$. Let~$\Gamma$ be a word hyperbolic group with boundary~$\partial_{\infty} \Gamma$.
    Let~$\rho: \Gamma \rightarrow G$ be a~$\Theta$-Anosov representation such that the set 
    \begin{equation*}
        \mathcal{O}_{\rho} := \Fl(\g, \Theta) \smallsetminus \bigcup_{\eta \in \partial_{\infty} \Gamma} \hyp_{\xi^\opp_{\rho} (\eta)}
    \end{equation*}
is nonempty. Let~$x_0 \in \mathcal{O}_{\rho}$. Then the map 
\begin{equation*}
    \Psi:
    \operatorname{Hom}_{\Theta\text{-}\mathsf{An}}(\Gamma, G) \longrightarrow \big{\{}\text{closed subsets of } \Fl(\g, \Theta)\big{\}}; \ 
    \rho' \longmapsto \overline{\rho'(\Gamma) \cdot x_0}
\end{equation*}
is continuous at~$\rho$ for the Hausdorff topology.
\end{lem}

\begin{proof} By Fact~\ref{prop_ggkw}.(4), it suffices to prove the lemma for~$G = \PGL(n, \mathbb{R})$ and~$\Theta = \{\alpha_1\}$ the first simple root of~$G$, i.e.\ $\Fl(\g, \Theta) = \mathbb{P}(\mathbb{R}^n)$ and~$\Fl(\g, \Theta)^\opp = \mathbb{P}((\mathbb{R}^n)^*)$, where~$n \in \mathbb{N}_{\geq 2}$. For any infinite-order element~$g \in \Gamma$, we denote by~$g^+$ and~$g^\opp$ the attracting and repelling fixed point of~$g$ in~$\partial_\infty \Gamma$.

It suffices to prove that for any sequence of representations~$(\rho_k) \in \operatorname{Hom}_{\{\alpha_1\}\text{-}\mathsf{An}}(\Gamma, \PGL(n, \mathbb{R}))^{\mathbb{N}}$ and for any diverging sequence~$(g_k) \in \Gamma^{\mathbb{N}}$, any limit point of the sequence~$(\rho_k(g_k) \cdot x_0)$ converges to an element of~$\xi_\rho(\partial_\infty \Gamma)$. We fix~$x$ an accumulation point of~$(\rho_k(g_k) \cdot x_0)$. Up to extracting, we may assume that~$\rho_k(g_k) \cdot x_0 \rightarrow x$.

Since~$\rho$ is~$\{\alpha_1\}$-Anosov, by \cite{guichard2012anosov} there exist a neighborhood~$\mathcal{U}$ of~$\rho$ in~$\operatorname{Hom}_{\{\alpha_1\}\text{-}\mathsf{An}}(\Gamma, \PGL(n, \mathbb{R}))$ and two constants~$D >1, L >0$ such that for all~$\rho' \in \mathcal{U}$ and~$g \in \Gamma$:
\begin{equation}\label{eq_ineq_singular_values}
    \alpha_1\big{(}\mu(\rho'(g)) \big{)} \geq \frac{1}{D} |g| - L,
\end{equation}
where~$\mu$ is the Cartan projection of~$\PGL(V)$ (see Section~\ref{sect_cartan_decomp}) and~$|\cdot|$ is the word length on~$\Gamma$, determined by a finite generating set of~$\Gamma$. Since for~$k$ large enough we have~$\rho_k \in \mathcal{U}$, we may apply~\eqref{eq_ineq_singular_values} to~$\rho_k(g_k^N)$, for~$k$ large enough and all~$N \in \mathbb{N}$. Dividing it by~$N \in \mathbb{N}$ and making~$N \rightarrow + \infty$, we get
\begin{equation}\label{eq_ineq_lyapounov}
    \alpha_1\big{(}\lambda(\rho_k(g_k)) \big{)} \geq \frac{1}{D} |g_k|_\infty,
\end{equation}
where~$\lambda$ is the Lyapunov projection of~$\PGL(n, \mathbb{R})$ defined in Section~\ref{sect_real_lie_alg}, and~$|\cdot|_{\infty}$ is the \emph{stable length} on~$\Gamma$ defined by~$|g|_\infty = \lim_{k \rightarrow + \infty} |g^k|/k$ for all~$g \in \Gamma$.

Up to further extracting, we may assume that there exist~$a,b \in \partial_\infty \Gamma$ such that~$g_k \rightarrow a$ and~$g_k^{-1} \rightarrow b$ in~$\Gamma \sqcup \partial_\infty \Gamma$. By \cite[Thm 1.7]{guichard2012anosov}, up to further extracting and considering~$(fg_k)$ instead of~$(g_k)$ for a suitable~$f \in  \Gamma$, we may assume that~$a \ne b$ and that~$g_k$ has infinite order for all~$k \in \mathbb{N}$. Note that we have~$a = \lim_{k \rightarrow + \infty }g_k^+$ and~$b = \lim_{k \rightarrow + \infty} g_k^-$. Since~$a \ne b$, we have~$|g_k|_{\infty} \rightarrow + \infty$ as~$k \rightarrow + \infty$. Then~\eqref{eq_ineq_lyapounov} implies that~$\alpha_1\big{(}\lambda(\rho_k(g_k)) \big{)} \rightarrow + \infty$.

Since~$a \ne b$, we may assume that~$g_k^+ \ne g_k^-$ for all~$k \in \mathbb{N}$. Then we have the decomposition~$\mathbb{R}^{n} = \xi_\rho(a) \oplus \xi_{\rho}^\opp(b) = \xi_{\rho_k}(g_k^+) \oplus \xi_{\rho_k}^\opp(g_k^\opp)$ for all~$k \in \mathbb{N}$. We denote by~$\Proj_k$ (resp.\ $\Proj$) the projection onto~$\xi_{\rho_k}(g_k^+)$ (resp.\ onto~$\xi_\rho(a)$) parallel to~$\xi_{\rho_k}^\opp(g_k^\opp)$ (resp.\ parallel to~$\xi_\rho^\opp(b)$), and by~$\Proj_k^\opp$ (resp.\ $\Proj^\opp$) the projection onto~$\xi_\rho^\opp(g_k^\opp)$ (resp.\ onto~$\xi_\rho^\opp(b)$) parallel to~$\xi_{\rho_k}(g_k^+)$ (resp.\ parallel to~$\xi_\rho(a)$). Note that~$\Proj_k \rightarrow \Proj$ and~$\Proj_k^\opp \rightarrow \Proj^\opp$ as~$k \rightarrow + \infty$, in~$\operatorname{End}(\mathbb{R}^{n})$.

Let~$\lVert\cdot \rVert$ be the~$L^2$-norm on on~$\mathbb{R}^{n}$. For all~$k \in \mathbb{N}$, we denote by~$A_k$ the unique lift of~$\rho_k(g_k)$ in~$\GL(n, \mathbb{R})$ such that the induced endomorphism of~$\xi_{\rho_k(g_k^+)}$ by~$A_k$ is~$\pm \operatorname{id}$ (such a lift exists because $\xi_{\rho_k}(g_k^+)$ is fixed by~$\rho_k(g_k)$). Let~$v_0 \in \mathbb{R}^n \smallsetminus \{ 0 \}$ be a lift of~$x_0$. By Remark~\ref{rmk_lyapunov_projection}, we have
\begin{equation*}
    \big{\lVert} A_k \cdot \Proj_k^\opp(v_0) \big{\rVert} \leq e^{-\alpha_1 \big{(} (\lambda(\rho_k(g_k))\big{)}} \big{\lVert} \Proj_k^\opp(v_0) \big{\rVert} \quad \forall k \in \mathbb{N},
\end{equation*}
with~$\lVert \Proj_k^\opp(v_0) \rVert \rightarrow \lVert \Proj^\opp(v_0) \rVert$ and~$e^{-\alpha_1 \big{(} \lambda(\rho_k(g_k))\big{)}} \rightarrow 0$, so~$\lVert A_k \cdot \Proj^\opp_k(v_0) \rVert \rightarrow 0$. Then we have:
\begin{equation}\label{eq_convergence_g_k}
    A_k \cdot v_0 = A_k \cdot \Proj^\opp_k(v_0) + A_k \cdot \Proj_k(v_0) = A_k \cdot \Proj^\opp_k(v_0) \pm \Proj_k(v_0) \underset{k \rightarrow + \infty}{\longrightarrow} \pm \Proj(v_0).
\end{equation}
Since~$x_0 \notin \xi_\rho^\opp(b)$, we know that~$\Proj(v_0) \ne 0$. Hence, we can project~\eqref{eq_convergence_g_k} in~$\mathbb{P}(\mathbb{R}^n)$, and we get~$\rho_k(g_k) \cdot x_0 \rightarrow \xi_\rho(a)$, so~$x = \xi_\rho(a)$. $\qed$
\end{proof}

\subsection{Causal flag manifolds}\label{sect_causal_flag_manifolds}

If~$G$ is a Hermitian simple Lie group of tube type, that is, if the symmetric space~$\mathbb{X}_G$ of~$G$ is irreducible and Hermitian of tube type, then we will say that~$G$ is a \emph{$\HTT$ Lie group}, and~$\g$ a \emph{$\HTT$ Lie algebra}. In this section, we fix an~$\HTT$ Lie group~$G$ and recall basic properties of~$G$.

\subsubsection{Strongly orthogonal roots and root system} \label{sect_root_syst_shilov} Two roots~$\alpha, \beta \in \Sigma$ are called \emph{strongly orthogonal} if neither~$\alpha + \beta$ nor~$\alpha- \beta$ is a restricted root. Since~$G$ is of tube type, there exists a (maximal) set~$\{ 2\e_1, \cdots, 2\e_r \} \subset \Sigma$ of strongly orthogonal roots, such that the set~$\FS= \{ \alpha_1, \cdots, \alpha_r \}$ is a fundamental system of~$\Sigma$, where~$\alpha_i = \e_i - \e_{i+1}$ for~$i < r$ and~$\alpha_r = 2 \e_r$. The system~$\Sigma$ is then of type~$C_r$ (see e.g.\ \cite{faraut1994analysis}):
\begin{equation}\label{eq_fund_sys_Cr}
    \begin{split}
        \Sigma  &= \{\pm \e_i \pm \e_j \mid 1 \leq i \leq j \leq r\}; \quad  \Sigma^+ = \{\e_i \pm \e_j \mid 1 \leq i < j \leq r\} \cup \{2 \e_i \mid 1 \leq i \leq r\}; \\
        \FS &= \{ \alpha_i := \e_i - \e_{i+1} \mid 1 \leq i  \leq r-1\} \cup \{\alpha_r := 2\varepsilon_r\}.
    \end{split}
\end{equation}

If~$\Theta = \{\alpha_r\}$, then the flag manifold~$\Fl(\g, \Theta)$ is called the \emph{Shilov boundary} of~$\mathbb{X}_G$ and we will denote it by~$\SB(\g)$. These flag manifolds are all listed in Table~\ref{table_shilov_bnds} and are all self-opposite.

\begin{notation}\label{sect_notation} When~$G$ is a~$\HTT$ Lie group and~$\Theta = \{ \alpha_r\}$, we will always use the following simplified notation:
\begin{equation*}
   \uu^{\pm} = \uu_{\{\alpha_r\}}^{\pm}, \ U^{\pm} = U_{\{\alpha_r\}}^{\pm}, \ \mathfrak{l} = \mathfrak{l}_{\{\alpha_r\}}, \ L=  L_{\{\alpha_r\}}, \ \LP^\pm = \LP_{\{\alpha_r\}}^\pm, \ \ \LP^\pm = \LP_{\{\alpha_r\}}^\pm.
\end{equation*}
\end{notation}
 Note that the Lie algebras~$\uu^{\pm}$ are abelian.
 
The center of~$\mathfrak{l}$ is one-dimensional, and one can write~$\mathfrak{l} = \mathfrak{l}_s \oplus \mathbb{R} H_0$, where~$H_0$ is in the center of~$\mathfrak{l}$ and~$\mathfrak{l}_s$ is the semisimple part of~$\mathfrak{l}$. The possible values of~$\mathfrak{l}_s$ are listed in Table \ref{table_shilov_bnds}. We will denote by~$L_s$ the connected Lie subgroup of the identity component~$L^0$ of~$L$ with Lie algebra~$\mathfrak{l}_s$.

For all~$1 \leq i \leq r-1$, we write~$\alpha_i := \varepsilon_i - \varepsilon_{i+1}$. Then: 
\begin{fact}\label{eq_root_syst_L_s}
    A fundamental system of simple restricted roots of~$\mathfrak{l}_s$ (resp.\ of~$L_s$) is~$\{\alpha_1, \dots, \alpha_{r-1}\}$, of type~$A_{r-1}$.
\end{fact}

    \begin{table}[H]
        \centering
        \begin{tabular}{|c|c|c|}
        \hline
          ~$\g$   &~$\SB(\g)$ &~$\mathfrak{l}_{s}$ \\
         \hline
            ~$\soo(2,n)$,~$n \geq 3$ & ~$\Ein^{n-1,1}$&~$\soo(n-1,1)$  \\
           \hline
             ~$\spp(2r, \mathbb{R})$ &~$\Lag_r(\mathbb{R}^{2r})$&$\mathfrak{sl}(r, \mathbb{R})$  \\
           \hline
             ~$\uu(r,r)$  &$\Lag_r(\mathbb{C}^{2r})$& ~$\mathfrak{sl}(r, \mathbb{C})$  \\
           \hline
              ~$\soo^*(4r)$  &~$\Lag_r(\mathbb{H}^{2r})$&  ~$\mathfrak{sl}(r, \mathbb{H})$  \\
           \hline
      ~$\mathfrak{e}_{7(-25)}$ & $(E_{6(-26)}/ F_4) \times \Rf$  & ~$\mathfrak{e}_{6(-26)}$  \\
           \hline
        \end{tabular}
        \caption{Shilov boundaries associated with all~$\HTT$ Lie algebras. For the notation~$\mathfrak{e}_{7(-25)}$ and~$\mathfrak{e}_{6(-26)}$, see \cite{faraut1994analysis} or~\cite{onishchik2012lie}. See Example~\ref{ex_shilov_bndrs} for the other notation used in the table.} \label{table_shilov_bnds}
    \end{table}

\begin{ex}\label{ex_shilov_bndrs}
(1). \emph{The Lagrangians.} Let~$\Kf = \Rf, \Cf$ or~$\Hf$, and let~$r \geq 2$. Let~$J_{\Kf} = \begin{pmatrix}
        0 & -I_r \\ I_r & 0
    \end{pmatrix}$, where~$I_r$ is the identity matrix of size~$r$. Given a matrix~$g \in \operatorname{Mat}_{2r}(\Kf)$, we denote by~$\overline{g}$ the matrix whose~$(i,j)$-th entry is the conjugate (in~$\Kf$) of the~$(i,j)$-th entry of~$g$. Let~$G_{\Kf} := \left\{g \in \SL(2r, \Kf) \mid   \ ^t\! \overline{g} J_{\Kf} g = J_{\Kf}\right\}$. Then~$G_\mathbb{K}$ is a~$\HTT$ Lie group, and one has~$G_{\Rf}= \Sp(2r, \Rf)$,~$G_{\Cf} = \SU(r,r)$ and~$G_{\Hf} = \SO^* (4r)$. The space~$\aaa := \{D = \operatorname{diag}(\lambda_1, \dots, \lambda_r, -\lambda_1, \dots, - \lambda_r ) \mid \lambda_i \in \mathbb{R} \ \forall 1 \leq i \leq r\}$ is a Cartan subspace of~$\g_{\mathbb{K}}$. If we define~$\varepsilon_i: D \mapsto \lambda_i$ on~$\aaa$, then the strongly orthogonal roots defined in Section~\ref{sect_root_syst_shilov} can be taken to be~$(2 \varepsilon_i)_{1 \leq i \leq r}$. Let~$\FS$ be the associate fundamental system simple restricted roots of~$\g_{\mathbb{K}}$ (by Equation~\eqref{eq_fund_sys_Cr}). 

Now let~$\mathbf{b}$ be the bilinear form whose matrix in the canonical basis of~$\Kf^{2r}$ is~$J_{\Kf}$. The space~$\Lag_r(\Kf^{2r})$ is the space of Lagrangians of~$(\Kf, \mathbf{b})$, on which the group~$G_\Kf$ acts transitively. If~$(e_1, \dots e_{2r})$ is the canonical basis of~$\Kf^{2r}$, then the parabolic~$P$ of Notation~\ref{sect_notation} is the stabilizer in~$G_{\Kf}$ of~$\operatorname{Span}(e_1, \dots , e_r)$, and~$P^{\opp}_{\{\alpha_r\}}$ is the stabilizer in~$G_{\Kf}$ of~$\operatorname{Span} (e_{r+1}, \dots, e_{2r})$. Thus Equation~\eqref{ident_flag_mfds} gives a~$G_{\mathbb{K}}$-equivariant identification~$\Lag_r(\Kf^{2r}) \simeq  \SB(\g_\mathbb{K})$. This model gives~$\uu^- = \Big{\{} \begin{pmatrix} 0_r & 0_r \\ X & 0_r\end{pmatrix} \mid \ ^t\! \overline{X} = X \Big{\}}$.
\vspace{0.2cm}

(2). \emph{The Lorentzian Einstein universe.}
Let~$n \geq 2$. Let~$\b$ a bilinear form of signature~$(n, 2)$ on~$\mathbb{R}^{n+2}$ (where~$n$ is the number of positive directions and~$2$ the number of negative directions). The \emph{Einstein universe} of signature~$(n-1, 1)$, or \emph{Lorentzian Einstein universe}, is the space of isotropic lines of~$(\mathbb{R}^{n+2}, \b)$, given by~$\Ein^{n-1, 1} = \mathbb{P}(\{v \in \mathbb{R}^{n+2} \mid  \b(v,v)  = 0\})$. The group~$G := \PO(\b) \simeq \PO(n, 2)$ acts transitively on~$\Ein^{n-1, 1}$. If~$P_x$ is the stabilizer of a point~$x \in \Ein^{n-1, 1}$, then~$P_x$ is conjugate to the parabolic subgroup~$P_{\{\alpha_2\}}$ of~$G$. Thus Equation~\eqref{ident_flag_mfds} gives a~$\PO(n, 2)$-equivariant identification~$\Ein^{n-1, 1} \simeq \SB(\soo(n,2))$.

Given a point~$x \in \Ein^{n-1, 1}$, we have~$\hyp_x = \mathbb{P}(x^\perp) \cap \Ein^{n-1, 1}$, and the set~$\hyp_x$ is called the \emph{lightcone of~$x$}. The bilinear form induces a class of Lorentzian metrics~$[\psi]$ on the affine chart~$\Affstd_x$. We then have~$\Affstd_\std \cap \hyp_{y_0} = \{y \in \Affstd_\std \mid \psi(y-y_0) = 0\}$ for all~$y_0 \in \Affstd_x$ (see Figure~\ref{figure_past_and_future}). See \cite[Sect.\ 3.4.2]{galiay2024rigidity} for a more explicit description of~$\Ein^{n-1, 1}$.
\end{ex}

\begin{rmk}\label{rmk_racines_inversées} With our conventions, the root defining $\Ein^{n-1, 1} = \SB(\soo(n,2))$ is $\alpha_2$. This is because we consider the root system of $\soo(n,2)$ to be of type $C_2$. However, in the literature, the order of the roots is often reversed, and the root system of $\soo(n,2)$ is more commonly seen as being of type $B_2$. With this convention, the root defining $\SB(\g)$ is then $\alpha_1$.
\end{rmk}

\subsubsection{An invariant cone and causality} \label{sect_causalité}
It is a classical fact from \cite{kostant2010root} that the identity component~$L^0$ of~$L$ acts irreducibly on~$\uu^-$. By \cite[Prop.\ 4.7]{benoist2000automorphismes} applied to this action, there exists an open~$L^0$-invariant properly convex cone~$c^0$ in~$\uu^-$ (see e.g.\ \cite{guichard2022generalizing}). This cone is defined as the interior of the convex hull in~$\uu^-$ of the orbit~$\Ad(L^0) \cdot v^-$, where~$v^\opp$ is a nonzero vector of~$\g_{-\alpha_r}$. 

Let~$\Affstd$ be an affine chart of~$\SB(\g)$. There exists~$g \in G$ such that~$\Affstd = g \cdot \Affstd_{\mathsf{std}}$ (recall Equation~\eqref{eq_standard_chart}). Given some point~$x \in \Affstd$, there exists a unique~$X \in \uu^-$ such that~$x = g\exp(X) \cdot \LP^+$. The set~$ \I_{\Affstd}(x) := (g \exp(X + c^0) \cdot \LP^+) \cup  (g \exp(X - c^0) \cdot \LP^+)$ only depends on~$\Affstd$ and~$x$, not on~$g$. It has two connected components, denoted by~$\I_{\Affstd}^{+} (x) $, called the \emph{future} of~$x$ in~$\Affstd$, and by~$\I_{\Affstd}^\opp (x)$, called the \emph{past} of~$x$ in~$\Affstd$. The choice of these components depends on~$g$. However, we can chose them in a continuous way, i.e.\ we can chose~$\I_{\Affstd}^{+} (x) $ (resp.\ $\I_{\Affstd}^\opp (x)$) for all~$x \in \Affstd$ so that the map $x \mapsto \I_{\Affstd}^{+} (x)$ (resp.\ $ x \mapsto \I_{\Affstd}^- (x)$)
is continuous on~$\Affstd$ for the Hausdorff topology. We will implicitly make such a continuous choice (called a choice of \emph{time orientation}) each time we fix an affine chart~$\Affstd$ of~$\SB(\g)$. 

\begin{rmk}
    We have just endowed the manifold~$M := \SB(\g)$ with an invariant causal structure (see \cite{kaneyuki2006causal}), i.e.\ a smooth~$G$-equivariant (up to opposition) family of properly convex open cones~$(c_x)_{x \in M}$ in~$TM$. By~\cite{neeb2025open}, Shilov boundaries associated with ~$\HTT$ Lie groups are the only flag manifolds admitting a causal structure, which is why we will sometimes refer to them as \emph{causal flag manifolds}. 
\end{rmk}

We can now define  
\begin{align*}
    &\mathbf{J}_{\Affstd}^{\pm}(x) := \overline{\I_{\Affstd}^{\pm} (x)} \text{, the \emph{large future} (resp.\ \emph{large past}) of }x \text{ in } \Affstd; \\
        &\mathbf{C}_{\Affstd}^{\pm} (x) := \partial \I_{\Affstd}^{\pm} (x) \text{, the \emph{future lightcone} (resp.\ \emph{past lightcone}) of }x \text{ in } \Affstd; \\
       &\mathbf{C}_{\Affstd}(x) := \mathbf{C}_{\Affstd}^+(x) \cup \mathbf{C}_{\Affstd}^- (x) \text{, the \emph{lightcone} of }x \text{ in } \Affstd. 
\end{align*}

When~$\Affstd = \Affstd_{\mathsf{std}}$, we will ommit the ``$\Affstd$" in subscript. These sets satisfy the following straightforward properties:

\begin{fact}\label{fact_cone_in_schubert_subvrty}
Let~$\Affstd$ be an affine chart.
    \begin{enumerate}
    \item The lightcone~$\mathbf{C}_{\Affstd}(x)$ of~$x \in \mathbb{A}$ is always contained in~$\hyp_x \cap \mathbb{A}$, and~$\I_{\Affstd}^{\pm}(x)$ are connected components of~$\Affstd \smallsetminus \hyp_x$.
    \item 
    For all~$x,y,z \in \Affstd$, one has:
    \begin{enumerate}
        \item [*]\emph{(reflexivity)}~$x \in \mathbf{J}_{\Affstd}^+(y)\Leftrightarrow y \in \mathbf{J}_{\Affstd}^-(x)$;
        \item [*]\emph{(antisymmetry)}~$\mathbf{J}_{\Affstd}^+(y) \cap \mathbf{J}_{\Affstd}^-(y) = \{ y \}$;
        \item [*]\emph{(transitivity)}~$[x \in \mathbf{J}_{\Affstd}^{\pm}(y) \text{ and } y \in \mathbf{J}_{\Affstd}^\pm(z) \text{ } ] \Rightarrow x \in \mathbf{J}_{\Affstd}^{\pm} (z)$
    \end{enumerate}
\end{enumerate}
Reflexivity and antisymmetry are also true replacing ``$\mathbf{J}$" with ``$\mathbf{C}$". Reflexivity and transitivity are also true replacing ``$\mathbf{J}$" with ``$\mathbf{I}$".
\end{fact}

\subsection{Reminders on proper domains in flag manifolds}\label{sect_proper_domains_flag_manifolds}
In this section, we recall some definitions and properties of domains in a flag manifold, generalizing those of classical convex projective geometry.

\subsubsection{Generalities on proper domains}\label{sect_generalities_proper_domains}
Let~$G$ be a noncompact real semisimple Lie group and~$\Theta \subset \FS$ a subset of the simple restricted roots.

\begin{definition}\label{def_domain_proper}
    Let~$X \subset \Fl(\g, \Theta)$ be a subset. We say that~$X$ is:
    \begin{enumerate}
        \item a \emph{domain} if~$X$ is open, nonempty and connected;
        \item \cite{zimmer2018proper} \emph{proper} if there exists~$\xi \in \Fl(\g, \Theta)^\opp$ such that~$\overline{X}\cap \hyp_{\xi}  
 = \emptyset$. In particular, if~$\xi = \LP_\Theta^\opp$, then we will say that~$X$ is \emph{proper in~$\Affstd_\std$}. This is equivalent to saying that~$\overline{X}\subset \Affstd_\std$.
    \end{enumerate}    
\end{definition}

\begin{rmk}\label{Rmk_contenu_dans_A_std}
    Given a proper domain~$\O$ of~$\Fl(\g, \Theta)$, since~$G$ acts transitively on~$\Fl(\g, \Theta)^-$, we will always be able to assume that~$\O$ is proper in~$\Affstd_\std$.
\end{rmk}

\subsubsection{The automorphism group} \label{sect_group_aut_omega}

Given a subset~$X \subset \Fl(\g, \Theta)$, the \emph{automorphism group} of~$X$ is
\begin{equation*}
    \Aut(X) = \left\{ g \in \Aut_\Theta(\g) \mid g \cdot X = X \right\}.
\end{equation*}
\begin{fact}\label{fact_autom_group_proper}\cite{zimmer2018proper} Let~$\O \subset \Fl(\g, \Theta)$ be an open subset. The group~$\Aut(\O)$ is a Lie subgroup of~$G$. Moreover, it acts properly on~$\O$ as soon as~$\O$ is proper.
\end{fact}

Given a proper open subset~$\O \subset \Fl(\g, \Theta)$, the \emph{full orbital limit set} of~$H$ in~$\O$ is:
\begin{equation*}
    \Lambda_\O^{\operatorname{orb}}(H)= \bigcup_{x \in \O} \overline{H \cdot x} \smallsetminus(H \cdot x).
\end{equation*}
Since~$H$ acts properly on~$\O$ (recall Fact~\ref{fact_autom_group_proper}), we have 
\begin{equation}\label{eq_lim_orb_included_bndry}
    \Lambda_\O^{\operatorname{orb}}(H) \subset \partial \O.
\end{equation}
\begin{rmk}
\begin{enumerate}
    \item In the case where~$\g$ is a~$\HTT$ Lie algebra and~$\Fl(\g, \Theta) = \SB(\g)$, the group~$\Aut(\O)$ is commensurable to the \emph{conformal group} of~$\O$, that is, the group of all invertible maps from~$\O$ to itself that preserve the causal structure of~$\O$ \cite[Thm 2.3]{kaneyuki2011automorphism}.
    \item If~$\g$ is a~$\HTT$ Lie algebra of real rank~$r \geq 1$ and~$\Theta = \{\alpha_r\}$, or if~$\g = \soo(p+1, q+1)$ and~$\Theta = \{\alpha_1\}$, and it is proved in \cite{galiay2024rigidity} and~\cite{chalumeau2024rigidity} respectively that, whenever  \eqref{eq_lim_orb_included_bndry} is an equality, the domain~$\O$ is a \emph{diamond} (see Section~\ref{ssect_def_diamonds} for the definition of \emph{diamond} in~$\SB(\g)$).
\end{enumerate}
\end{rmk}

\subsubsection{The dual }\label{sect_dual_proper_domain} Let~$X \subset \Fl(\g, \Theta)$ be a subset. The \emph{dual of~$X$} is the set
\begin{equation*}
    X^*:= \{\xi \in \Fl(\g, \Theta)^\opp \mid \hyp_{\xi} \cap \ X = \emptyset\} \subset \Fl(\g, \Theta)^\opp.
\end{equation*}
Let us recall some properties of the dual (see \cite{zimmer2018proper}):
\begin{enumerate}
    \item The dual~$X^*$ of a set~$X \subset \Fl(\g, \Theta)$ is~$\Aut(X)$-invariant. Equivalently, one has~$\Aut(X) \subset \Aut(X^*)$.
    \item If~$\O \subset \Fl(\g, \Theta)$ is open, then~$\O^*$ is compact. 
    \item An open subset~$\O$ is proper if and only if its dual~$\O^*$ has nonempty interior.
\end{enumerate} 

We will need the following lemma:

\begin{lem}\label{lem_proper_inter_hyp}
    Let~$\O \subset \Fl(\g, \Theta)$ be a proper open subset, and~$\xi \in \Lambda_{\oppinv(\Theta)}(\Aut(\O))$. Then~$\xi \in \O^*$.
\end{lem}

\begin{proof}
 There exists a~$\Theta$-contracting sequence~$(g_k) \in \Aut(\O)^\mathbb{N}$ such that~$(g_k)$ has~$\Theta$-limit~$\xi$. Since~$\O^*$ has nonempty interior, there exists~$z \in \O^*$ such that~$g_k \cdot z \rightarrow \xi$. Since~$\O^*$ is~$\Aut(\O)$-invariant and closed, we have~$\xi \in \O^*$. Thus by definition of~$\O^*$, one has~$\hyp_{\xi} \cap \O = \emptyset$.
\end{proof}

In \cite{zimmer2018proper}, Zimmer defines the following notion of convexity: 

\begin{definition}\label{def_dual_convexity}
    An open subset~$\O \subset \Fl(\g, \Theta)$ is \emph{dually convex} if for all~$a \in \partial \O$, there exists~$\xi \in \O^*$ such that~$a \in \hyp_\xi$.
\end{definition}
Given a proper domain~$\O \subset \Fl(\g, \Theta)$, its \emph{bidual}~$\O^{**}$ is a proper open set containing~$\O$, not necessarily connected (see e.g.\ \cite[Rmk 5.3.3]{galiay2025convex}). It is however dually convex, and each of its connected components are also dually convex.

\begin{definition}\label{def_enveloppe_dc}
    Let~$\O \subset \Fl(\g, \Theta)$ be a proper domain. We denote by~$\O^{**}_0$ the connected component of~$\O^{**}$ which contains~$\O$. It is a proper~$\Aut(\O)$-invariant dually convex domain of~$\Fl(\g, \Theta)$ containing~$\O$, and we call it the \emph{dual convex hull} of~$\O$.
\end{definition}
If~$G = \PGL(n, \mathbb{R})$ and~$\Theta = \{\alpha_1\}$, then~$\Fl(\g, \Theta) = \mathbb{P}(\mathbb{R}^n)$, and the domain~$\O_0^{**}$ coincides with the classical convex hull of~$\O$.

\section[Convexity in causal flag manifolds]{Convexity in causal flag manifolds}\label{sect_proj_geom_causal}

The goal of this section is to define a natural notion of convexity in causal flag manifolds. By ``natural", we mean that the property for a subset~$X \subset \SB(\g)$ to be convex should not depend on the choice of affine chart containing it. A notion that proves to be suitable is that of \emph{causal convexity}, which involves the causal structure of~$\SB(\g)$, in analogy with the already-existing causal convexity in \emph{conformal spacetimes}~\cite{Sanchez2008}; see Definition~\ref{def_causally_convex_indep}.

In Sections ~\ref{ssect_def_diamonds} and~\ref{sect_def_Ois}, we investigate the open~$L^0$-orbits in~$\SB(\g)$, two of which are preferred and called \emph{diamonds}. We then define causal convexity (Section~\ref{sect_causal_convexity_def_properties}); to this end, we first define it in affine charts (Section~\ref{sect_def_causal_affine_chart}), and  compare it to the notion of dual convexity (Section~\ref{sect_link_dual_convexity} and Proposition~\ref{prop_dual_implies_causal}). Finally, we show that the notion of causal convexity is in fact independent of the choice of affine chart, making it an intrinsic property for domains of~$\SB(\g)$ contained in affine charts (Section~\ref{sect_causal_convexity_def}).

For all this section, we take Notation~\ref{sect_notation}.

\subsection{Reminders on diamonds }\label{ssect_def_diamonds} This section contains reminders on diamonds, which can also be found in \cite{galiay2024rigidity}. 

Fix a~$\HTT$ Lie group~$G$. Given two transverse points~$x,y \in \SB(\g)$, the set~$\SB(\g) \smallsetminus (\hyp_x \cup \hyp_y)$ admits several connected components, exactly two of which are proper.
\begin{definition}\label{def_diamant} A subset~$\O$ of~$\SB(\g)$ is called a \emph{diamond} if there exists a (unique) pair of transverse points~$x,y \in \SB(\g)$ such that~$\O$ is one of the two proper connected components of~$\SB(\g) \smallsetminus (\hyp_x \cup \hyp_y)$. The two points~$x,y$ are then called the \emph{endpoints of~$\O$}.
\end{definition}

Let~$\Diams_{\mathsf{std}} := \I^+(\LP^+)$. Recall from Section~\ref{sect_parab_subgroups} that there exists an order-two element~$k_0\in G$ such that~$k_0 \cdot \LP^+ = \LP^\opp$. Then~$\Diams_{\mathsf{std}}' :=  \mathbf{I}^-(\LP^+) = k_0 \cdot \Diams_{\mathsf{std}}$ is the interior of the dual of~$\Diams_{\mathsf{std}}$ (see e.g.\ \cite[Lem.\ 13.11]{guichard2022generalizing}), and the domains~$\Diams_{\mathsf{std}}$ and~$\Diams_{\mathsf{std}}'$ are exactly the two diamonds with endpoints~$\LP^+$ and~$\LP^{\opp}$. The explicit description of~$\Diams_\std$ with the identifications given in Example~\ref{ex_shilov_bndrs} is given in Example~\ref{ex_Ois} (see also~\cite[Ex.\ 4.10]{galiay2024rigidity}).

Given two transverse points~$x,y \in \SB(\g)$, one has~$(x,y)=g \cdot (\LP^+, \LP^{\opp})$ for some~$g \in G$. The two diamonds~$g \cdot \Diams_{\mathsf{std}}$ and~$g \cdot \Diams_{\mathsf{std}}'$ are the diamonds with endpoints~$x$ and~$y$. Thus, any diamond is a~$G$-translate of~$\Diams_{\mathsf{std}}$.

It is convenient to consider models of diamonds that are proper in affine charts:

\begin{definition} Let~$\Affstd \subset \SB(\g)$ be an affine chart. If~$x,y \in \Affstd$ and~$y \in \I^+_\Affstd (x)$, we define~$\Diams_\Affstd(x,y)$ as the set~$\I^+_\Affstd(x) \cap \I^-_\Affstd(y)$. It is one of the two diamonds with endpoints~$x$ and~$y$.
\end{definition}
For~$x,y \in \Affstd$ and~$y \in \I^+_\Affstd (x)$, the diamond~$\Diams_\Affstd(x,y)$ is the only one of the two diamonds with endpoints~$x$ and~$y$ that is proper in~$\Affstd$; see Figure \ref{fig:enter-label}. When~$\Affstd = \Affstd_\std$, we will ommit the ``$\Affstd_\std$" in subscript.

The following fact is well known (see e.g.\ \cite[Thm 2.3 and 3.5]{kaneyuki2011automorphism}): 
\begin{fact}\label{fact_autom_diam}
The action of the identity component~$L^0$ of~$L$ on~$\Diams_{\mathsf{std}}$ is transitive and the stabilizer of a point in~$\Diams_{\mathsf{std}}$ is a maximal compact subgroup of~$L^0$, so that any diamond is a model for the symmetric space of~$L^0$. 
\end{fact}
Since~$L^0$ is commensurable to~$L_s \times \mathbb{R}$ (Section~\ref{sect_root_syst_shilov}), the diamond~$\Diams_{\mathsf{std}}$ is~$L^0$-equivariantly diffeomorphic to~$ \mathbb{X}_{L_s} \times \Rf$, where~$\mathbb{X}_{L_s}$ is the symmetric space of~$L_s$. The corresponding identifications, for~$\g$ ranging in~$\HTT$ Lie algebras, are listed in \cite[Table 2]{galiay2024rigidity}.

\begin{figure}[H]
\begin{center}
\begin{minipage}[b]{.5\linewidth}
\begin{center}
\begin{tikzpicture}
            \fill[fill=black] (0,0) circle (1pt);
            \draw node at (0.4,0) {$\LP^+$};
             \draw (-1.7,-1.7) -- (0,0) -- (-1.7,1.7);
             \draw (1.7,-1.7) -- (0,0) -- (1.7,1.7);   \draw (-1,1) arc (180:360:1 and 0.15);
          \draw[dashed] (1,1) arc (0:180:1 and 0.15); 
           \draw (-1,-1) arc (180:360:1 and 0.15);
          \draw[dashed] (1,-1) arc (0:180:1 and 0.15);
          \draw node at (0,1.6) {$\mathbf{I}^+(\LP^+)~$};
          \draw node at (0,-1.6) {$\mathbf{I}^-(\LP^+)~$};
          
        \end{tikzpicture}
\caption{Past and future of~$\LP^+$ in~$\Affstd_\std$ in the case where~$(p,q) = (2,1)$}
    \label{figure_past_and_future}
\end{center}
\end{minipage} \hfill
\begin{minipage}[b]{.47\linewidth}
\begin{center}
\begin{tikzpicture}
          \draw (-2,1) -- (0,-1) -- (2,1);
          
          \draw (-2, -1) -- (0,1) -- (2,-1);
          \draw (-1,0) arc (180:360:1 and 0.15);
           \draw [fill, pattern = north east lines, opacity=0.7] (-1,0) -- (0,1) -- (1,0) -- (0,-1);
          
          \draw[dashed] (1,0) arc (0:180:1 and 0.15);
          \fill[fill=black] (0,1) circle (1pt);
          \fill[fill=black] (0,-1) circle (1pt);
          \draw node at (0.15,1.15) {$y$};
          \draw node at (0.15,-1.15) {$x$};
        \end{tikzpicture}
\caption{The diamond $\Diams(x,y)$ for $x,y \in \Affstd_\std$ and $y \in \mathbf{I}^+(x)$ (greyed-out area), seen in~$\Affstd_\std \simeq \mathbb{R}^{2,1}$ for~$(p,q) = (2,1)$.}
    \label{fig:enter-label}
\end{center}
\end{minipage}
\end{center}
\end{figure}

Given an affine chart~$\Affstd$ and~$x,y \in \Affstd$ such that~$y \in \I^+_{\Affstd} (x)$, the diamond~$\Diams_{\Affstd}(x,y)$ is the only one of the two diamonds with endpoints~$x$ and~$y$ that is proper in~$\Affstd$. The converse is true:

\begin{lem}\label{lem_diamond_bounded} Let~$\Affstd$ be an affine chart of~$\SB(\g)$. Let~$x,y \in \Affstd$ be two transverse points. Assume that there exists a diamond~$D$ with endpoints~$x$ and~$y$ such that~$\overline{D} \subset \Affstd$. Then one of the following is satisfied:
\begin{enumerate}
    \item One has~$y \in \I_{\Affstd}^+(x)$ and~$D = \Diams_\Affstd(x,y)$.
    \item One has~$y \in \I_{\Affstd}^-(x)$ and~$D = \Diams_\Affstd(y,x)$.
\end{enumerate}
\end{lem}

\begin{proof} Let~$z \in \SB(\g)$ be such that~$\Affstd = \Affstd_z$. Since~$\overline{D} \subset \Affstd_z$, the point~$z$ lies in the interior of~$D^*$. Then by \cite[Prop.\ 13.15]{guichard2018positivity}, the point~$x$ belongs to one of the two diamonds with endpoints~$y,z$. These two diamonds are exactly~$\I_{\Affstd}^+(y)$ and~$\I_{\Affstd}^-(y)$, and the lemma follows.~$\qed$
\end{proof}

If~$x,y \in \Affstd$ and~$y \in \mathbf{J}_{\Affstd}^+(x)$, then we will denote by~$\Diams_{\Affstd}^c(x,y)$ the Hausdorff limit of the sequence of~$\overline{\Diams_{\Affstd}(x,y_k)}$, where~$(y_k) \in \I^+(x)$ and~$y_k \rightarrow y$; this limit does not depend on the choice of the sequence~$(y_k)$. If~$y \in \I^+(x)$, then one has~$\Diams_{\Affstd}^c(x,y) = \overline{\Diams_{\Affstd}(x,y)}$.

\subsection{Other~$L^0$-orbits}\label{sect_def_Ois} In Section~\ref{ssect_def_diamonds}, we have defined \emph{diamonds} in~$\SB(\g)$. By Fact~\ref{fact_autom_diam}, we know that the standard diamond~$\Diams_\std$ is an open~$L^0$-orbit in~$\SB(\g)$. In this section, we investigate the other~$L^0$-orbits in~$\SB(\g)$.
Recall the strongly orthogonal roots~$2\varepsilon_1, \dots, 2\varepsilon_r$ defined in Section~\ref{sect_root_syst_shilov}. For all~$1 \leq i \leq r$, let~$v_i \in \g_{-2 \varepsilon_i}$ be such that~$[v_i, h_{- 2 \varepsilon_i}, \Cartinv(v_i)]$ is an~$\mathfrak{sl}_2$-triple, where~$h_{-2 \varepsilon_i}$ is defined in Section~\ref{sect_real_lie_alg}. For all~$1 \leq i,j \leq r$ such that~$i+j \leq r$, we define
\begin{equation*}
    X_{i,j} := v_1 + \cdots + v_i -v_{i+1} - \cdots - v_{i+j} \in \uu^-.
\end{equation*}
Let~$V_{i, j}$ be the~$L^0$-orbit of~$X_{i, j}$. We write~$\mathcal{O}_i := V_{i, r-i}$. Using the terminology of Jordan algebras and generalizing classical Sylvester's law of inertia, Kaneyuki proves that the set~$V_{k,l}$ is open if and only if~$k+l = r$, and that~$\mathcal{O}_i := V_{i, r-i}$ is the connected component of~$\uu^- \smallsetminus \varphi_{\std}^{-1}(\hyp_{\LP^+})$ containing~$X_{i, r-i}$~\cite{kaneyuki1988sylvester}. Kaneyuki also proves that~$\varphi_{\std}^{-1}(\hyp_{\LP^+}) = \bigsqcup_{i+j \leq r-1} V_{i,j}$ and
\begin{equation}\label{eq_adherence_O_i}
    \overline{\mathcal{O}}_i = \bigsqcup_{k \leq i, \ \ell \leq r-i} V_{k,l} \quad \forall 1 \leq i \leq r.
\end{equation}
Finally, one has~$c^0 = \mathcal{O}_r$. The following fact is proven in \cite[Lem.\ 2.3.(1)]{takeuchi1988basic}:

\begin{fact}\label{fact_takeuchi}
    There exists a maximal compact subgroup~$K_0$ of~$L^0$ such that for all~$1 \leq i \leq r$, we have~$\bigsqcup_{s+t = i}V_{s, t} = \Ad(K_0) \cdot \big{\{}\sum_{k=1}^r \lambda_k v_k \mid \operatorname{card}\{k \mid \lambda_k \ne 0\} = i\big{\}}$.
\end{fact}

\begin{rmk}\label{rmk_Nagano_space}
    The above fact is actually stated in \cite{takeuchi1988basic} in the more general case where the flag manifold~$\Fl(\g, \alpha)$ is a \emph{irreducible Nagano space} (or \emph{symmetric~$R$-space}) in the sense of \cite[Chap.\ 5]{galiay2025convex}.
\end{rmk}

\begin{ex} \label{ex_Ois}
\begin{enumerate}
    \item When~$\g$ is as in Example~\ref{ex_shilov_bndrs}.(1), one has~$v_i = \begin{pmatrix} 0_r & 0_r \\ E_{i,i} & 0_r\end{pmatrix}$, where~$E_{i,i}$ is the~$(r \times r)$-matrix with every coefficient equal to~$0$ except the one on the~$i$-th row and~$i$-th column, and
\begin{equation}\label{eq_Ois_lag}
    \mathcal{O}_i = \Big{\{} \begin{pmatrix} 0_r & 0 \\ X & 0_r\end{pmatrix} \mid  X \in  \Mat_{r}(\Kf), \ ^t\! \overline{X} = X \text{ and } \operatorname{sgn}(X) = (i, r-i, 0)\Big{\}},
\end{equation}
where we have denoted by~$\operatorname{sgn}(X) = (p_+,p_-, r-p_+-p_-)$ the \emph{signature} of~$X$, where~$p_{\pm}$ is the maximal dimension of a totally positive (resp.\ negative) subspace of~$\mathbb{K}^r$ for~$X$. When~$i = r$, we recover the cone~$c^0$, and~$\Diams_\std = \varphi_\std(\mathcal{O}_r)$ identifies exactly the set of positive definite symmetric matrices. See \cite{kaneyuki1988sylvester, kaneyuki1998sylvester} for more details.

\item When~$\g = \soo(n,2)$ for some~$n \geq 2$, in the notation of Example~\ref{ex_shilov_bndrs}.(2), Example 2, and the identification~$\uu^- \simeq \Affstd_\std$ (given in~\eqref{eq_A_egal_exp}), we chose an orthonormal basis~$(e_1, \dots, e_n)$ of~$\Affstd$ for~$\psi$, so that~$\psi(e_i) =1$ for~$1 \leq i \leq n-1$ and~$\psi(e_n) = -1$. We write any element~$v \in \Affstd_\std$ as a linear combination~$v = \sum_{i=1}^n v_i e_i$; Then one has~$\uu^- = \mathcal{O}_0 \cup \mathcal{O}_1 \cup \mathcal{O}_2$, where:
\begin{equation*}
\begin{split}
   \Diams_\std' = \mathcal{O}_0 &= \{v \in \uu^- \mid \psi(v) <0 \text{ and } v_n <0\} ; \ \mathcal{O}_1 = \{v \in \uu^- \mid \psi(v) >0 \} ; \\
    \Diams_\std = \mathcal{O}_2 &= \{v \in \uu^- \mid \psi(v) <0 \text{ and } v_n > 0\}.\quad \qed
\end{split}
\end{equation*}
\end{enumerate}
\end{ex}

The goal of this section is to prove Lemma~\ref{lem_inclusion_Ois} below. To this end, we need the following lemma, which is a direct consequence of Fact~\ref{fact_takeuchi}:

\begin{lem}\label{lem_decomp_spectrale}
In the notation of Fact~\ref{fact_takeuchi}, one has~$\Ad(K_0) \cdot \operatorname{Span}(v_1, \dots, v_r) = \uu^\opp$. Moreover, for all~$X \in \uu^\opp$, writing~$X = \Ad(g) \cdot \sum_{k=1}^r \lambda_k v_k$ for some~$g \in K_0$ and~$\lambda_1, \dots, \lambda_r \in \mathbb{R}$, one has~$X \in V_{i_{+}, i_{-}}$, where~$i_{\pm} = \operatorname{card}\{1 \leq k \leq r \mid \pm\lambda_k > 0\}$. 
\end{lem}

\begin{proof} The first assertion is a direct consequence of Fact~\ref{fact_takeuchi} and the fact that~$\uu^\opp = \bigsqcup_{i+j \leq r} V_{i,j}$.

To prove the second assertion, just notice that if~$X = \Ad(g) \cdot \sum_{k=1}^r \lambda_k v_k$ for some~$g \in K_0$ and~$\lambda_1, \dots, \lambda_r \in \mathbb{R}$, then one can write~$X = \Ad(ga) \cdot (\sum_{k \in J^+} \varepsilon_k - \sum_{l \in J^-} \varepsilon_l)$ for 
    \begin{equation*}
        a = \exp\Big{(}\frac{1}{2}\sum_{k=1}^r \log(|\lambda_k|)h_{2 \varepsilon_k}\Big{)} \in \exp(\aaa)
    \end{equation*}
    and~$J^{\pm} = \{1 \leq k \leq r \mid \pm\lambda_k > 0\}$ (recall that the~$h_{\alpha}$, for~$\alpha \in \Sigma$, are defined in Section~\ref{sect_real_lie_alg}). Note that~$a \in L^0$, by~\eqref{eq_decomposition_l}. Since the semisimple part of~$L^0$ is of type~$A_{r-1}$, and a fundamental system of restricted roots is given by the roots~$\varepsilon_i - \varepsilon_{i+1}$, for~$1 \leq i \leq r-1$ (see Fact~\ref{eq_root_syst_L_s}), the Weyl group~$W_0$ of~$L^0$ acts by permutations on the~$(\varepsilon_i)_{1 \leq i \leq r}$. Hence there exists~$w \in W_0$ such that
    \begin{equation*}
        \sum_{k \in J^+} \varepsilon_k - \sum_{l \in J^-} \varepsilon_l = w \cdot X_{i_+,i_-},
    \end{equation*}
    with~$i_{\pm} = \operatorname{card} (J^{\pm})$. Since~$g, a , w \in L^0$, we have~$X \in \Ad(L^0) \cdot X_{i_+, i_-} = V_{i_+, i_-}$.~$\qed$
\end{proof}

The orbits~$\mathcal{O}_i$, for~$0 \leq i \leq r$, satisfy the following elementary properties:

\begin{lem}\label{lem_inclusion_Ois}

\begin{enumerate}
    \item \label{lem_inclusion_Ois_item_inclusion} For all~$i$, one has~$\mathcal{O}_i + c^0 \subset \bigcup_{j \geq i} \overline{\mathcal{O}}_j$, and~$\mathcal{O}_i - c^0 \subset \bigcup_{j \leq i} \overline{\mathcal{O}}_j$. 
    \item \label{lem_inclusion_Ois_item_interior} For all~$i$, the set~$\mathcal{O}_i$ is equal to the interior of its closure.
    \item \label{lem_inclusion_Ois_item_intersection} For~$1 \leq i \leq k \leq j \leq r$, one has
    \begin{equation*}
        \overline{\mathcal{O}}_i \cap \overline{\mathcal{O}}_j \subset \overline{\mathcal{O}}_k.
    \end{equation*}
    \item For all~$0 \leq i \leq r$, one has~$- \mathcal{O}_i = \mathcal{O}_{r-i}$.
\end{enumerate}
\end{lem}

\begin{proof} Points (2), (3) and (4) follow from Equation~\eqref{eq_adherence_O_i}. It remains to prove (1). It suffices to prove the first inclusion, the second one admitting an analogue proof.

Let~$K_0$ be the maximal compact subgroup of~$L^0$ defined in Fact~\ref{fact_takeuchi}. By Fact~\ref{fact_autom_diam}, the group~$K_0$ has a fixed point~$Z_0$ in~$c^0$. By Lemma~\ref{lem_decomp_spectrale}, there exist~$w \in K_0$ and~$\mu_1, \dots, \mu_r >0$ such that~$Z_0 = \Ad(w) \cdot \sum_{k=1}^r \mu_k v_k$. Since~$Z_0$ is~$\Ad(K_0)$-invariant, we actually have~$Z_0 = \Ad(w)^{-1} Z_0 = \sum_{k=1}^r \mu_k v_k$.

Let~$X \in \mathcal{O}_i$. Let us first prove that~$X + Z_{0} \in \overline{\mathcal{O}_j}$ for some~$j \geq i$. By Lemma~\ref{lem_decomp_spectrale}, there exist~$w' \in K_0$ and~$\lambda_1, \dots \lambda_r \in \mathbb{R}$ such that~$X = \Ad(w') \cdot \sum_{k=1}^r \lambda_k v_k$, with~$i = \operatorname{card}\{1 \leq k \leq r \mid \lambda_k > 0\}$ and~$r-i = \operatorname{card}\{1 \leq k \leq r \mid \lambda_k < 0\}$. To simplify the notations, we may assume that~$\lambda_1, \dots, \lambda_i > 0$ and~$\lambda_{i+1}, \dots , \lambda_r < 0$. Since~$K_0$ stabilizes~$Z_{0}$, we have
\begin{equation*}
    X + Z_{0} = \Ad(w') \cdot \sum_{i=k}^r (\lambda_k +1) v_i.
\end{equation*}
For all~$1 \leq k \leq i$, we have~$\lambda_k +1 >0$, so by Lemma~\ref{lem_decomp_spectrale}, we have~$X + Z_{0} \in V_{j, j'}$, for some~$j \geq i$, and some~$j'$ such that~$j+j' \leq r$. By~\eqref{eq_adherence_O_i}, we then have~$X + Z_{0} \in \overline{\mathcal{O}_j}$.

Now if~$Z \in c^0$, we can write~$Z = \Ad(\ell) \cdot Z_{0}$ for some~$\ell \in L^0$. Since~$\mathcal{O}_i$ is~$L^0$-stable, we have~$\Ad(\ell)^{-1} \cdot X \in \mathcal{O}_i$. We then apply the argument of the previous paragraph to~$\Ad(\ell)^{-1} \cdot X$, and get~$\Ad(\ell)^{-1} \cdot X + Z_{0} \in \overline{\mathcal{O}_j}$ for some~$j \geq i$. But~$\overline{\mathcal{O}_j}$ is~$L^0$-stable, so~$X+ Z = \Ad(\ell) \cdot (\Ad(\ell)^{-1} \cdot X +  Z_0) \in \overline{\mathcal{O}_j} $. This ends the proof of Point~(1).~$\qed$
\end{proof}

\begin{rmk}
Let us give an elementary proof of Lemma~\ref{lem_inclusion_Ois}.(1) in the case where~$\g = \spp(2r, \mathbb{R})$ for some~$r \geq 2$. According to the description of the $\mathcal{O}_i$'s given in Equation~\eqref{eq_Ois_lag}, it suffices to show that for any two symmetric matrices $X, Y$ with $Y$ positive-definite and $X$ of signature $(i, r - i, 0)$, the signature of $X + Y$ is $(j, k, r - j - k)$ with $j \geq i$. This follows directly from the definition of the signature: if $V$ is an $i$-dimensional real vector subspace of~$\Rf^{r}$ on which~$X$ is positive definite, then by the positivity of~$Y$, one has $^t\! \overline{\mathsf{v}} (X + Y) \mathsf{v} > 0$ for all~$\mathsf{v} \in V \smallsetminus \{ 0 \}$. Hence, $j \geq i$. The same reasoning on the signature allows one to prove Lemma~\ref{lem_inclusion_Ois}.(1) in the cases~$\g = \suu(r,r)$ and~$\g = \soo(4r)^*$, where~$r \geq 2$. However, the cases~$\g = \soo(n,2)$ with~$n \geq 2$ and~$\g = \eeee_{7(-25)}$ must be handled separately (even though the case~$\g = \soo(n,2)$ with~$n \geq 2$ is well known). The proof provided above has the advantage of encompassing all cases at once.
\end{rmk}

\subsection{A notion of causal convexity}\label{sect_causal_convexity_def_properties} In this section, we introduce a notion of convexity, called \emph{causal convexity}, inspired from causal convexity in conformal space-times (see \cite{Sanchez2008} and Remark~\ref{sect_causal_convexity_einstein_universe}). Contrary to the dual convexity recalled in Definition~\ref{def_dual_convexity}, causal convexity is specific to flag manifolds admitting a causal structure. We investigate the properties of causally convex domains and relate the two notions of convexity.

\begin{rmk}\label{sect_causal_convexity_einstein_universe} A domain~$\O$ of a \emph{conformal spacetime}~$(M,[g])$ is said to be \emph{causally convex} if every \emph{causal curve} of~$M$ joining two points of~$\O$ is contained in~$\O$; see e.g.\ \cite[pp 8]{Sanchez2008}. Any affine chart~$\Affstd$ of~$\Ein^{n-1, 1}$ can be identified with Minkowski space (recall Example~\ref{ex_shilov_bndrs}.(2)), and it is then easy to check that a domain~$\O\subset \Affstd$ is causally convex if and only if for every pair~$a,b\in \O$, the diamond~$\Diams_\Affstd(a,b)$ is contained in~$\O$. This observation is the inspiration of the notion of causal convexity we will introduce in Definition~\ref{def_causal_convexity} below.

\end{rmk}

\subsubsection{Causal convexity in an affine chart}\label{sect_def_causal_affine_chart} In this section, we generalize the classical notion of causal convexity in Minkowski space, recalled in Remark~\ref{sect_causal_convexity_einstein_universe}, to affine charts of general causal flag manifolds.

We fix an affine chart~$\Affstd$ of~$\SB(\g)$. We define:

\begin{definition}\label{def_causal_convexity} 
    We say that a subset~$X \subset \Affstd$ is \emph{causally convex in~$\Affstd$} if for all~$x,y \in X$ with~$y \in \mathbf{J}_{\Affstd}^+(x)$, the closed diamond~$\Diams_{\Affstd}^c(x,y)$ is contained in~$X$.
\end{definition}

Note that if~$\O \subset \Affstd$ is open, then it is causally convex if and only if for all~$x,y \in \O$ such that~$y \in \I^+_\Affstd(x)$, one has~$\Diams_\Affstd(x,y) \subset \O$. 

\begin{rmk} Neeb defines the same notion of causal convexity in affine charts, and poses several questions concerning this property in \cite[Sect.\ 9.1]{neeb2025open}.
\end{rmk}

The intersection of two causally convex sets is still causally convex. This leads to the following definition:

\begin{definition}\label{def_convex_hull_causal} Let~$X \subset \Affstd$. The \emph{causally convex hull}~$\operatorname{Conv}_{\Affstd}(X)$ \emph{of~$X$ in~$\Affstd$} is the smallest causally convex subset of~$\Affstd$ containing~$X$. 
Equivalently, the set~$\operatorname{Conv}_{\Affstd}(X)$ is the intersection of all  causally convex subsets of~$\Affstd$ containing~$X$.
\end{definition}

\begin{lem}\label{lem_convex_envelope_egal_plus_petit_convex}Let~$X \subset \Affstd$ be a subset. The causally convex hull of~$X$ in~$\Affstd$ is equal to the union of all  diamonds~$\Diams_{\Affstd}^c(x,y)$, for~$x,y \in X$ and~$y \in \J_{\Affstd}^+(x)$. In particular, it is connected whenever~$X$ is. If~$\O \subset \Affstd$ be an open subset, then the causally convex hull of~$\O$ in~$\Affstd$ is equal to the union of all diamonds~$\Diams_{\Affstd}(x,y)$, for~$x,y \in X$ and~$y \in \I_{\Affstd}^+(x)$.
\end{lem}

\begin{proof} We prove the first assertion, the second one admitting a similar proof. Let us define~$X' := \bigcup_{x,y \in X,\ y \in \J_{\Affstd}^+(x)} \Diams_{\Affstd}^c(x,y)$. By definition of the causally convex hull, we have~$X' \subset \operatorname{Conv}_{\Affstd}(X)$. Since~$X \subset X'$, to prove the converse inclusion, it suffices to prove that~$X'$ is  causally convex.  Let~$x, y \in X'$ be such that~$y \in \J_{\Affstd}^+(x)$. By definition of~$X'$, there exist~$x_1, x_2 \in X$ such that~$x_2 \in \J_{\Affstd}^+(x_1)$ and~$x \in \Diams_{\Affstd}^c(x_1, x_2)$, and~$y_1, y_2 \in \Affstd$ such that~$y_2 \in \J_{\Affstd}^+(y_1)$ and~$y \in \Diams_{\Affstd}^c(y_1, y_2)$. Then by transitivity, we get that~$y_2 \in \J_{\Affstd}^+(x_1)$. Since~$x_1, y_2 \in X$, by definition of~$X'$ we have~$\Diams_{\Affstd}^c(x,y) \subset \Diams_{\Affstd}^c(x_1, y_2) \subset X'$. Hence~$X'$ is causally convex.~$\qed$
\end{proof}

\begin{rmk}\label{rmk_convex_envelop_projective}  Lemma~\ref{lem_convex_envelope_egal_plus_petit_convex} states a property of convexity studied in this section --- causal convexity --- that significantly distinguishes it from classical convexity in the real projective setting. Indeed, in the latter case, the convex hull of a set~$F$ is in general not equal to the union of the projective segments connecting two points of~$F$; instead, every element of~$F$ is a convex combination of at most~$n$ points, where~$F \subset \mathbb{P}(\mathbb{R}^n)$. In the case of causal convexity in causal flag manifolds, Lemma~\ref{lem_convex_envelope_egal_plus_petit_convex} is a consequence of the intrinsic causality of our definition of convexity, and will be crucial in the proof of the implication~$(1) \Rightarrow (2)$ of Theorem~\ref{thm_equiv_anosov}.
\end{rmk}

\subsubsection{Link with dual convexity}\label{sect_link_dual_convexity} The goal of this section is to relate causal convexity and dual convexity, see Proposition~\ref{prop_dual_implies_causal} and Remark~\ref{rmk_causal_weaker_dual} below. To this end, we need the following lemma, which answers in particular a question of Neeb \cite[Problem 9.8]{neeb2025open}:

\begin{lem} \label{lem_comp_causally_convex}
Every connected component of~$\Affstd_{\mathsf{std}} \smallsetminus \hyp_{\LP^+} = \SB(\g) \smallsetminus (\hyp_{\LP^+} \cup \hyp_{\LP^{\opp}})$ is causally convex in~$\Affstd_{\mathsf{std}}$.
\end{lem}

\begin{proof} Le~$\mathcal{U}$ be a connected component of~$\Affstd_\std \smallsetminus \hyp_{\LP^+}$. In the notation of Section~\ref{sect_def_Ois}, we know that~$\mathcal{U} = \varphi_\std(\mathcal{O}_i)$ for some~$0 \leq i \leq r$. Let~$x,y \in \mathcal{U}$ be such that~$y \in \I^+(x)$, and let~$X,Y \in \mathcal{O}_i$ be such that~$x = \varphi_{\std}(X)$,~$y = \varphi_{\std}(Y)$. Then, by Points~(1) and~(3) of Lemma~\ref{lem_inclusion_Ois}, one has
\begin{equation*}
    \Diams(x,y) = \varphi_{\std}\big{(}(X + c^0) \cap (Y - c^0)\big{)} \subset \varphi_{\std} \Big{(} \big{(}\bigcup_{j \geq i} \overline{\mathcal{O}}_j \big{)} \cap \big{(}\bigcup_{j \leq i} \overline{\mathcal{O}}_j \big{)} \Big{)} \subset \overline{\mathcal{U}}.
\end{equation*}
Since~$\Diams(x,y)$ is open, by Lemma~\ref{lem_inclusion_Ois}.(2), we have~$\Diams(x,y) \subset \mathcal{U}$.~$\qed$
\end{proof}

We can now prove:

\begin{prop}\label{prop_dual_implies_causal} Let~$\Affstd$ be an affine chart. Any dually convex domain contained in~$\Affstd$ is causally convex in~$\Affstd$.
\end{prop}
\begin{proof} Since~$G$ acts transitively on~$\SB(\g)$, we may assume that~$\Affstd = \Affstd_{\mathsf{std}}$. Let~$\O$ be a dually convex domain of~$\SB(\g)$ which is contained in~$\Affstd_{\mathsf{std}}$ and let~$x,y \in \O$, with~$y \in \I^{+}(x)$. Assume that there exists~$a \in \Diams(x,y) \cap \partial \O$. By dual convexity, there exists~$z \in \SB(\g)$ such that~$\O \cap \hyp_z = \emptyset$ and~$a \in \hyp_z$. Since~$U^-$ acts transitively on~$\Affstd_\std$, we may assume that~$z = \LP^+$. By connectedness, the domain~$\O$ is then contained in one of the connected components of~$\Affstd_\std \smallsetminus \hyp_{\LP^+}$, let us denote it by~$\mathcal{O}$. In particular, we have~$x,y \in \mathcal{O}$. Then, by Lemma~\ref{lem_comp_causally_convex}, we have~$a \in \mathcal{O}$, which contradicts the fact that~$a \in \hyp_{\LP^+}$. Thus we must have~$\Diams(x,y) \subset \O$.~$\qed$ 
\end{proof}

\begin{rmk}\label{rmk_causal_weaker_dual} The implication of Proposition~\ref{prop_dual_implies_causal} is not an equivalence, as there exist causally convex domains which are not dually convex. For instance, take a diamond~$D$ bounded in an affine chart~$\Affstd$, and consider~$\O := D \smallsetminus \J_{\Affstd}^+(x)$ for some point~$x \in D$. Then~$\O$ is causally convex but not dually convex.
\end{rmk}

\subsubsection{Causal convexity in~$\SB(\g)$}\label{sect_causal_convexity_def} In this section, our goal is to prove that causal convexity is an intrinsic notion in~$\SB(\g)$, namely, whenever a subset~$X \subset \SB(\g)$ is connected, the property for~$X$ of being causally convex does not depend on the choice of an affine chart containing~$X$. This statement is contained in Proposition~\ref{lem_convex_hul_indép_carte_aff} below. We will need the following lemma:
\begin{lem}\label{lem_convex_hul_indép_auxiliaire}
    Let~$\O \subset \SB(\g)$ be a domain contained in an affine chart. For any subset~$X \subset \O$, the causally convex hull of~$X$ in~$\Affstd$ does not depend on the affine chart~$\Affstd$ containing~$\O$. We denote it by~$\operatorname{Conv}^{\O}(X)$.
\end{lem}

\begin{proof}
     Let~$\Affstd, \Affstd'$ be two affine charts containing~$\O$. There exists~$g \in G^0$ such that~$g \cdot \Affstd' = \Affstd~$. For all~$x \in \O$, one has~$g \cdot x \in g \cdot \Affstd' = \Affstd$. 

Let~$C_\Affstd(X)$ be the convex hull of~$X$ in~$\Affstd$, and~$C_{\Affstd'}(X)$ its convex hull in~$\Affstd'$. Recall the domain~$\O_0^{**}$ introduced in Definition~\ref{def_enveloppe_dc}. Since~$\O \subset \Affstd \cap \Affstd'$, we have~$\O_0^{**} \subset \Affstd \cap \Affstd'$.

Let us prove that~$C_{\Affstd}(X) \subset C_{\Affstd'}(X)$, the converse inclusion then holding by symmetry. 

Let~$x,y \in X$ be such that~$y \in \J_{\Affstd}^+(x)$, and~$D := \Diams_{\Affstd}^c(x,y) \subset C_{\Affstd}(X)$. Let~$(x_k), (y_k) \in \O^{\mathbb{N}}$ be such that~$x_k \rightarrow x$,~$y_k \rightarrow y$, and~$y_k \in \I_{\Affstd}^+(x_k)$ for all~$k$. For all~$k \in \mathbb{N}$, by causal convexity of~$\O_0^{**} \subset \Affstd'$ (Proposition~\ref{prop_dual_implies_causal}), we have~$\overline{\Diams_{\Affstd}(x_k,y_k)} \subset \overline{\O_0^{**}} \subset \Affstd'$. Then by Lemma~\ref{lem_diamond_bounded}, we have~$x_k \in \I^{\pm}_{\Affstd'}(y_k)$, let's say~$x_k \in \I^{-}_{\Affstd'}(y_k)$ for instance. Thus~$\Diams_{\Affstd}(x_k,y_k) = \Diams_{\Affstd'}(x_k,y_k)$. This is true for all~$k \in \mathbb{N}$, so~$\Diams_{\Affstd}^c(x,y) = \Diams_{\Affstd'}^c (x,y) \subset C_{\Affstd'}(X)$. By Lemma~\ref{lem_convex_envelope_egal_plus_petit_convex}, this gives~$C_{\Affstd}(X) \subset C_{\Affstd'}(X)$.~$\qed$
\end{proof}

Lemma~\ref{lem_convex_hul_indép_auxiliaire} allows us to prove:

\begin{prop}\label{lem_convex_hul_indép_carte_aff}
Let~$X \subset \SB(\g)$ be a connected subset contained in two affine charts~$\Affstd, \Affstd'$. Then we have 
    \begin{equation*}\operatorname{Conv}_{\Affstd'}(X) = \operatorname{Conv}_{\Affstd}(X).
    \end{equation*}
\end{prop}

\begin{proof} Since~$X$ is connected and contained in the open set~$\Affstd \cap \Affstd'$, there exists a connected open neighborhood~$\O$ of~$X$ contained in both~$\Affstd$ and~$\Affstd'$. By Lemma~\ref{lem_convex_hul_indép_auxiliaire}, the set~$\operatorname{Conv}^{\O }(X)$ is then equal to both~$\operatorname{Conv}_{\Affstd} (X)$ and~$\operatorname{Conv}_{\Affstd'} (X)$.~$\qed$
\end{proof}

Proposition~\ref{lem_convex_hul_indép_carte_aff} implies in particular that the convex hull of a connected subset~$X \subset \SB(\g)$ contained in an affine chart does not depend on the affine chart containing it:

\begin{definition}\label{def_causally_convex_indep}
    Let~$X \subset \SB(\g)$ be a connected subset contained in an affine chart. The \emph{causally convex hull}~$\operatorname{Conv}(X)$ is by definition the causally convex hull of~$X$ in any affine chart containing it. We say that~$X$ is \emph{causally convex} if it is equal to its causally convex hull.
\end{definition}

We can now prove Proposition~\ref{prop_dual_implique_causall}:

\begin{proof}[Proof of Proposition~\ref{prop_dual_implique_causall}] Since~$\O \ne \SB(\g)$, there exists~$a \in \partial \O$. By dual convexity, there exists~$z \in \SB(\g)$ such that~$a \in \hyp_z$ and~$\hyp_z \cap \ \O = \emptyset$. Thus~$\O$ is contained in an affine chart, namely~$\Affstd_z$. By Proposition~\ref{prop_dual_implies_causal}, the domain~$\O$ is causally convex in~$\Affstd_z$ and thus causally convex.~$\qed$
\end{proof}

Another straightforward but useful corollary of Proposition~\ref{lem_convex_hul_indép_carte_aff} is the following:

\begin{cor}\label{cor_consal_invariant_by_autom_group}
    Let~$X \subset \SB(\g)$ be a connected subset, contained in an affine chart. Then the causally convex hull of~$X$ is~$\Aut(X)$-invariant.
\end{cor}

\section{Topological restrictions}\label{sect_topo_restrictions}
By \cite{benoist2000automorphismes}, a necessary and sufficient condition for an irreducible subgroup of \( \PGL(n, \mathbb{R}) \) to preserve a properly convex open subset of \( \mathbb{P}(\mathbb{R}^n) \) is that, up to finite index, the group \( \Gamma \) contains proximal elements (in \( \mathbb{P}(\mathbb{R}^n) \)), and that all of them have a real positive highest (in modulus) eigenvalue. The goal of this section is to determine a necessary condition for a subgroup of a Lie group \( G \) to preserve a proper domain in a self-opposite flag manifolds of \( G \) (see Proposition~\ref{prop_restrictions_maslov_index_2} and its Corollary~\ref{cor_restrictions_maslov_index}). This condition is related to the \emph{relative positions} of three pairwise-transverse points of their limit set in the flag manifold. We first give reminders on this notion and its properties, in Section~\ref{sect_stable_connected components} below.

\subsection{Stable connected components}\label{sect_stable_connected components} 
Let~$\g$ be a semisimple Lie algebra of noncompact type and~$\Theta \subset \Delta$ be a self-opposite subset of the simple restricted roots of~$\g$. Following \cite{dey2024restrictions}, consider the involution
\begin{equation}\label{eq_def_opposition}
    \inver_{\Theta}:
        \Affstd_{\std} \longrightarrow \Affstd_\std; \ \
        \varphi_\std(X) \longmapsto \varphi_\std (-X).
\end{equation}
The map~$\inver_{\Theta}$ induces a homeomorphism of~$\Affstd_\std  \smallsetminus \hyp_{\LP_{\Theta}^+}$ \cite{dey2024restrictions}. We denote by~$\mathcal{E}_{\Theta}$ the set of connected components of~$\Fl(\g, \Theta) \smallsetminus (\hyp_{\LP_{\Theta}^+} \cup \hyp_{\LP_\Theta^\opp})$. The map~$\inver_\Theta$ induces a permutation of~$\mathcal{E}_\Theta$.

 \begin{rmk}\label{rmk_prop_I_anosov}
     Dey--Greenberg--Riestenberg proved \cite{dey2024restrictions} that if \( \mathcal{E}_\Theta \) has no~$\inver_\Theta$-invariant element, then any hyperbolic group \( \Gamma \) admitting a \( \Theta \)-Anosov representation \( \Gamma \rightarrow \Aut_\Theta(\g) \) is either virtually free or virtually a surface group. The question of what flag manifolds~$\Fl(\g, \Theta)$ satisfy that~$\mathcal{E}_\Theta$ has~$\inver_\Theta$-invariant elements is thus deeply related to questions on~$\Theta$-Anosov representations. It has been investigated by several authors:
     \begin{enumerate}
\item If~$\g = \mathfrak{sp}(2n, \mathbb{R})$, taking the root system given in Section~\ref{sect_root_syst_shilov}, if~$\Theta$ contains a simple root~$\alpha_i$ for an odd~$1 \leq i\leq n$, then~$\mathcal{E}_\Theta$ has no~$\inver_\Theta$-element \cite{dey2024restrictions}.
    
\item If~$\g = \sll(d, \mathbb{R})$ with~$d \ne 5$ and~$d = 2,3,4,5 \mod 8$ and~$\Theta = \Delta$, then~$\mathcal{E}_\Theta$ has no~$\inver_\Theta$-invariant element \cite{dey2022borel}.
     
\item Kineider--Troubat have classified all the elements of~$\mathcal{E}_\Theta$, where~$\g = \soo(p,q)$ and~$\Theta$ is any subset of the simple restricted roots of~$\g$, and investigated which ones are~$\inver_{\Theta}$-invariant, see \cite{kineider2024connected}.~$\qed$
\end{enumerate}
 \end{rmk}

From now on, we will use the following terminology:
\begin{definition}
    Let~$F \subset \Fl(\g, \Theta)$ be a subset. We denote by~$F^{3, * }$ the set of triples of pairwise transverse points in~$F$.
\end{definition}
Note that if~$F$ is any subset of~$\Fl(\g, \Theta)$, then the set~$F^{3, *}$ might be empty. 

Until the end of this section, we fix~$G := \Aut_\Theta(\g)$. The subgroup~$L_\Theta$ of~$G$ acts on~$\mathcal{E}_{\Theta}$. If~$(x,y,z) \in \Fl(\g, \Theta)$ are three pairwise transverse points, then there exists a unique~$[g] \in G/L_\Theta$ such that~$g \cdot (x,y) = (\LP_{\Theta}^+, \LP_\Theta^\opp)$. We denote by~$\typ(x,y,z)$ the~$L_\Theta$-orbit of the connected component~$\mathcal{O}$ of~$\Fl(\g, \Theta) \smallsetminus (\hyp_{\LP_{\Theta}^+} \cup \hyp_{\LP_\Theta^\opp})$ containing~$g\cdot y$. 

\begin{definition} The \emph{type} of a triple~$(x,y,z) \in \Fl(\g, \Theta)^{3,*}$ is by definition the orbit~$\typ(x,y,z) \in \mathcal{E}_\Theta/L_\Theta$.
\end{definition}
The type is~$G$-invariant, and is a generalization of the well-known \emph{Maslov index} for~$\HTT$ Lie groups, which we will recall in Section~\ref{sect_maslov}. Since~$\inver_\Theta$ commutes with the action of~$L_\Theta$ on~$\uu_\Theta^\opp$, it induces a bijection of~$\mathcal{E}_\Theta / L_\Theta$, still denoted by~$\inver_\Theta$. We then have
\begin{equation}\label{eq_relation_type}
    \typ(x,y,z) = \inver_\Theta(\typ(x,z,y)) \quad \forall (x,y,z) \in \Fl(\g, \Theta)^{3,*}.
\end{equation}
Note that~$\inver_\Theta$ induces a bijection between~$\pi_0(\Fl(\g, \Theta)^{3,*} /G)$ and~$\mathcal{E}_\Theta / L_\Theta$.

\begin{rmk}\label{rmk_typ_nagano}
    If there exists~$\ell \in L_\Theta$ such that~$\ell|_{{\Affstd}_\std} = \inver_\Theta$, then Equation~\eqref{eq_relation_type} becomes:
\begin{equation*}
    \typ(x,y,z) = \typ(x,z,y) \quad \forall (x,y,z) \in \Fl(\g, \Theta)^{3,*}.
\end{equation*}
This is the case for instance if the flag manifold~$\Fl(\g, \Theta)$ is a causal flag manifold. More generally, it is the case if and only if~$\Fl(\g, \Theta)$ is a \emph{Nagano space} in the sense of \cite{galiay2025convex}.
\end{rmk}

\subsection{The general case}

In this section, we first study the restrictions imposed on a group by the property of preserving a proper domain, on the type of triples of the~$\Theta$-limit set (these restrictions are thus topological
). Recall the notions introduced in Section~\ref{sect_divergent_groups}.

\begin{prop}[see Theorem~\ref{prop_restrictions_maslov_index}]\label{prop_restrictions_maslov_index_2} Let~$G$ be a noncompact real semisimple Lie group, and~$\Theta$ be a self-opposite subset of the simple restricted roots of~$G$. Let~$H \leq G$. If one of the two following conditions is satisfied, then there exists an~$\inver_\Theta$-invariant element~$\mathcal{O} \in \mathcal{E}_\Theta$, and for any triple~$(a,b,c) \in \Lambda_\Theta (H)^{3, *}$, we have~$\typ(a,b,c) = [\mathcal{O}] \in \mathcal{E}_\Theta/L_\Theta$:

\begin{enumerate}
    \item The~$\Theta$-limit set~$\Lambda_{\Theta}(H)$ contains at least four pairwise transverse points and~$H$ preserves a (not necessarily proper) domain~$\O \subset \Fl(\g, \Theta)$ such that \begin{equation}\label{eq_condition_limit_set}
        \hyp_p \cap \ \O = \emptyset \quad \forall p \in \Lambda_{\Theta}(H).
    \end{equation}
 
    \item The~$\Theta$-limit set~$\Lambda_\Theta (H)$ contains at least three pairwise transverse points and~$H$ preserves a proper domain in~$\Fl(\g, \Theta)$.
\end{enumerate}

In particular, in this case, the set~$\mathcal{E}_\Theta$ admits an~$\inver_\Theta$-invariant connected component.
\end{prop}

\begin{proof}
(1) \emph{Let us assume that (1) is satisfied.} First note that~$a \in \overline{\O}$ for all~$a \in \Lambda_\Theta( H)$: indeed, let~$a \in \Lambda_\Theta(H)$, and let~$(h_n) \in H^{\mathbb{N}}$ and~$b \in \Fl(\g, \Theta)$ such that~$(h_n)$ is~$\Theta$-contracting with respect to~$(a,b)$. Since~$\O$ is open, there exists~$x \in \O \smallsetminus \hyp_b$. Thus~$h_n \cdot x \rightarrow a$. Since~$\O$ is~$H$-invariant, we have~$a \in \overline{\O}$.  
    
Since~$G$ acts transitively on pairs of transverse points of~$\Fl(\g, \Theta)^2$, we may assume that~$\LP_{\Theta}^+, \LP_\Theta^\opp \in \Lambda_\Theta(H)$ and that there exist two transverse points~$x,y \in \Lambda_\Theta(H)\smallsetminus (\hyp_{\LP_{\Theta}^+} \cup \hyp_{\LP_\Theta^\opp})$. By Equation~\eqref{eq_condition_limit_set} and connectedness of~$\O$, there exists~$\mathcal{O} \in \mathcal{E}_\Theta$ such that~$\O \subset \mathcal{O}$. Since~$x,y \in \overline{\O}$, we have~$x,y \in\overline{\mathcal{O}}$. By transversality, we then have~$x,y \in \mathcal{O}$. Let~$X,Y \in \uu^-_\Theta$ be such that~$(x,y) = (\varphi_\std(X), \varphi_\std(Y))$. 

Let~$Z \in \{X, Y\}$. Since~$\varphi_\std(Z) \in \Lambda_\Theta(H)$, by Equation~\eqref{eq_condition_limit_set} and connectedness of~$\O$, there exists an element~$\mathcal{O}' \in \mathcal{E}_\Theta$ such that~$\O \subset \exp(Z) \cdot \mathcal{O}'$. Then~$\LP_{\Theta}^+ \in \overline{\O} \subset \overline{\mathcal{O}}'$, and since~$\LP_{\Theta}^+$ is transverse to~$\varphi_\std(Z)$, we have~$\LP_{\Theta}^+ \in \exp(Z) \cdot \mathcal{O}'$. Thus~$\inver_{\Theta}(\varphi_\std(Z)) = \exp(-Z) \cdot \LP_{\Theta}^+ \in \mathcal{O}'$, so~$\varphi(Z) \in \inver_{\Theta}(\mathcal{O}')$. Since~$\varphi_\std(Z) \in \mathcal{O}$ and~$\mathcal{O}, \inver_{\Theta}(\mathcal{O}')$ are two connected components of~$\Fl(\g, \Theta) \smallsetminus (\hyp_{\LP_{\Theta}^+} \cup \hyp_{\LP_\Theta^\opp})$, we must have~$\mathcal{O}' = \inver_{\Theta}(\mathcal{O}).$

We have proven that~$\O \subset \exp(Z) \inver_{\Theta}(\mathcal{O})$. Since~$x \in \overline{\O}$ and~$x$ is transverse to~$y$, for~$Z = Y$ we have 
\begin{equation}\label{eq_equality_connected_components}
    x \in \exp(Y) \inver_{\Theta}(\mathcal{O}).
\end{equation} 
Similarly, taking~$Z = X$, one gets~$y \in \exp(X) \inver_{\Theta}(\mathcal{O})$. Thus we have 
\begin{equation*}
     \exp(-Y)\exp(X) \cdot \LP_{\Theta}^+ \in \inver_{\Theta}(\mathcal{O})
\end{equation*}
and
\begin{equation*}
     \exp(-X)\exp(Y) \cdot \LP_{\Theta}^+ = \inver_{\Theta}\big{(}\exp(-Y)\exp(X) \cdot \LP_{\Theta}^+ \big{)}  \in \inver_{\Theta}(\mathcal{O}).
\end{equation*}
Thus~$\inver_{\Theta}(\mathcal{O}) = \inver_{\Theta}(\inver_{\Theta}(\mathcal{O})) = \mathcal{O}$. Moreover~\eqref{eq_equality_connected_components} implies that~$[\mathcal{O}] = \typ(x,y, \LP_{\Theta}^+)$. Now it is readily checked that for all~$z \in \Lambda_\Theta(\Gamma) \smallsetminus (\hyp_x \cup \hyp_y)$, we have~$\mathcal{O} \in \typ(x,y,z)$.

(2) \emph{Let us assume that (2) is satisfied. } By Lemma~\ref{lem_proper_inter_hyp}, Equation~\eqref{eq_condition_limit_set} is satisfied.

The rest of the proof is similar to that of point (1). Since~$G$ acts transitively on pairs of transverse points of~$\Fl(\g, \Theta)^2$, we may assume that~$\overline{\O} \subset \Affstd_\std$, that~$\LP_{\Theta}^+ \in \Lambda_\Theta(H)$, and that there exist two transverse points~$x,y \in \Lambda_\Theta(H) \smallsetminus \hyp_{\LP_{\Theta}^+}$.
 By Equation~\eqref{eq_condition_limit_set}, there is an element~$\mathcal{O} \in \mathcal{E}_\Theta$ such that~$\O \subset \mathcal{O}$. Then~$x,y \in \overline{\O} \subset \Affstd_\std \cap \overline{\mathcal{O}}$, and since~$x,y$ are transverse to~$\LP_{\Theta}^+$, we have~$x,y \in \mathcal{O}$. As in point (1), we prove that~$\mathcal{O} = \inver_{\Theta}(\mathcal{O})$, and~$\mathcal{O} \in \typ(x,y,\LP_{\Theta}^+)$. Now if~$z \in \Lambda_\Theta(H) \smallsetminus \hyp_{\LP_{\Theta}^+}$, then it is readily checked that~$[\mathcal{O}] = \typ(x,y,z)$.~$\qed$
\end{proof}

\subsection{Maslov index}\label{sect_maslov}
In the case where~$\g$ is a~$\HTT$ Lie algebra of rank~$r \geq 1$ and~$\Theta = \{\alpha_r\}$, the relative positions of three transverse points of~$\Fl(\g, \Theta) = \SB(\g)$ are described by the classical \emph{Maslov index}. In Corollary~\ref{cor_restrictions_maslov_index} below, we reformulate Proposition~\ref{prop_restrictions_maslov_index_2} in this particular context.

We take Notation~\ref{sect_notation} and the one of Section~\ref{sect_def_Ois}; recall in particular the definition of the domains~$\mathcal{O}_i$ for~$1 \leq i \leq r$. By Lemma~\ref{lem_inclusion_Ois}.(4), and since~$\mathfrak{l}$ contains an element~$H_0$ such that~$\ad(H_0)v = -v$ for all~$v \in \uu^-$ (see e.g.\ \cite[Sect.\ 3.2]{galiay2024rigidity}), the possible values of the type of a triple of pairwise transverse points~$(x,y,z) \in \SB(\g)$ are:

\begin{equation}\label{eq_typ_shilov}
    \typ(x,y,z) = \big{\{} \varphi_\std(\mathcal{O}_i), \varphi_\std(\mathcal{O}_{r-i})\big{\}} \ \text{ for some } 0 \leq i \leq \frac{r}{2}.
\end{equation}
The \emph{Maslov index} of~$(x,y,z)$, denoted by~$\indx(a,b,c)$, is then the well-defined integer~$|r- 2i|$, where~$\typ(x,y,z) = \{ \varphi_\std(\mathcal{O}_i), \varphi_\std(\mathcal{O}_{r-i})\}$. With these notations, Proposition~\ref{prop_restrictions_maslov_index_2} gives:

\begin{cor}[see Theorem~\ref{prop_restrictions_maslov_index}]\label{cor_restrictions_maslov_index}
    Let~$G$ be a~$\HTT$ Lie group of rank~$r \geq 1$ and let~$H \leq G$ be subgroup preserving a domain~$\O \subset \SB(\g)$. If one of the following conditions in satisfied, then~$r$ is even and~$\indx(a,b,c) = 0$ for any triple~$(a,b,c) \in \Lambda_{\{\alpha_r\}}(H)^{3, *}$:

    \begin{enumerate}
        \item If~$\Lambda_{\{\alpha_r\}}(H)$ contains at least~$4$ transverse points and~$\hyp_{\LP^+} \cap \O = \emptyset$ for all~$p \in \Lambda_{\{\alpha_r\}}(H)$;
        \item If~$\Lambda_{\{\alpha_r\}}(H)$ contains at least three transverse points and~$\O$ is proper.
    \end{enumerate}
\end{cor}

\begin{proof} By Proposition~\ref{prop_restrictions_maslov_index_2}, we know that Points (1) and (2) both imply that there exists an~$\inver_{\{\alpha_r\}}$-invariant connected component~$\mathcal{O}$ of~$\Affstd_\std \smallsetminus \hyp_{\LP^+}$ such that~$\mathcal{O} \in \typ(x,y,z)$ for any triple of pairwise transverse points~$(x,y,z) \in \Lambda_{\{\alpha_r\}}(H)$. But by Lemma~\ref{lem_inclusion_Ois}.(4), the only connected component of~$\Affstd_{\std} \smallsetminus \hyp_{\LP^+}$ which is~$\inver_{\{\alpha_r\}}$-invariant is~$\varphi_\std(\mathcal{O}_{r/2})$, so~$\mathcal{O} = \varphi_\std(\mathcal{O}_{r/2})$. In particular~$r$ is even. Moreover, by Equation~\eqref{eq_typ_shilov}, the Maslov index of any triple of distinct points~$(x,y,z) \in \Lambda_{\{\alpha_r\}}(H)$ is equal to~$0$.~$\qed$
\end{proof}

\begin{rmk} Let~$\Gamma$ be the fundamental group of a closed surface, and let~$\rho: \Gamma \rightarrow G$ be an~$\{\alpha_r\}$-Anosov representation. By
Theorem~\ref{prop_restrictions_maslov_index}, if the group~$\rho(\Gamma)$  preserves a proper domain, then the Maslov index of~$\rho$ is equal to~$0$, and in particular~$r$ is even. This implies in particular that \emph{maximal} representations of~$\Gamma$ into~$G$, i.e.\ those of Maslov index~$r$, never preserve a proper domain in~$\SB(\g)$. On the contrary, in the context of groups preserving proper domains, we are interested in~$\{\alpha_r\}$-Anosov representations~$\rho: \Gamma \rightarrow G$ which are ``as far as possible" from being maximal.
\end{rmk}

\section{Groups acting cocompactly on a closed subset}\label{sect_groups_acting_cocompactly}

In this section, we investigate the properties of groups preserving proper domains~$\O$ in flag manifolds and acting cocompactly on a closed subset of~$\O$. The main results of this section are Proposition~\ref{prop_finitely_generated} and Lemma~\ref{lem_p_conic_shilov}.

\subsection{Reminders on projective convex cocompactness}\label{sect_reminders_projective_convex_cocompactness}
In this section, we recall the definition and properties of \emph{projective convex cocompact} subgroups defined in \cite{DGKproj, zimmer2021projective}. These serve as motivation for the present section and will not be used in the rest of the paper.

Once a subgroup of \( \PGL(n, \mathbb{R}) \) preserves a properly convex domain of real projective space, it is sometimes possible to read the dynamical properties of \( \Gamma \) in the geometry of an open set and the convex subsets it preserves: for instance, by~\cite{DGKproj}, the group \( \Gamma \) is \( P_1 \)-Anosov (in the sense of Section~\ref{sect_anosov}) if and only if it is \emph{strongly convex-cocompact} in \( \mathbb{P}(\mathbb{R}^n) \) in the following sense:
\begin{definition}\cite{DGKproj}\label{def_convex_cocompact}
    A discrete subgroup \( \Gamma \leq \PGL(n, \mathbb{R}) \) is \emph{strongly convex-cocompact} if it preserves a properly convex domain \( \O' \subset \mathbb{P}(\mathbb{R}^n) \) and acts cocompactly on a closed (in~$\O'$) convex subset \( \mathcal{C} \subset \O' \) whose ideal boundary \( \partial_i \mathcal{C} := \overline{\mathcal{C}} \cap \partial \O \) contains the full orbital limit set of~$\Gamma$ in~$\O$ (in the sense of Section~\ref{sect_group_aut_omega}) and does not contain any nontrivial projective segment.
\end{definition}
If \( G \) is a real reductive linear Lie group and \( \Theta \) is a non-empty subset of the simple restricted roots of \( G \), then we can consider an irreducible representation~$(V, \delta)$ of~$G$ given by Fact~\ref{prop_ggkw}, and the associated~\( \delta \)-equivariant embedding \( \iota_\delta : \Fl(\g, \Theta) \rightarrow \mathbb{P}(V) \). 

Now let \( \Gamma \) be a group and \( \rho: \Gamma \rightarrow G \) a representation. If \( \rho(\Gamma) \) already preserved a proper domain \( \O \subset \Fl(\g, \Theta) \), then it is easy to see that \( \delta \circ \rho(\Gamma) \) will preserve a proper domain in \( \mathbb{P}(V) \) (typically the interior of the convex hull of \( \iota_\delta(\O) \)). Furthermore, the representation \( \rho \) is \( \Theta \)-Anosov if and only if \( \delta \circ \rho \) is \( P_1 \)-Anosov \cite{gueritaud2017anosov}. This gives:
\begin{fact}[\cite{gueritaud2017anosov, DGKproj}]
    Let \( \rho: \Gamma \rightarrow G \) be a representation preserving a proper domain \( \O \subset \Fl(\g, \Theta) \). The representation \( \rho \) is \( \Theta \)-Anosov if and only if the image \( \delta \circ \rho(\Gamma) \) is strongly convex-cocompact in \( \mathbb{P}(V) \).
\end{fact}
This provides a geometric characterization of \( \Theta \)-Anosov subgroups of \( G \) that preserve proper domains in \( \Fl(\g, \Theta) \). However, one can wonder if an intrinsic geometric characterization within the flag manifold \( \Fl(\g, \Theta) \) is possible. A lot of properties of projective convex-cocompact groups relies on the ``convexity" part of convex-cocompactness in Definition~\ref{def_convex_cocompact}. The properties of projective convexity that are used here are not shared by dual convexity in general flag manifolds (see e.g.\ \cite[Chap.\ 3]{galiay2025convex}). A first approach to studying groups that might be considered ``convex-cocompact'' in a flag manifold~$\Fl(\g, \Theta)$ is to weaken the convexity requirement and instead focus on the properties of groups that preserve a proper domain~$\O \subset \Fl(\g, \Theta)$ and act cocompactly on a closed subset of~$\O$ which is \emph{not necessarily convex}. This is the purpose of the two following subsections.

\subsection{Finite generation}
In this section, we prove Proposition~\ref{prop_finitely_generated} below, whose main consequence in this paper will be the implication ``$(2) \Rightarrow \Gamma$ is finitely generated" of Theorem~\ref{thm_equiv_anosov}:

\begin{prop}\label{prop_finitely_generated}
Let~$G$ be a noncompact real semisimple Lie group and~$\Theta \subset \Delta$ a subset of the simple restricted roots. Let~$\Gamma \leq G$ be a discrete subgroup preserving a proper domain in~$\Fl(\g, \Theta)$. Assume that there exists a closed~$\Gamma$-invariant connected subset~$F$ of~$\O$, and a compact subset~$\compact$ of~$F$ such that~$F = \Gamma \cdot \compact$. Then~$\Gamma$ is finitely generated.
\end{prop}

We will need the following elementary lemma (recall Definition~\ref{def_enveloppe_dc}):

\begin{lem}\label{lem_replace_by_bidual}
    Let~$H$ be a subgroup of~$G$, and assume that there exist a proper~$H$-invariant domain~$\O$ and a closed~$H$-invariant subset~$F$ of~$\O$, such that~$H$ acts cocompactly on~$F$. Then~$F$ is still closed in~$\O_0^{**}$.
\end{lem}

\begin{proof} Note that~$\O_{0}^{**}$ is~$H$-stable. This inclusion~$\partial F \cap \partial \Od \subset  \partial F \cap \partial \O$ is due to the fact that~$\partial \Od \subset \partial \O$. Now let~$a \in \partial \O \cap \partial F$. Then  there exists~$(x_k) \in \O^\mathbb{N}$ converging in~$\O$ (and thus in~$\Od$), and~$(g_k) \in H^\mathbb{N}$ such that~$g_k \cdot x_k \rightarrow a$. Since~$a \notin \O$, the sequence~$(g_k)$ diverges in~$G$. Since~$H$ acts properly on~$\Od$ and~$(x_k)$ converges in~$\Od$, the sequence~$(g_k \cdot x_k)$ diverges in~$\Od$. Thus~$p \in \partial \Od$. Hence~$\partial F \cap \partial \O = \partial F \cap \partial \O_0^{**}$. Inn particular~$F$ is closed in~$\O_0^{**}$.~$\qed$
\end{proof}

The proof of Proposition~\ref{prop_finitely_generated} is contained in Sections~\ref{sect_reduction_projective} and~\ref{sect_proof_in_the_projective_case} below. The proof uses \emph{Hilbert metric} on properly convex domains of~$\mathbb{P}(\mathbb{R}^n)$; we give some reminders on this metric in the next Section~\ref{sect_hilbert_metric_def}.

\subsubsection{Reminders on the Hilbert metric}\label{sect_hilbert_metric_def} An open set~$\O \subset \mathbb{P}(\mathbb{R}^n)$ is said to be \emph{properly convex} if there exists an affine chart of~$\mathbb{P}(\mathbb{R}^n)$ containing~$\O$ as a convex bounded open subset. The \emph{Hilbert metric} on~$\O$ is then defined as follows: given two points~$x,y \in \O$, there exists a projective line~$\ell_{x,y}$ through~$x$ and~$y$. Let~$a,b$ be the endpoints of the projective interval~$\ell_{x,y} \cap \O$, such that~$a,x,y,b$ are aligned in this order. We define the \emph{Hilbert distance between~$x$ and~$y$} as the quantity~$\mathsf{H}_\O(x,y) := (a:x:y:b)$, where~$(\cdot: \cdot: \cdot : \cdot)$ is the classical cross ratio on~$\mathbb{P}(\mathbb{R}^2) \simeq \ell_{x,y}$. Recall that the \emph{cross ratio} on~$\mathbb{P}(\mathbb{R}^2)$ is~$\SL_2(\mathbb{R})$-invariant and satisfies~$([1:0] : [1:1] : [1:t] : [0:1]) = t$.

The Hilbert metric is a proper~$\Aut(\O)$-invariant geodesic metric, and projective segments are geodesics. For a more in-depth description of the Hilbert metric, see for instance \cite{de1993hilbert, papadopoulos2014handbook, goldman2022geometric}.

\subsubsection{Proof of Proposition~\ref{prop_finitely_generated} in the projective case}\label{sect_proof_in_the_projective_case} In this section, we take the notation of Proposition~\ref{prop_finitely_generated} when~$G = \PGL(n, \mathbb{R})$ for some~$n \in \mathbb{N}_{\geq 2}$ and~$\Theta$ is the first simple restricted root of~$G$, that is,~$\Fl(\g, \Theta) = \mathbb{P}(\mathbb{R}^n)$.

By Lemma~\ref{lem_replace_by_bidual}, up to considering the convex hull of~$\O$ instead of~$\O$, we may assume that~$\O$ is properly convex in~$\mathbb{P}(\mathbb{R}^{n})$, and we denote by~$\mathsf{H}_\O$ the classical Hilbert metric on~$\O$; recall Section~\ref{sect_hilbert_metric_def} for the terminology. Let us chose~$D \subset \mathbb{P}(\mathbb{R}^n)$ a properly convex domain, and denote by~$\mathcal{D}$ the set of~$G$-translates of~$D$. Note that~$\mathcal{D}$ forms a basis of neighborhoods of~$\mathbb{P}(\mathbb{R}^n)$.

 Let~$\compact \subset F$ be a compact set such that~$F = \Gamma \cdot \compact$. By compactness of~$\compact$, there exists a finite subset~$\mathcal{D}' \subset \mathcal{D}$ such that~$\compact \subset \bigcup_{D \in \mathcal{D}'} D$ and such that~$\overline{D} \subset \O$ for all~$D \in \mathcal{D}'$. Let~$\compact' := \bigcup_{D \in \mathcal{D}'} \overline{D} \subset \O$. Let~$F' := \Gamma \cdot \compact'$. Then, by construction, the group~$\Gamma$ acts cocompactly on~$F'$. Moreover, since~$\Gamma$ acts properly discontinuously on~$\O$ (Fact~\ref{fact_autom_group_proper}), the set~$F'$ satisfies the following property: 

 \begin{lem}\label{eq_condition_F'}
     For all~$x \in F'$ and for any neighborhood~$\mathcal{V}$ of~$x$ in~$\O$ with closure contained in~$\O$, there exists some integer~$m > 0$ and elements~$g_1, \dots, g_m \in \Gamma$ such that~$\mathcal{V} \cap F' \subset \bigcup_{D \in \mathcal{D}'} \bigcup_{i=1}^m g_i \cdot \overline{D}$.
 \end{lem}

 \begin{proof} Since~$\Gamma$ acts properly discontinuously on~$\O$, the set~$S' := \{ g \in \Gamma \mid  g^{-1} \cdot \mathcal{V} \cap \compact' \ne \emptyset\}$ is finite. By definition of~$F'$, we have
 \begin{equation*}
   \mathcal{V} \cap F' \subset \bigcup_{g \in S'} g \cdot \compact',
 \end{equation*}
 so we may take~$g_1, \dots, g_m$ to be the elements of~$S'$.~$\qed$
 \end{proof}

Given two points~$x,y \in F'$, we denote by~$\mathcal{C}_{x,y}^{F'}(\O)$ the set of continuous paths~$\beta:[0, 1] \rightarrow F'$ from~$x$ to~$y$ which are piecewise projective, that is, such that there exists a subdivision~$t_0 = 0 < t_1 < \cdots < t_{m} < t_{m+1} = 1$ of~$[0,1]$ such that~$\beta([t_i, t_{i+1}])$ is a projective segment contained in~$F'$. The length of such a path for the Hilbert metric~$\mathsf{H}_\O$ on~$\O$ is then:

\begin{equation}\label{eq_longueur_chaîne_Hilbert}
     \leng(\beta) = \sum_{1 \leq i \leq m} \mathsf{H}_\O\big{(}\beta(t_i), \beta(t_{i+1})\big{)}.
\end{equation}
We define the following map on~$F'$:
\begin{equation}\label{eq_ineq_hilbert}
    \delta(x,y) = \inf_{\beta \in \mathcal{C}_{x,y}^{F'}(\O)} \leng(\beta). 
\end{equation}

\begin{lem}\label{lem_props_hyperbolic} The map~$\delta$ is a proper geodesic~$\Gamma$-invariant metric on~$F'$, generating the standard topology.
\end{lem}

\begin{proof}

The map~$\delta$ obviously satisfies the triangle inequality, the symmetry property, and is~$\Gamma$-invariant. Moreover, by construction of~$F'$, for all~$x,y \in F'$, the set~$\mathcal{C}_{x,y}^{F'}(\O)$ is nonempty, so by Equation~\eqref{eq_longueur_chaîne_Hilbert}, this implies that~$\delta(x,y ) < + \infty$. The~$\Gamma$-invariance is straightforward. By definition, we have~$\delta(x,y) \geq \mathsf{H}_{\O}(x,y)$ for all~$x,y \in F'$. Since~$F'$ is closed in~$\O$ and~$\mathsf{H}_\O$ is a proper metric, this implies that~$\delta$ is a proper metric. 

Now let~$\mathcal{C}_{x,y}'(F')$ be the set of all rectifiable curves joining~$x$ and~$y$ in~$F'$. By the definition of the length of a curve, one has~$\delta(x,y) \leq \inf \left\{ \ell_{\delta}(\beta) \mid \beta \in \mathcal{C}_{x,y}'(F') \right\}$ (where~$\ell_{\delta}$ is the length for the metric~$\delta$). Since the elements of~$\mathcal{C}_{x,y}^{F'}(\O)$ are rectifiable, this last inequality is an equality. Hence~$\delta$ is a length metric. 

It remains to prove that~$\delta$ generates the standard topology on~$F'$. By Equation~\eqref{eq_ineq_hilbert} and since~$\mathsf{H}_\O$ generates the standard topology, it suffices to prove that~$\delta$ is continuous with respect to the standard topology. By the inequality

\begin{equation*}
    |\delta(x_0, y_0 ) - \delta(x, y)| \leq \delta(x_0, x ) + \delta(y_0, y) \quad \forall x_0, y_0, x, y \in F',
\end{equation*}
one only needs to prove that for any~$x_0 \in \O$ the map~$x \mapsto \delta(x_0, x)$ is continuous at~$x_0$. Let~$(x_k) \in (F')^{\mathbb{N}}$ be such that~$x_k \rightarrow x_0$. By Lemma~\ref{eq_condition_F'}, up to extracting we may assume that there exist~$g \in \Gamma$ and~$D \in \mathcal{D}'$ such that~$x_k \in g \cdot \overline{D}$ for all~$k \in \mathbb{N}$. By definition of~$F'$, we also have~$\overline{D} \subset F'$. Hence the projective segment~$[x_0, x_k]$ is in~$\mathcal{C}'_{x_0,x_k}(F')$. Hence~$\delta(x,x_k) \leq \mathsf{H}_\O(x_0,x_k)$. Since~$\mathsf{H}_\O(x_0,x_k) \rightarrow 0$ as~$k \rightarrow + \infty$, the map~$\delta(x_0, \cdot)$ is continous at~$x_0$.

We have proven that the metric space~$(F', \delta)$ is proper length metric space, it is thus geodesic.
$\qed$
\end{proof}

We have endowed~$F'$ with a~$\Gamma$-invariant, proper, geodesic metric~$\delta$, generating the standard topology (and thus locally compact). Hence by Svarc-Milnor's Lemma, the group~$\Gamma$ is finitely generated. This ends the proof of Proposition~\ref{prop_finitely_generated} in the case where~$G= \PGL(n, \mathbb{R})$ for some~$n \geq 1$ and~$\Theta = \{\alpha_1\}$.

\subsubsection{Proof of Proposition~\ref{prop_finitely_generated} in the general case}\label{sect_reduction_projective} We take the notation of Proposition~\ref{prop_finitely_generated}. We fix a representation~$(V, \rho)$ of~$G$ given by Fact~\ref{prop_ggkw}, with respect to~$\Theta$.

One can consider the connected component~$\mathcal{O}$ of~$\iota^\opp(\O^*)^*$ containing~$\iota_\rho(\O)$. It is convex, proper, and open by Fact~\ref{fact_generat}, and~$\rho(\Gamma)$-invariant. Moreover, it contains the subset~$\iota_\rho(F)$. We have:

\begin{lem}
    The set~$\iota_\rho(F)$ is closed in~$\mathcal{O}$.
\end{lem}

\begin{proof}
    Let~$(x_k) \in F^{\mathbb{N}}$ be such that~$\iota_\rho(x_k) \rightarrow y \in \mathcal{O}$. Since~$\Gamma$ acts cocompactly on~$F$, there exists a converging sequence~$(z_k) \in \O^{\mathbb{N}}$ and~$(g_k) \in \Gamma^{\mathbb{N}}$ such that~$g_k \cdot z_k = x_k$ for all~$k \in \mathbb{N}$. 
    
    Since~$F$ is closed in~$\O$, if~$(x_k)$ does not converge in~$F$, up to extraction it converges to a point~$x \in \partial \O$. Thus the sequence~$(g_k)$ diverges in~$G$, so~$(\rho(g_k))$ diverges in~$\PGL(V)$. Since~$\mathcal{O}$ contains~$\O$, the sequence~$(\iota_\rho(z_k))$ converges in~$\mathcal{O}$. Since~$\Aut_{\PGL(V)}(\mathcal{O}) := \{g \in \PGL(V) \mid g \cdot \mathcal{O} = \mathcal{O}\}$ acts properly on~$\mathcal{O}$, the sequence~$\rho(g_k) \cdot \iota_\rho(z_k)$ cannot converge in~$\mathcal{O}$, which is a contradiction. Thus~$(x_k)$ converges in~$F$, so~$y \in \iota_\rho(F)$. This proves that~$\iota_\rho(F)$ is closed in~$\mathcal{O}$.~$ \qed$
\end{proof}

Thus the discrete subgroup~$\rho(\Gamma)$ of~$\PGL(V)$ preserves a proper domain~$\mathcal{O}$ and acts cocompactly on the closed (in~$\mathcal{O}$) connected subset~$\iota_\rho(F)$ of~$\mathcal{O}$. Then by Section~\ref{sect_proof_in_the_projective_case}, the group~$\rho(\Gamma)$ is finitely generated. Now since~$G$ is simple, the representation~$\rho$ has finite kernel, so~$\Gamma$ is finitely generated. This ends the proof of Proposition~\ref{prop_finitely_generated}.

\subsection{Dynamics of groups preserving proper domains in causal flag manifolds}\label{sect_dynamics_causal} Let~$G$ be a~$\HTT$ Lie group. In this section, we investigate basic dynamical properties of subgroups~$H \leq G$ preserving a proper domain~$\O  \subset \SB(\g)$ and acting cocompactly on a closed subset~$\mathcal{C}$ of~$\O$. The main result of the section, Lemma~\ref{lem_p_conic_shilov}, will be used in the proof of the implication ``$(1) \Rightarrow (2)$" of Theorem~\ref{thm_equiv_anosov}. Lemma~\ref{lem_strongly_convex_cocompact_shilov} will be used in the implication ``$(2) \Rightarrow (3)$" of Theorem~\ref{thm_equiv_anosov}.

Given a subset~$\mathcal{C}\subset \O$ which is closed in~$\O$, the \emph{ideal boundary~$\borc$ of~$\mathcal{C}$} is the set~$\overline{\mathcal{C}} \smallsetminus \mathcal{C}= \overline{\mathcal{C}} \cap \partial \O$.

Recall that a subset~$F \subset \SB(\g)$ is \emph{transverse} if any two distinct points~$a,b \in F$  are transverse.

\begin{lem}\label{lem_p_conic_shilov}
Let~$\O \subset \SB(\g)$ be a proper domain, and let us assume that there exist a subgroup~$H \leq \Aut(\O)$ and a closed subset~$\mathcal{C}$ of~$\O$ such that:
\begin{enumerate}
    \item[*] the set~$\mathcal{C}$ is~$H$-invariant, and its ideal boundary is transverse;
    \item[*] the group~$H$ acts cocompactly on~$\mathcal{C}$;
    \item[*] One has~$\Lambda_\O^{\operatorname{orb}}(H) \subset \borc$.
\end{enumerate}
Then the following hold:
\begin{enumerate}
        \item For any~$a \in \borc$, there exist~$x_0 \in \mathcal{C}$ and~$(g_k) \in \Gamma^{\mathbb{N}}$ such that~$g_k \cdot x_0 \rightarrow a$. 
        
        \item For any~$a \in \borc$ and any sequence~$(g_k) \in \Gamma^{\mathbb{N}}$ such that there exists~$x_0 \in \mathcal{C}$ with~$g_k \cdot x_0 \rightarrow a$, the sequence~$(g_k)$ is~$\{\alpha_r\}$-contracting, with~$\{\alpha_r\}$-limit~$a$.
 
        \item The group~$\Gamma$ is~$\{\alpha_r\}$-divergent, and~$\Lambda_{\{\alpha_r\}}(\Gamma) = \partial_i \mathcal{C} = \Lambda_\O^{\operatorname{orb}}(\Gamma)$.

        \item For any~$a \in \borc$, one has~$\hyp_a \cap \ \O = \emptyset$.
    \end{enumerate}
\end{lem}

\begin{rmk}
    In Section~\ref{sect_complement}, we prove an analogue of Lemma~\ref{lem_p_conic_shilov} (see Lemma~\ref{lem_p_conical}), which holds for any flag manifold.
\end{rmk}

To prove Lemma~\ref{lem_p_conic_shilov}, we need the following property of the \emph{Caratheodory metric} introduced by Zimmer: 
\begin{fact}\label{fact_caratheodory_metric}\cite{zimmer2018proper} Let~$\g$ be any real semisimple Lie algebra of noncompact type and~$\Theta$ be a subset of the simple restricted roots of~$\g$. Let~$\O \subset \Fl(\g, \Theta)$ be a proper domain. There exists an~$\Aut(\O)$-invariant metric~$C_\O$ on~$\O$ (called the \emph{Caratheodory metric of~$\O$}) such that for any two sequences~$(x_k), (y_k) \in \O^\mathbb{N}$, if~$C_\O(x_k, y_k) \rightarrow 0$, then~$y_k \rightarrow x$ whenever~$x_k \rightarrow x \in \overline{\O}$ (even if~$x \in \partial\O$).   
\end{fact}

\begin{proof}
(1) Let~$\compact \subset \mathcal{C}$ be a compact subset such that~$\mathcal{C} = \Gamma \cdot \compact$. Let~$x_k \in \mathcal{C}^{\mathbb{N}}$ such that~$x_k \rightarrow a$. For all~$k \in \mathbb{N}$ there exists~$g_k \in \Gamma$ and~$z_k \in \mathcal{C}$ such that~$g_k \cdot z_k = x_k$. Up to extracting we may assume that there exists~$x_0 \in \compact$ such that~$z_k \rightarrow x_0$. By~$\Aut(\O)$-invariance of the Caratheodory metric, in the notation of Fact~\ref{fact_caratheodory_metric}, we have 
\begin{equation}
    C_{\O}(g_k \cdot z_k, g_k \cdot x_0) = C_{\O}(z_k, x_0) \longrightarrow 0.
\end{equation}
Thus by Fact~\ref{fact_caratheodory_metric}.(1), we have~$g_k \cdot x_0 \rightarrow z_\infty \in \Lambda_\O^{\operatorname{orb}}(\Gamma)$. 

(2) Let~$y \in  \O$ and let~$a'$ be a limit point of~$(g_k \cdot y)$, then~$a' \in \Lambda_\O^{\operatorname{orb}}(\Gamma) \subset \borc$. Let us now use \emph{photons}: a \emph{photon} is the image of a specific embedding of the circle inside~$\SB(\g)$; see \cite{galiay2024rigidity} for more details. The fact we need to use here is that if~$a,b \in \SB(\g)$ are on a same photon, then they are nontransverse (see e.g.\ \cite[Rmk 5.6]{galiay2024rigidity}). Here, there exists a \emph{chain of photons} between~$y$ and the point~$x_0$ determined in Point (1), contained in~$\O$, meaning that there exists a finite sequence~$(x_0 = x, x_1, \dots, x_N = y)$ of points of~$\O$ with~$N \in \mathbb{N}_{>0}$ such that for all~$1 \leq i \leq N$, the points~$x_i$ and~$x_{i+1}$ lie on a same photon. By induction, we may thus assume that~$y$ is on a photon through~$x_0$. In particular, it is nontransverse to~$x_0$. Hence for all~$k \in \mathbb{N}$, the points~$g_k \cdot y$ and~$g_k \cdot x_0$ are nontransverse. Thus~$a$ and~$a'$ are nontransverse. Since~$\borc$ is transverse, we have~$a = a'$.

We have proven that~$g_k \cdot y \rightarrow a$ for all~$y \in \O$. Now let~$\compact' \subset \O$ be a compact subset with nonempty interior. Then Fact~\ref{fact_caratheodory_metric}.(1) implies that~$g_k \cdot \compact' \rightarrow \{ a \}$ for the Hausdorff topology. Then by Fact~\ref{fact_KAK_divergent}.(1), the sequence~$(g_k)$ is~$\{\alpha_r\}$-contracting with~$\{\alpha_r\}$-limit~$a$. This proves (2).

(3) Let~$(g_k) \in \Gamma^\mathbb{N}$ be a sequence of distinct elements of~$H$ and let~$(\delta_k)$ be a subsequence of~$(g_k)$. Let~$a$ be a limit point of~$(\delta_k \cdot x)$. Then~$a \in \borc$ and there exists a subsequence~$(\delta_k')$ of~$\delta_k$ such that~$\delta_k ' \cdot x \rightarrow a$. Thus by Point (2), the sequence~$(\delta_k')$ is~$\{\alpha_r\}$-contracting. By Fact~\ref{fact_KAK_divergent}.(2), the sequence~$(g_k)$ is thus~$\{\alpha_r\}$-divergent. This is true for every infinite sequence of~$H$. Thus~$H$ is~$\{\alpha_r\}$-divergent. One then has~$\Lambda_{\{\alpha_r\}} (\Gamma) \subset \Lambda_\O^{\operatorname{orb}}(\Gamma) \subset \borc$, and  the converse inclusion follows from Points (1) and (2).

(4) Let~$a \in \borc$. By Points (1) and (2), there exists a~$\{\alpha_r\}$-contracting sequence~$(g_k) \in \Gamma^\mathbb{N}$ with limit~$a$. Thus there exists~$b \in \SB(\g)$ such that~$g_k \cdot y \rightarrow a$ for all~$y \notin \hyp_b$. Let~$y \in \O^* \smallsetminus \hyp_b$; such an element y exists because the set $\SB(\g) \smallsetminus \hyp_b$
is dense in~$\SB(\g)$, and~$\O^*$ is open. Then~$(g_k \cdot y)$ converges to~$a$. By~$\Gamma$-invariance and closedness of~$\O^*$, we have~$a \in \O^*$.~$\qed$
\end{proof}

Now, recall Definition~\ref{def_enveloppe_dc}.

\begin{lem}\label{lem_strongly_convex_cocompact_shilov}
    Let~$H$ be a subgroup of~$G$, and assume that there exist a proper~$H$-invariant domain~$\O$ and a closed~$H$-invariant subset~$\mathcal{C}$ of~$\O$, such that~$H$ acts cocompactly on~$\mathcal{C}$. Then~$\partial \mathcal{C} \cap \partial \O_0^{**} = \partial \mathcal{C} \cap \partial \O$. Moreover, if~$\Lambda_\O^{\operatorname{orb}}(H) \subset \partial \mathcal{C} \cap \partial \O$ and~$\borc$ is transverse, then~$\Lambda^{\operatorname{orb}}_{\O_0^{**}}(H) = \partial \mathcal{C} \cap \partial \O_0^{**}$.
\end{lem}

\begin{proof} The first part of the lemma is a reformulation of Lemma~\ref{lem_replace_by_bidual}.

Now assume that~$\borc := \partial \mathcal{C} \cap \partial \O$ is transverse and that~$\Lambda_\O^{\operatorname{orb}}(H) \subset \partial \mathcal{C} \cap \partial \O$. Let~$a \in \Lambda^{\operatorname{orb}}_{\O_0^{**}}(H)$. There exists~$y \in \Od$ and~$(g_k) \in H^{\mathbb{N}}$ such that~$g_k \cdot y \rightarrow a$. Now let~$x \in \mathcal{C}$. Since~$H$ acts properly on~$\O$, up to extracting we may assume that there exists~$b \in \borc$ such that~$g_k \cdot x \rightarrow b$. Then by Lemma~\ref{lem_p_conical}.(2), the sequence~$(g_k)$ is~$\Theta$-contracting with limit point~$b$. Let~$b' \in \SB(\g)$ be such that~$g_k \cdot z \rightarrow b$ for all~$z \in \SB(\g) \smallsetminus \hyp_{b'}$. Then~$b' \in \O^*$ by the same argument as in the proof of Lemma~\ref{lem_p_conical}, so~$\Od \cap \hyp_{b'} = \emptyset$, and in particular~$y \notin \hyp_{b'}$. Thus~$g_k \cdot y \rightarrow b$, and~$a=b$. We have proven that~$\Lambda^{\operatorname{orb}}_{\O_0^{**}}(H) \subset \partial \mathcal{C} \cap \partial \O_0^{**}$. 

For the converse inclusion, just note that one has 
\begin{equation*} \partial \mathcal{C} \cap \partial \O_0^{**} = \partial \mathcal{C} \cap \partial \O \overset{\text{Lem.\ \ref{lem_p_conic_shilov}.(3)}}{=} \Lambda_\O^{\operatorname{orb}}(H) \subset \Lambda^{\operatorname{orb}}_{\O_0^{**}}(H). \ \qed
\end{equation*}
\end{proof}

\section{Transverse groups preserving proper domains in causal flag manifolds}\label{sect_convex_transverse}

In this section, we prove Theorem~\ref{thm_equiv_anosov}. We fix a~$\HTT$ Lie group~$G$ of rank~$r \geq 1$, and take Notation~\ref{sect_notation}. In analogy with the well-known convex cores in real hyperbolic geometry and in real projective geometry, we define:

\begin{definition}\label{def_convex_core}
    Let~$\Gamma \leq G$ be a discrete subgroup and let~$\O \subset \SB(\g)$ be a proper~$\Gamma$-invariant domain. A \emph{convex core of~$(\O, \Gamma)$} is a closed connected~$\Gamma$-invariant causally convex subset~$\mathcal{C}$ of~$\O$ such that~$\borc$ contains~$\limorb$.
\end{definition}

Contrary to the projective case, where one can consider the convex hull of~$\limorb$ in~$\O$ \cite{DGKproj}, there is in our case no preferred convex core. The aim of this section is to prove Theorem~\ref{thm_equiv_anosov}. Let us restate it with the notation introduced in Section~\ref{sect_prelminaries}: 

\begin{reptheorem}[reformulation of Theorem~\ref{thm_equiv_anosov}]
Let~$G$ be a~$\HTT$ Lie group and~$\Gamma \leq G$ a discrete subgroup. The following are equivalent:

\begin{enumerate}
    \item The group~$\Gamma$ is finitely generated,~$\{\alpha_r\}$-transverse, preserves a proper domain~$\O \subset \SB(\g)$, and~$\Lambda_{\{\alpha_r\}}(\Gamma)$ contains at least three points;
    \item There exists a proper~$\Gamma$-invariant \emph{causally convex} domain~$\O \subset \SB(\g)$ such that~$\Gamma$ acts cocompactly on a convex core~$\mathcal{C}$ of~$(\O, \Gamma)$ whose ideal boundary is transverse and contains at least three points;
    \item There exists a proper~$\Gamma$-invariant \emph{dually convex} domain~$\O' \subset \SB(\g)$ such that~$\Gamma$ acts cocompactly on a convex core~$\mathcal{C}'$ of~$(\O', \Gamma)$ whose ideal boundary is transverse and contains at least three points.
\end{enumerate}
If these statements hold, then we have~$\partial_i \mathcal{C} = \Lambda_{\{\alpha_r\}}(\Gamma) = \limorb = \Lambda_{\O'}^{\operatorname{orb}}(\Gamma) = \partial_i \mathcal{C}'.$
\end{reptheorem}

\begin{ex}\label{ex_illustration_einstein}
    \begin{enumerate}
        \item Let~$\Gamma \leq G$ be an infinite finitely generated discrete subgroup with symmetric generating set~$S$, and let~$x_0 \in \mathbb{H}^{n-1}$. Let~$\mathcal{V} \subset \mathbb{H}^{n-1}$ be a bounded connected neighborhood of~$S \cdot x_0$, and let~$\mathcal{C} := \Gamma \cdot \overline{\mathcal{V}}$. The conformal identification~$\Diams_\std \simeq \mathbb{H}^{n-1} \times (- \mathbb{R})$ gives an equivariant embedding~$\mathbb{H}^{n-1} \hookrightarrow \Diams_{\mathsf{std}}$. Via this embedding, the set~$\mathcal{C}$ is a subset of~$\Ein^{n-1, 1}$, causally convex subset in~$\Diams_{\mathsf{std}}$, on which~$\Gamma$ acts cocompactly, closed in the proper domain~$\Diams_{\std}$, and such that~$ \mathcal{C}$ has strictly convex boundary. Thus~$\Gamma$ satisfies condition (2) of Theorem~\ref{thm_equiv_anosov}. This example illustrates why this condition (2) is not a good candidate for defining a notion of convex cocompactness in~$\Ein^{n-1, 1}$ (and in~$\SB(\g)$ for~$G$ a general~$\HTT$ Lie group): the notion of convexity involved does not imply enough constraints of the \emph{spatial} shape of~$\mathcal{C}$. 
    \item Note that the assumption that~$|\Lambda_{\{\alpha_r\}}(\Gamma)| > 2$ in~$(1) \Rightarrow (2)$ is necessary. For instance, take~$g \in L$ be an element with attracting fixed point~$\LP^+$ and repelling fixed point~$\LP^\opp$. Then~$\Gamma := \langle g \rangle$ preserves~$\Diams_\std$, and~$\Lambda_{\Diams_\std}^{\operatorname{orb}}(\Gamma) = \{\LP^+, \LP^\opp\}$. The only closed causally convex subset of~$\Diams_\std$ whose ideal boundary contains~$\Lambda_{\Diams_\std}^{\operatorname{orb}}(\Gamma)$ is~$\Diams_\std$ itself, and~$\Gamma$ clearly does not acts cocompactly on it.
    \end{enumerate}
\end{ex}

\begin{rmk}\label{rmk_intuition_main_thm}
Theorem~\ref{thm_equiv_anosov} may seem counter-intuitive, as it appears to state that all~$\{\alpha_r\}$-transverse groups are \emph{strongly convex cocompact}, for a natural notion of strong convex cocompactness analoguous to the one of Definition~\ref{def_convex_cocompact}. Indeed, in the projective case, as mentioned in Section~\ref{sect_reminders_projective_convex_cocompactness}, any discrete subgroup of~$\PGL(n, \mathbb{R})$ acting strongly convex cocompactly on a proper domain of~$\mathbb{P}(\mathbb{R}^n)$ is~$\{ \alpha_1\}$-Anosov, and not merely~$\{ \alpha_1\}$-transverse. Theorem~\ref{thm_equiv_anosov} actually highlights the orthogonality between the notion of causal convexity (defined in Section~\ref{sect_causal_convexity_def}), which is timelike (i.e.\ of ``maximal Maslov index") by definition, and the spacelike (i.e.\ of ``Maslov index~$0$") dynamical behavior of a group preserving a proper domain in~$\SB(\g)$, already observed in  Corollary~\ref{cor_restrictions_maslov_index}.
\end{rmk}

\subsection{Proof of implication~$(2) \Rightarrow (1)$ of Theorem~\ref{thm_equiv_anosov}}

Let~$\Gamma \leq G$ be a discrete subgroup satisfying Condition (2) of Theorem~\ref{thm_equiv_anosov}. Lemma~\ref{lem_p_conic_shilov}.(3) and the transversality of~$\borc$ imply that~$\Gamma$ is~$\{\alpha_r\}$-transverse and~$\Lambda_{\{\alpha_r\}}(\Gamma) = \partial_i \mathcal{C}$. Moreover, since~$\borc$ contains at least three points, so does~$\Lambda_{\{\alpha_r\}}(\Gamma)$. Finally, Proposition~\ref{prop_finitely_generated} implies that~$\Gamma$ is finitely generated.

\subsection{Proof of equivalence~$(2) \Leftrightarrow (3)$ of Theorem~\ref{thm_equiv_anosov}} \label{sect_proof_implication_2_3}
Let us take the notation of Theorem~\ref{thm_equiv_anosov} above.

Assume~$(2)$. Recall the dual convex hull~$\O_{0}^{**}$ of~$\O$ introduced in Definition~\ref{def_enveloppe_dc}. Then~$\O^{**}_0$ is a proper~$\Gamma$-invariant dually convex domain of~$\SB(\g)$ containing~$\mathcal{C}$, and by Lemma~\ref{lem_strongly_convex_cocompact_shilov}, we may replace~$\O$ by~$\O^{**}_0$, so we get~$(3)$. Conversely,~$(3) \Rightarrow (2)$ is just a consequence of the fact that any proper dually convex domain of~$\SB(\g)$ is causally convex (Proposition~\ref{prop_dual_implique_causall}).

\subsection{Proof of implication~$(1) \Rightarrow (2)$ of Theorem~\ref{thm_equiv_anosov}}\label{sect_proof_implication_1_2} Let us take the notation of Theorem~\ref{thm_equiv_anosov} above, and assume that (1) is satisfied. We may assume that~$\overline{\O}\subset \Affstd_\std$ and that~$\O$ is causally convex, up to considering the causally convex hull of~$\O$ (in the sense of Definition~\ref{def_enveloppe_dc}) instead of~$\O$ --- note that this causally convex hull is still~$\Gamma$-invariant, by Corollary~\ref{cor_consal_invariant_by_autom_group}.

Let~$S := \{g_1, \dots, g_N \}$ be a symmetric family of generators of~$\Gamma$, containing the identity element. Let~$x \in \O$ and let~$\mathcal{V}$ be a connected neighborhood of~$S \cdot x$ such that~$\overline{\mathcal{V}}_1\subset  \O$. The open set~$\mathcal{A} := \Gamma \cdot  \overline{\mathcal{V}}$ is thus connected,~$\Gamma$-invariant, and closed in~$\O$. Moreover, by construction, we have~$\partial_i \mathcal{A} = \Lambda_{\{\alpha_r\}}(\Gamma)$.

\begin{lem}\label{lem_omega_inter_futur_vide}
    For all~$b \in \partial_i \mathcal{A}$, we have~$\O \subset \Affstd_\std \smallsetminus (\J^+(b) \cap \J^-(b))$.
\end{lem}

\begin{proof} Assume that we have~$\O \cap \J^+(b) \ne \emptyset$, then by openness of~$\O$ we must have~$\O \cap \I^+(b) \ne \emptyset$. Now by Lemma~\ref{lem_proper_inter_hyp}, we have~$\O  \cap \hyp_b = \emptyset$ (recall that~$\partial_i \mathcal{A} = \Lambda_{\{\alpha_r\}}(\Gamma)$). Hence by connectedness of~$\O$, we have~$\O  \subset \I^+(b)$. Now let~$c \in \Lambda_{\{ \alpha_r \}}(\Gamma) \smallsetminus \{ b \}$. Again, Lemma~\ref{lem_proper_inter_hyp} gives~$\O \cap \hyp_c = \emptyset$. Since~$\Lambda_{\{\alpha_r\}}(\Gamma)$ is transverse, the point~$b$ is transverse to~$c$, so it is in~$\I^-(c)$. Since~$b \in \overline{\O}$, by connectedness of~$\O$, we thus have~$\O \subset \I^-(c)$.  Hence~$\O \subset \Diams(b,c)$. By causal convexity we must have~$\Diams(b, c) = \O$. But~$|\Lambda_{\{\alpha_r\}}(\Gamma)| > 2$, so there exists~$\eta \in \Lambda_{\{\alpha_r\}}(\Gamma) \smallsetminus \{b,c \}$. Since~$\Lambda_{\{\alpha_r\}}(\Gamma)$ is transverse, we must have~$\eta \notin \hyp_{c} \cup \hyp_{b}$. But~$\eta \in \partial \O  = \partial \Diams(b, c) \subset  \hyp_{b} \cup \hyp_{c}$, contradiction. 

Hence~$\O \cap \J^+(b) = \emptyset$, and similarly, we have~$\O \cap \J^-(b) = \emptyset$.~$\qed$
\end{proof}
  
Let~$\mathcal{C}$ be the causally convex hull of~$\mathcal{A}$, in the sense of Definition~\ref{def_causally_convex_indep}. 

\begin{lem}
    The set~$\mathcal{C}$ is closed in~$\O$.
\end{lem}

\begin{proof} Let~$\Affstd$ be an affine chart containing~$\overline{\O}$. Let~$(x_k) \in \mathcal{C}^{\mathbb{N}}$ be a sequence such that~$x_k \rightarrow x \in \O$. By Lemma~\ref{lem_convex_envelope_egal_plus_petit_convex}, for all~$k \in \mathbb{N}$ there exist~$a_k, b_k \in \mathcal{A}$ such that~$b_k \in \J_{\Affstd}^+(a_k)$ and~$x_k \in \Diams^c_{\Affstd}(a_k, b_k)$. Up to extracting, we may assume that~$a_k \rightarrow a \in \overline{\O}$ and~$b_k \rightarrow b \in \overline{\O}$. If~$a \in \partial \O$, then~$a \in \partial_i \mathcal{A} = \Lambda_{\{\alpha_r\}}(\Gamma)$. But~$x \in \J^+(a) \cap \O$, which contradicts Lemma~\ref{lem_omega_inter_futur_vide}. Hence~$a \in \O$. Similarly, we have~$b \in \O$. In particular, the points~$a,b$ are in~$\mathcal{A}$. Hence, by definition of~$\mathcal{C}$, the point~$x \in \Diams^c(a,b)$ is in~$\mathcal{C}$.~$\qed$
\end{proof}

\begin{lem}
    One has~$\borc = \Lambda_{\{\alpha_r\}}(\Gamma)$.
\end{lem}

\begin{proof}
    By construction, the set~$\borc$ contains~$\Lambda_{\{\alpha_r\}}(\Gamma)$. Let us prove the converse inclusion. Let~$a \in \borc$. There exists~$(z_k) \in (\mathcal{C} \cap \O)^{\mathbb{N}}$ such that~$z_k \rightarrow a$. For all~$k \in \mathbb{N}$, by definition of the convex hull, there exist~$x_k, y_k \in \mathcal{A}$ such that~$y_k \in \J^+(x_k)$ and~$z_k \in \Diams^c(x_k, y_k)$. Let~$b, b'$ be the limits (up to extraction) of~$(x_k)$ and~$(y_k)$. For all~$k \in \mathbb{N}$, we have~$y_k \in \J^+(z_k)$, so~$b' \in \J^+(a)$. Similarly, one has~$b \in \J^-(a)$. Hence, by transitivity, one has~$b' \in \J^+(b)$. 

    If~$b,b' \in \O$, then, by causal convexity of~$\O$, we have~$a \in \O$, contradiction. Hence either~$b$ or~$b'$ is in~$\partial \O \cap \overline{\mathcal{A}} = \Lambda_{\{\alpha_r\}}(\Gamma)$, let us say~$b$ for instance. Since~$b' \in \J^+(b)$, by Lemma~\ref{lem_omega_inter_futur_vide} we must have~$b' \in \partial \O$. Hence~$b' \in \partial_i \mathcal{A}$.
    
    Since~$\partial_i \mathcal{A} = \Lambda_{\{ \alpha_r \}}(\Gamma)$, the points~$b$ and~$b'$ are either transverse or equal. If they are transverse, then~$b' \in \I^-(b)$. Since~$b' \in \overline{\O}$, we have~$\O \cap \I^+(b) \ne \emptyset$, which contradicts Lemma~\ref{lem_omega_inter_futur_vide}. Hence we have~$b = b'$, so~$a \in \Diams^c(b,b') = \{b\}$ is equal to~$b$, and thus is in~$\Lambda_{\{\alpha_r\}}(\Gamma)$. We have proven that~$\partial_i \mathcal{C} = \Lambda_{\{\alpha_r\}}(\Gamma)$.~$\qed$
\end{proof}

Implication~$(1) \Rightarrow (2)$ of Theorem~\ref{thm_equiv_anosov} is then a direct consequence of the following lemma:

\begin{lem}
    The group~$\Gamma$ acts cocompactly on~$\mathcal{C}$.
\end{lem}
\begin{proof}
    Let~$S' := \{g \in \Gamma \mid \overline{\mathcal{V}} \cap \J^+ (g \cdot \overline{\mathcal{V}}) \ne \emptyset \text{ or } \overline{\mathcal{V}} \cap \J^- (g \cdot \overline{\mathcal
    V}) \ne \emptyset\}$. Assume that~$S'$ is infinite. Then there exists a sequence~$(g_k)$ of distinct elements of~$\Gamma$, and~$x_k, y_k \in \overline{\mathcal{V}}$ such that~$x_k \in \J^+(g_k \cdot y_k)$ (for instance) for all~$k \in \mathbb{N}$. Up to extracting one has~$x_k \rightarrow x \in \overline{\mathcal{V}}$, and since~$\Gamma$ acts properly on~$\O$ by Fact~\ref{fact_autom_group_proper}, we have~$g_k \cdot y_k \rightarrow a \in \Lambda_{\{\alpha_r\}}(\Gamma)$. Then~$x \in \J^+(a) \cap \O$, which contradicts Lemma~\ref{lem_omega_inter_futur_vide}. Thus~$S'$ is finite. Let~$\mathcal{B} := \bigcup_{g \in S'} \ g \cdot \overline{\mathcal{V}}$, and let~$\compact$ be the causally convex hull of~$\mathcal{B}$ in~$\Affstd$. Then~$\compact$ is compact because~$\mathcal{B}$ is. It is contained in~$\O$, as~$\O$ is causally convex. 

Thus it remains to prove that~$\mathcal{C} \subset \Gamma \cdot \compact$. Let~$x \in \mathcal{C}$. There exist~$a,b \in \overline{\mathcal{V}}$ and~$g_1, g_2 \in \Gamma$ such that~$x \in \Diams^c(g_1 \cdot a, g_2 \cdot b)$. Set~$x' := g_1^{-1} \cdot x$ and~$g := g_1^{-1} g_2$. Then~$x' \in \Diams^c( a, g \cdot b)$ with~$a,b \in  \overline{\mathcal{V}}$ and~$g \in S$, so~$x' \in \compact$ and~$x \in g_1 \cdot \compact$. Hence~$\mathcal{C} \subset \Gamma \cdot \compact \subset \O$. Since~$\mathcal{C}$ is closed in~$\O$, the group~$\Gamma$ acts cocompactly on~$\mathcal{C}$.~$\qed$
\end{proof}

\section[Anosov groups and proper domains]{Examples of Anosov subgroups preserving proper domains}\label{sect_examples}
In this section, taking Notation~\ref{sect_notation} we construct examples of Zariski-dense~$\{\alpha_r\}$-Anosov subgroups of~$\HTT$ Lie groups~$G$ of real rank~$r \geq 1$ preserving a proper domain in~$\SB(\g)$. We will use the structural stability of~$\{\alpha_r\}$-Anosov subgroups of~$G$ and deform well-chosen Anosov representations from~$\Gamma$ to~$L_s$. However, the property for a~$\{\alpha_r\}$-Anosov representation into~$G$ to  be the deformation of a representation of~$\Gamma$ into~$L_s$ turns out to be restrictive, as we will see in Section~\ref{sect_restrictions_L}.

\subsection{Openness property}
In this Section, we consider any noncompact real semisimple Lie group~$G$ and~$\Theta \subset \FS$ a subset of the simple restricted roots of~$G$. Next Corollary~\ref{cor_preserving_proper_domain_open} of Lemma~\ref{lem_convergence_orbit_anosov_hausdorff} states a stability property for~$\Theta$-Anosov representations preserving a proper domain in~$\Fl(\g, \Theta)$:

\begin{cor}[see Proposition~\ref{prop_preserving_proper_domain_open}]\label{cor_preserving_proper_domain_open}
  Let~$\Gamma$ be a word hyperbolic group and let~$\rho: \Gamma \rightarrow G$ be a~$\Theta$-Anosov representation. Assume that there exists a proper~$\rho(\Gamma)$-invariant domain~$\O \subset \Fl(\g, \Theta)$. Then there exists a neighborhood~$\mathcal{U}$ of~$\rho$ in~$\operatorname{Hom}_{\Theta\text{-}\mathsf{An}}(\Gamma, G)$ such that for every representation~$\rho' \in \mathcal{U}$, there exists a proper~$\rho'(\Gamma)$-invariant domain~$\O' \subset \Fl(\g, \Theta)$.
\end{cor}

\begin{proof} It suffices to prove that for any sequence~$(\rho_k) \in \operatorname{Hom}_{\Theta\text{-}\mathsf{An}}(\Gamma, G)^{\mathbb{N}}$ such that~$\rho_k \rightarrow \rho$, there exists~$k_0 \in \mathbb{N}$ such that for all~$k \geq k_0$ there exists a proper~$\rho_k(\Gamma)$-invariant domain~$\O_k \subset \Fl(\g, \Theta)$.

Let~$z_0$ be in the interior of~$\O^*$ (exists because~$\O$ is proper) and~$x_0 \in \O$. Since the interior of~$\O^*$ is proper, by Lemma~\ref{lem_proper_inter_hyp} we have~$z_0 \in \Fl(\g, \Theta)^\opp \smallsetminus \bigcup_{\eta \in \partial_\infty} \hyp_{\xi_{\rho}^\opp(\eta)}$. By Lemma~\ref{lem_convergence_orbit_anosov_hausdorff}, the sets~$X_k := \overline{\rho_k(\Gamma) \cdot z_0}$ converge to~$X := \overline{\rho(\Gamma) \cdot z_0}$ for the Hausdorff topology as~$k \rightarrow + \infty$. 

Let us fix some symmetric generating set~$S$ of~$\Gamma$ containing the identity. Since~$X \subset \overline{\O^*}$, we have~$\hyp_x \cap \O = \emptyset$ for all~$x \in X$. Thus for~$k$ large enough, there exists a connected neighborhood~$\mathcal{V}_k$ of~$\{\rho_k(g) \cdot x_0 \mid g \in S\}$ such that~$\hyp_x \cap \mathcal{V}_k = \emptyset$ for all~$x \in X_k$; Then~$\O_k := \rho_k(\Gamma) \cdot \mathcal{V}_k$ is a connected,~$\rho_k(\Gamma)$-invariant domain, such that
\begin{equation*}
    \overline{\O_k} = \bigcup_{g \in \Gamma} \overline{\rho_k(g) \cdot \mathcal{V}_k} \cup \xi_{\rho_k}(\partial_{\infty} \Gamma) \subset \Affstd_{z_0}.
\end{equation*}
Thus~$\O_k$ is proper in~$\Affstd_{z_0}$.~$\qed$
\end{proof}

\subsection{Deformations of Anosov representations into~$L_s$}\label{sect_deformations} 
In this Section, we take Notation~\ref{sect_notation}. Using Corollary~\ref{cor_preserving_proper_domain_open}, we construct Zariski-dense~$\{\alpha_r\}$-Anosov subgroups of~$\HTT$ Lie groups~$G$ preserving a proper domain in~$\SB(\g)$, hence proving Theorem~\ref{prop_existence_CC_subgroups_1}.

Recall from Corollary~\ref{cor_restrictions_maslov_index} that the assumption that~$r$ is even is necessary in Theorem~\ref{prop_existence_CC_subgroups_1}. Note that Example~\ref{ex_cas_SP_rigidity}.(2) below allows, in the case where~$r$ is a multiple of~$4$, to produce examples that are neither virtually free groups nor virtually surface groups.

From now on, we assume that~$r = 2p$ with~$p \in \mathbb{N}$. Recall from Fact~\ref{eq_root_syst_L_s} that~$\{\alpha_1, \dots, \alpha_{r-1}\}$ is a fundamental system of simple restricted roots of~$\mathfrak{l}_s$, of type~$A_{r-1}$. We denote by~$\beta_{i,j}$ the positive restricted root~$\varepsilon_i - \varepsilon_j$, for~$i>j$. Let~$(h_{i,j} := h_{\beta_{i,j}})$ be the associated co-roots (see Section~\ref{sect_real_lie_alg}). For all~$i,j$ ($i<j$) there is a unique pair~$(e_{i,j}, f_{i,j}) \in (\mathfrak{l}_s)_{\beta_{i,j}} \times (\mathfrak{l}_s)_{-\beta_{i,j}}$ such that~$(e_{i,j}, h_{i,j}, f_{i,j})$ is an~$\sll_2$-triple. Now let
\begin{equation*}
    h:= \sum_{k=1}^{p} h_{k, p+k}; \quad e:= \sum_{k=1}^p e_{k, k+p}; \quad f:= \sum_{k=1}^p f_{k, k+p}.
\end{equation*}
Then~$(e, h,f)$ is an~$\sll_2$-triple.

The~$\sll_2$-triple~$(e,h,f)$ induces a Lie algebra embedding~$\sll_2 (\mathbb{R}) \hookrightarrow \mathfrak{l}_s$, which itself induces a group homomorphism~$\tau_p:\SL(2, \mathbb{R}) \rightarrow L_s$ with kernel included in~$\{ \pm I_2 \}$. Let~$\Gamma$ be the fundamental group of a closed surface~$\mathcal{S}$ of genus~$g \geq 2$. The natural inclusion~$\Gamma \hookrightarrow \SL(2, \mathbb{R})$ induces a representation~$\rho: \Gamma \hookrightarrow \SL(2, \mathbb{R}) \overset{\rho}{\longrightarrow} L_s$, which is~$\{\alpha_{p}\}$-Anosov.

The representation~$\tau_0: \Gamma \hookrightarrow \SL(2, \mathbb{R}) \overset{\tau_p}{\longrightarrow} L_s \hookrightarrow G$ is~$\{\alpha_r\}$-Anosov by \cite[Prop.\ 4.4]{guichard2012anosov}. Note that~$\tau_0$ preserves the two diamonds of~$\SB(\g)$ with endpoints~$\LP^+$ and~$\LP^\opp$ (see Fact~\ref{fact_autom_diam}). Thus, by Corollary~\ref{cor_preserving_proper_domain_open}, any small deformation of~$\tau_0$ in~$\operatorname{Hom}_{\{\alpha_r\}\text{-}\mathsf{An}}(\Gamma, G)$ still preserves a proper domain in~$\SB(\g)$. Since~$\{\alpha_r\}$-Anosov representations are discrete and faithful, Corollary~\ref{cor_preserving_proper_domain_open} and the following proposition imply in particular Theorem~\ref{prop_existence_CC_subgroups_1}:

\begin{prop}\label{prop_existence_CC_subgroups} There exists~$g_0 \geq 2$ such that for all~$g \geq g_0$ and any neighborhood~$\mathcal{U}$ of~$\tau_0$ in~$\operatorname{Hom}_{\{\alpha_r\}\text{-}\mathsf{An}}(\Gamma, G)$, there exists a Zariski-dense representation~$\rho \in \mathcal{U}$.
\end{prop}

\begin{proof} Proposition~\ref{prop_existence_CC_subgroups} is just a consequence of a theorem of Kim--Pansu, saying that surface groups of sufficiently large genus in classical real Lie groups~$G'$ admit small Zariski-dense deformations, unless~$G'$ is Hermitian \emph{not} of tube type \cite{kim2014density, kim2015flexibility}. Indeed, here the condition that the real rank~$r$ is even implies that the~$\HTT$ Lie group~$G$ we consider is \emph{not} locally isomorphic to~$E_{7(-25)}$ (which is of real rank~$3$), so~$G$ a classical Lie group, see Table~\ref{table_shilov_bnds}.~$\qed$
\end{proof}

\begin{rmk} If~$G$ is a real semisimple Lie group and~$\Theta$ is a subset of the simple restricted roots of~$G$ such that~$\Fl(\g, \Theta)$ is an \emph{irreducible Nagano space} (as mentioned in Remark~\ref{rmk_Nagano_space}), then by \cite{nagano1965transformation}, there exists a reductive Lie subgroup~$G^*$ of~$G$ that preserves a proper domain~$\O$ in~$\Fl(\g, \Theta)$. If~$\Theta$ is self-opposite and~$\mathcal{E}_\Theta$ has no~$\inver_\Theta$-invariant element (in the notation of Section~\ref{sect_stable_connected components}), then by Proposition~\ref{prop_restrictions_maslov_index_2}, the group~$G^*$ cannot contain non-virtually cyclic discrete subgroups which are~$\Theta$-Anosov in~$G$. As part of ongoing joint work with Fanny Kassel and Yosuke Morita, we prove that in any other case, the group~$G^*$ contains surface subgroups that are~$\Theta$-Anosov in~$G$, which, after small deformation, yield examples of Zariski-dense~$\Theta$-Anosov surface groups preserving a proper domain in~$\Fl(\g, \Theta)$.
\end{rmk}

\subsection{Restrictions on Anosov representations induced from~$L$}\label{sect_restrictions_L}
In Section~\ref{sect_deformations}, taking Notation~\ref{sect_notation}, we have constructed Zariski-dense~$\{\alpha_r\}$-Anosov subgroups of~$\HTT$ Lie groups~$G$ preserving a proper domain of~$\SB(\g)$, by deformations of~$P_{\{\alpha_{r/2}\}}$-Anosov representations of surface groups into~$L_s$ into~$G$, where~$r$ is even. In the present section, we investigate the restrictions on discrete subgroups of~$G$ built from deformations of representations~$\rho_0: \Gamma \rightarrow L_s \hookrightarrow G$, which goes back to investigating those on representations~$\Gamma \rightarrow L_s$ such that~$\Gamma \rightarrow L_s \hookrightarrow G$ is~$\{\alpha_r\}$-Anosov. 

We know from Fact~\ref{eq_root_syst_L_s} that the closed positive Weyl chamber of~$L_s$ associated with its root system~$\{\alpha_1, \dots, \alpha_{r-1}\}$ is:
\begin{equation}\label{eq_aaa_L_s}
    \overline{\aaa}_{L_s}^+ := \{X \in \aaa \mid \varepsilon_i (X) \geq \varepsilon_{i+1}(X) \ \ \forall 1 \leq i \leq r-1\}.
\end{equation}
We denote by~$\mu_{L_s}: L_s \rightarrow \overline{\aaa}^+_{L_s}$ the Cartan projection (as defined in Section~\ref{sect_real_lie_alg}) of~$L_s$. The following proposition is a consequence of work of Kassel \cite{kassel2008proper}:

\begin{prop}\label{prop_restrictions_L}Let~$\Gamma$ be a non-virtually cyclic Gromov-hyperbolic group, and let~$\rho: \Gamma \rightarrow L_s^0$ be a representation. We denote by~$\iota: L_s^0 \hookrightarrow G$ the natural inclusion.

If~$r$ is odd, then~$\iota \circ \rho$ cannot be~$\{\alpha_r\}$-Anosov.

If~$r$ is even, then~$\iota \circ \rho$ is~$\{\alpha_r\}$-Anosov if and only if all but finitely many elements of~$\mu_{L_s}(\rho(\Gamma))$ are contained in 
\begin{equation*}
  \{X \in \overline{\aaa}_{L_s}^+ \mid \varepsilon_{r/2} (X) > 0 > \varepsilon_{(r/2)) +1} (X)\}, 
\end{equation*}
and in this case, the representation~$\rho$ is~$\{\alpha_{r/2}\}$-Anosov.
\end{prop}

\begin{proof} Let~$S$ be a finite symmetric generating set of~$\Gamma$ containing the identity, and~$|\cdot|_S$ be the associated word length. By \cite{gueritaud2017anosov}, the condition that~$\rho$ is~$\{\alpha_r\}$-Anosov implies that there exist $C, C' \in  \mathbb{R}$ such that
\begin{equation}\label{eq_P_anosov}
    \big{\langle} \alpha_r, \mu(\iota \circ\rho(g)) \big{\rangle} \geq C |g| - C' \quad \forall g \in \Gamma.
\end{equation}
Let us fix~$g \in \Gamma$. According to the description of~$\overline{\aaa}_{L_s}^+$ given in Fact~\ref{eq_aaa_L_s}, and since the Weyl group of~$G$ acts by \emph{signed} permutations on the~$(\varepsilon_i)$ (by Fact~\ref{eq_root_syst_L_s}), there exists a permutation~$\sigma \in \mathfrak{S}_r$ such that
\begin{equation}\label{eq_anosov_G}
    \big{|}\big{\langle} \varepsilon_i, \mu_{L_s} (\rho(g)) \big{\rangle}\big{|} = \big{\langle} \varepsilon_{\sigma (i)}, \mu (\iota \circ\rho(g)) \big{\rangle} \geq \big{\langle} \varepsilon_r, \mu(\iota \circ\rho(g)) \big{\rangle} = \frac{1}{2} \big{\langle} \alpha_r, \mu(\iota \circ\rho(g)) \big{\rangle}
\end{equation}
for all~$1 \leq i \leq r$. 

Equation~\eqref{eq_P_anosov} then gives that~$|\langle \varepsilon_i, \mu_{L_s} (\rho(g)) \rangle| \rightarrow + \infty$ as~$|g|_S \rightarrow \infty$. Then by \cite{kassel2008proper}, there exists a connected component~$C$ of~$\overline{\aaa}_{L_s}^+ \smallsetminus \left( \bigcup_{i=1}^{r-1} \Ker(\varepsilon_i) \right)$ such that for all but finitely many~$g \in \Gamma$, we have~$\mu_{L_s}(\rho(g)) \in C$, and this connected component is invariant under the opposition involution of~$L_s$. Since~$L_s$ is of type~$A_{r-1}$, we must have that~$r-2$ is even and 
\begin{equation*}
       C =  \{X \in \overline{\aaa}_{L_s}^+ \mid \varepsilon_{r/2} (X) > 0 > \varepsilon_{(r/2) +1} (X)\}.
\end{equation*} 
Hence we have 
\begin{equation*}
    \big{|}\big{\langle} \varepsilon_{r/2} - \varepsilon_{r/2+1}, \mu_{L_s} (\rho(g)) \big{\rangle}\big{|} = \big{|}\big{\langle} \varepsilon_{r/2}, \mu_{L_s} (\rho(g)) \big{\rangle}\big{|} + \big{|}\big{\langle} \varepsilon_{r/2+1}, \mu_{L_s} (\rho(g)) \big{\rangle}\big{|}
\end{equation*}
for all but finitely many~$g \in \Gamma$. By Equations~\eqref{eq_anosov_G} and~\eqref{eq_P_anosov}, this implies
\begin{equation*}
    \big{|}\big{\langle} \varepsilon_{r/2} - \varepsilon_{r/2+1}, \mu_{L_s} (\rho(g)) \big{\rangle}\big{|} \geq \big{\langle} \alpha_r, \mu(\iota \circ\rho(g)) \big{\rangle} \geq C |g| - C' \quad \forall g \in \Gamma.
\end{equation*}

Hence, by \cite{gueritaud2017anosov}, the representation~$\rho$ is~$P_{\{\alpha_{r/2}\}}$-Anosov. $\qed$
\end{proof}

\begin{ex}\label{ex_cas_SP_rigidity}
Assume that~$r = 2p$, where~$p \in \mathbb{N}_{>0}$.

\begin{enumerate}
    \item If~$G = \Sp(2r, \mathbb{R})$ and~$p$ is odd, then it is proven in~\cite{tsouvalas2020borel} that if~$\Gamma$ is vitually either a free group or a surface group.

    \item There exist~$\{\alpha_p\}$-Anosov subgroups~$\Gamma$ of~$L_s = \SL(2p, \mathbb{K})$, with~$\mathbb{K} = \mathbb{C}, \mathbb{H}$, or~$\mathbb{R}$ if~$p$ is even, which are not virtually free or surface groups, see e.g.\ \cite[Ex.\ 4.1]{tsouvalas2020borel}: an explicit example is the group~$\Gamma = \Gamma_1 * F_2$, where~$\Gamma_1$ is the fundamental group of a closed surface of genus~$g \geq 2$ and~$F_2$ is the free group on two generators. For~$g$ large enough, one can reproduce the proof of Theorem~\ref{prop_existence_CC_subgroups} verbatim, to deform~$\Gamma_1$ into a Zariski-dense~$\{\alpha_{2p}\}$-Anosov subgroup~$\Gamma_2$ of~$G$. For small enough deformations, the group~$\Gamma_2 * F_2$ still preserves a proper domain of~$\SB(\g)$, and is still~$\{\alpha_{2p}\}$-Anosov. By Table~\ref{table_shilov_bnds}, this gives Zariski-dense eXamples of~$\{\alpha_r\}$-Anosov subgroups of~$G$ preserving a proper domain in~$\SB(\g)$ that are neither virtually free nor surface groups, for~$G$ with the following Lie algebra:
    \begin{enumerate}
        \item $\spp(2r, \mathbb{R})$, for~$r$ a multiple of~$4$;
        \item $\uu(r,r)$ for~$r$ even;
        \item $\soo^*(4r)$ for~$r$ even.
    \end{enumerate}
\end{enumerate}
\end{ex}

\subsection{Anosov subgroups preserving proper domains in Einstein universe}\label{sect_anosov_einstein} In this section, we take~$G = \PO(n,2)$, and taking Notation~\ref{sect_notation}, we investigate examples of~$\{\alpha_2\}$-Anosov subgroups of~$G$ preserving proper domains in~$\SB(\g) = \Ein^{n-1, 1}$, beyond those of Section~\ref{sect_deformations}.

\begin{rmk}
    In this section, we will use the references \cite{danciger2018convex, DGKproj, smai2022anosov}. In each of these references, the convention is that the root defining~$\SB(\soo(n,2)) = \Ein^{n-1, 1}$ is $\alpha_1$ (recall Remark~\ref{rmk_racines_inversées}). Thus~$\{\alpha_2\}$-Anosov representations (resp.\ $\{\alpha_2\}$-transverse groups) in our sense coincide with~$\{\alpha_1\}$-Anosov representations (resp.\ $\{\alpha_1\}$-transverse groups) in their sense.
\end{rmk}

A subset~$F \subset \Ein^{n-1, 1}$ is said to be \emph{negative~$3$ by~$3$} (in the sense of \cite{danciger2018convex, DGKproj}) if for every triple of pairwise distinct points~$(a,b,c) \in F^3$, one has~$\indx(a,b,c) = 0$; recall from Section~\ref{sect_maslov} that~$\indx$ is the Maslov index. In the notation of Example~\ref{ex_shilov_bndrs}.(2), this is equivalent to saying that there exists a lift~$\widetilde{F}$ of~$F$ in~$\mathbb{P}(\mathbb{R}^{n,2})$ such that~$\b(u,v) < 0$ for every pair of distinct points~$u,v \in \widetilde{F}$.

Let~$\Gamma \leq G$ be an~$\{\alpha_2\}$-transverse subgroup, and let 
\begin{equation*}
    \O_\Gamma := \big{\{}x \in \Ein^{n-1, 1} \mid \indx(\xi_1, x, \xi_2) = 0 \ \ \forall \xi_1, \xi_2 \in \Lambda_{\{\alpha_2\}}(\Gamma), \ \xi_1 \ne \xi_2\big{\}}.
\end{equation*}
The set~$\O_\Gamma$ is open, but not necessarily proper or connected. Since~$\Gamma$ is~$\{\alpha_2\}$-divergent, it acts properly discontinuously on~$\O_\Gamma$. Moreover:

\begin{lem}\label{lem_invisible_domain_causally_convex} Assume that~$\Lambda_{\{\alpha_2\}}(\Gamma)$ is negative~$3$ by~$3$. Then the set~$\O_\Gamma$ is \emph{photon-convex}, that is, for every photon~$\Phot \subset \Ein^{n-1, 1}$, the intersection~$\Phot \cap \O_\Gamma$ is connected. Moreover, we have~$\overline{\Phot \cap \O} = \Phot \cap \overline \O$.
\end{lem}

\begin{proof}
Let us take the notation of Example~\ref{ex_shilov_bndrs}.(2). By \cite{danciger2018convex}, the set~$\Lambda_{\{\alpha_2\}}(\Gamma)$ lifts to a cone~$F$ of~$\mathbb{R}^{n,2}$ on which~$\b$ is negative. The set 
$$C := \mathbb{P}(\{u \in \mathbb{R}^{n,2} \mid \b(u,v) < 0 \quad \forall v \in F\})$$
is convex in an affine chart of~$\mathbb{P}(\mathbb{R}^{n, 2})$, and one has~$\O_\Gamma = \Ein^{n-1, 1} \cap C$ (see \cite{smai2022anosov}). Since every photon~$\Phot$ of~$\Ein^{n-1,1}$ is a projective line which is contained in~$\Ein^{n-1,1}$, we conclude that~$\O_\Gamma \cap \Phot = C  \cap \Phot$ is connected, and that~$\overline{\Phot \cap \O} = \overline{\Phot \cap C} = \Phot \cap \overline{C} = \Phot \cap \overline \O$.~$\qed$
\end{proof}

\begin{lem}\label{lem_not_connected}
    If~$\O_\Gamma$ is nonempty and not connected, then every connected component of~$\O_\Gamma$ is proper. 
\end{lem}

\begin{proof}
Write~$\O_\Gamma = \bigsqcup_{i \in I} \O_i$, where the~$\O_i$ are pairwise disjoint connected components of~$\O_\Gamma$. 
Let~$i \in I$, and let us assume for a contradiction that there exist~$j \in I \smallsetminus \{i\}$ and~$(x, y) \in \O_i \times \O_j$ such that~$x$ and~$y$ are not transverse. Note that this implies that~$x$ and~$y$ are on a common photon~$\Phot$. By Lemma~\ref{lem_invisible_domain_causally_convex}, the intersection~$\Phot \cap \O_\Gamma$ must be connected, and gives a continuous path from~$x$ to~$y$ in~$\O_\Gamma$. This contradicts the fact that~$x$ and~$y$ belong to distinct connected components of~$\O_\Gamma$. We have proven that~$\O_j \subset \O_i^*$ for all~$j\ne i$. Since~$\O_j$ is open, this implies that~$\O_i$ is proper in~$\Ein^{n-1, 1}$.~$\qed$
\end{proof}

\begin{lem}\label{lem_proper_domain_contained} If~$\Gamma$ is an~$\{\alpha_2\}$-Anosov subgroup of~$G$ and if~$\O_\Gamma$ is nonempty and connected, then there exists a~$\Gamma$-invariant proper domain~$\O' \subset \O_\Gamma$.
\end{lem}

\begin{proof} Let~$S$ be a finite symmetric generating set of~$\Gamma$.
Let~$\mathcal{S}\subset \O_\Gamma$ be a 
$\Gamma$-invariant \emph{acausal Cauchy hypersurface}, that is, a closed subset of~$\O_\Gamma$ such that~$\indx(x,y,z) = 0$ for every triple of pairwise distinct points~$(x,y,z) \in \mathcal{S}^3$ and such that every photon meeting~$\O_\Gamma$ meets~$\mathcal{S}$ in exactly one point. Note that, in particular, two distinct points of~$\mathcal{S}$ are always transverse. Then one has~$\partial \mathcal{S} \smallsetminus \O_\Gamma = \Lambda_{\{\alpha_1 \}} (\Gamma)$. By \cite{smai2022anosov}, such a closed subset exists. Let~$x_0 \in \mathcal{S}$. By \cite{smai2022anosov}, the action of~$\Gamma$ on~$\mathcal{S}$ is properly discontinuous (and cocompact). Thus there exists a neighborhood~$\mathcal{V}$ of~$x$ such that~$\mathcal{V} \cap \Gamma \cdot x = \{x\}$. Let~$z \in (\mathcal{V} \smallsetminus \{x\}) \cap \mathcal{S}$. Since~$\mathcal{S}$ is acausal, we have~$\mathcal{S} \cap \bigcup_{g \in \Gamma} \hyp_{g \cdot x} = \Gamma \cdot x$, so~$\mathcal{S} \smallsetminus (\Gamma \cdot x)$ is connected. Thus there exists a connected neighborhood~$\mathcal{V}'$ of~$S \cdot z$ in~$\O_\Gamma$ such that~$\overline{\mathcal{V}}' \cap \bigcup_{g \in \Gamma} \hyp_{g \cdot x} = \emptyset$. Let~$\O' := \Gamma \cdot \mathcal{V}'$. Then~$\O'$ is connected, open, and 
\begin{equation*}
    \overline{\O}' = \Gamma \cdot \overline{\mathcal{V}}' \cup \Lambda_{\{\alpha_2\}}(\Gamma) \subset \Affstd_x.
\end{equation*}
The domain~$\O'$ is thus~$\Gamma$-invariant and proper.~$\qed$
\end{proof}

Let~$n \geq 3$. Let~$\Gamma$ be a word-hyperbolic group, such that~$\partial_\infty \Gamma$ is not an~$(n-1)$-sphere, and let~$\rho: \Gamma \rightarrow \PO(n, 2)$ be an~$\{\alpha_2\}$-Anosov representation. Then by \cite[Lem.\ 5.4]{smai2022anosov}, the domain~$\O_{\rho(\Gamma)}$ is nonempty. In this case, as mentionned in the proof of Lemma~\ref{lem_proper_domain_contained} there exists a~$\rho(\Gamma)$-invariant \emph{acausal Cauchy hypersurface} of~$\O_{\rho(\Gamma)}$, that is, a closed subset of~$\O_{\rho(\Gamma)}$ such that~$\indx(x,y,z) = 0$ for every triple of pairwise distinct points~$(x,y,z) \in \mathcal{S}^3$ and such that every photon meeting~$\O_\Gamma$ meets~$\mathcal{S}$ in exactly one point. In particular, we have~$\pi_0(\mathcal{S}) = \pi_0(\O_{\rho(\Gamma)})$. Note that~$\partial \mathcal{S} = \xi_{\rho}(\partial_\infty \Gamma)$. If~$\Gamma$ has cohomological dimension~$\leq n-2$, then~$\partial_\infty \Gamma$ has covering dimension~$\leq n-3$, so~$\mathcal{S}$ is connected, and so is~$\O_{\rho(\Gamma)}$.

This analysis, together with Lemma~\ref{lem_proper_domain_contained}, gives:

\begin{prop}\label{cor_cohomological_dimension} Let~$n \geq 3$. Let~$\Gamma$ be a word-hyperbolic group such that~$\partial_\infty \Gamma$ is not an~$(n-1)$-sphere, and let~$\rho: \Gamma \rightarrow \SO(n,2)$ be an~$\{\alpha_2\}$-Anosov representation which is negative 3 by 3. If~$\Gamma$ has cohomological dimension~$\leq n-2$, then~$\rho(\Gamma)$ preserves a proper domain~$\O \subset \Ein^{n-1, 1}$.
\end{prop}

\section{Appendix : Dynamics of groups preserving proper domains in general flag manifolds}\label{sect_complement}

This section is independent from the rest of the paper. We investigate the dynamics of subgroups~$H$ of general semisimple Lie groups~$G$ preserving a proper domain in a flag manifold of~$G$ and acting cocompactly on a closed subset~$\mathcal{C}$ of this domain whose ideal boundary is \emph{strictly convex}.  This was already done in Section~\ref{sect_dynamics_causal} for~$\HTT$ Lie groups~$G$ and \emph{transverse boundary} of~$\mathcal{C}$ (instead of \emph{strictly convex}). Although this last condition on~$\borc$ can be interpreted as a form of ``strict convexity", if~$G$ is an arbitrary semisimple Lie group and~$\Theta$ is a subset of simple restricted roots of~$G$, then the flag manifold~$\Fl(\g, \Theta)$ is not necessarily self-opposite, so one can no longer speak of a ``transverse boundary". The natural notion of \emph{strict convexity} of~$\borc$ that we introduce in this section makes sense in any flag manifold and induces dynamical properties on~$H$ (see Section~\ref{sect_dynamics_general} below, in particular Lemma~\ref{lem_p_conical}). The condition we impose on~$\borc$ generalizes the one used in convex projective geometry in the context of convex cocompactness; see Definition~\ref{def_convex_cocompact}. To define this condition, we will need the notion of~\emph{dual faces}, introduced in Section~\ref{sect_dual_faces}. To state and prove Lemma~\ref{lem_p_conical}, we will need the notion of \emph{incidence}, which we recall in Section~\ref{sect_incidence} below.

\subsection{Reminders on incidence} \label{sect_incidence} We take the notations of Section~\ref{sect_parab_subgroups}. Let us fix~$G$ a real semisimple Lie group. Given a subset~$\Theta \subset \FS$, the \emph{Weyl group with respect to~$\Theta$}, denoted by~$W_\Theta$, is the subgroup of~$W$ generated by the reflections~$s_\alpha$, with~$\alpha \in \Theta$.

Let~$\Theta, \Theta' \subset \Delta$. Any element~$(x,y) \in \Fl(\g, \Theta) \times \Fl(\g, \Theta')$ can be written~$(x,y) = (g \cdot \LP_{\Theta}^+, \ gw \cdot \LP_{\Theta'})$, with~$g \in G$ and~$w \in W$. Let 
$$\pos^{(\Theta, \Theta')}: \Fl(\g, \Theta) \times \Fl(\g, \Theta') \rightarrow W_{\FS \smallsetminus \Theta} \backslash W / W_{\FS \smallsetminus \Theta'}$$
be the~$\operatorname{diag}(G)$-invariant map such that~$\pos^{(\Theta, \Theta')}(\LP_{\Theta}^+, w \cdot \LP_{\Theta'}^+) = \overline{w}$ for all~$w \in W$. This map only depends on the Lie algebra~$\g$ of~$G$. The following properties are well known:
\begin{enumerate}
    \item For any~$x \in \Fl(\g, \Theta')$, the set~$\{y \in \Fl(\g, \Theta) \mid w_0 \in \pos(x,y)\}$ is dense in~$\Fl(\g, \Theta)$.
    \item Two points~$x \in \Fl(\g, \Theta)$ and~$y \in \Fl(\g, \Theta)^\opp$ are transverse if and only if one has~$w_0 \in \pos^{(\Theta, \oppinv(\Theta))}(x,y)$.
\end{enumerate}

\begin{lem}\label{lem_positions_weyl}

     Let~$x, y \in  \Fl(\g, \Theta)^\opp$ be a triple such that~$w_0 \in \pos^{(\oppinv(\Theta), \oppinv(\Theta))}(x,y)$. Then there exists~$z \in \Fl(\g, \Theta) \smallsetminus \hyp_y$ such that~$\id \in \pos^{(\Theta, \oppinv(\Theta))}(z,x)$.

\end{lem}

\begin{proof}
Since~$\pos^{(\oppinv(\Theta), \oppinv(\Theta))}(x,y) = \overline{w_0}$, we may assume that~$x = \LP_{\Theta}^\opp$ and $y = w_0 \cdot \LP_{\Theta}^\opp$. Let us set~$z := w_0 \cdot \LP_{\Theta}^+$. Then~$z$ is transverse to~$y$, and by~$G$-invariance of~$\pos$, we have
$$\pos^{(\Theta, \oppinv(\Theta))}(z,x) = \pos^{(\Theta, \oppinv(\Theta))}(\LP_{\Theta}^+, w_0  \cdot\LP_{\Theta}^\opp) = \pos^{(\Theta, \oppinv(\Theta))}(\LP_{\Theta}^+, \LP_{\oppinv(\Theta)}^+) =  \overline{\id}.\quad \qed$$  
\end{proof}

\subsection{Dual faces}\label{sect_dual_faces} Let~$\O \subset \Fl(\g, \Theta)$ be a proper domain. Given some point~$x \in \partial \O$, the \emph{dual support to~$\O$ at~$x$} is the set~$\operatorname{Supp}_{\O}(x):= \{\xi \in \O^* \mid x \in \hyp_\xi\}$. This set is nonempty whenever~$\O$ is dually convex. The \emph{dual face of~$x$} is then the subset 
\begin{equation*}
    \Fl_\O^d(x) := \bigcap_{\xi \in \operatorname{Supp}_\O(x)} \partial \O \cap \hyp_\xi
\end{equation*}
of~$\partial \O$. The dual faces of~$\O$ are always closed. If~$\Fl(\g, \Theta)$ is the real projective space and~$\O$ is convex, then we recover the classical closed faces of~$\O$.

\subsection{Dynamics of groups preserving proper domains}\label{sect_dynamics_general}
In this section, we fix a noncompact real semisimple Lie group~$G$ and a subset of the simple restricted roots~$\Theta$ of~$G$. We establish an analogue of Lemma~\ref{lem_p_conic_shilov} that holds for any flag manifold.

As in Section~\ref{sect_dynamics_causal}, given a proper domain~$\O \subset \Fl(\g, \Theta)$ and a subset~$\mathcal{C}\subset \O$ which is closed in~$\O$, the \emph{ideal boundary~$\borc$ of~$\mathcal{C}$} is the set~$\overline{\mathcal{C}} \smallsetminus \mathcal{C}= \overline{\mathcal{C}} \cap \partial \O$.

\begin{definition}\label{def_strictly_convex_bnd}
    A closed subset~$\mathcal{C}$ of~$\O$ is said to have \emph{strictly convex boundary} if for any two distinct points~$x,y \in \borc$, one has~$x \notin \Fl_\O^d(y)$.
\end{definition}

The next Lemma~\ref{lem_p_conical} investigates the dynamical properties of a group preserving a proper domain and acting cocompactly on a closed subset with strictly convex boundary containing all the information on its dynamics, i.e.\ whose ideal boundary contains its full orbital limit set. Even not assuming any convexity assumption, we recover several properties of strongly convex cocompact subgroups of real projective space.

\begin{lem}\label{lem_p_conical}Let~$\O \subset \Fl(\g, \Theta)$ be a proper domain, and let us assume that there exist a subgroup~$H \leq \Aut(\O)$ and a closed subset~$\mathcal{C}$ of~$\O$ such that:
\begin{enumerate}
    \item[*] the set~$\mathcal{C}$ is~$H$-invariant and has strictly convex boundary;
    \item[*] the group~$H$ acts cocompactly on~$\mathcal{C}$;
    \item[*] One has~$\Lambda_\O^{\operatorname{orb}}(H) \subset \borc$.
\end{enumerate}
Then the following hold:

    \begin{enumerate}
        \item For any~$a \in \borc$, there exist~$x_0 \in \mathcal{C}$ and~$(g_k) \in H^{\mathbb{N}}$ such that~$g_k \cdot x_0 \rightarrow a$. 
        
        \item For any~$a \in \borc$ and any sequence~$(g_k) \in H^{\mathbb{N}}$ such that there exists~$x_0 \in \mathcal{C}$ with~$g_k \cdot x_0 \rightarrow a$, the sequence~$(g_k)$ is~$\Theta$-contracting, with~$\Theta$-limit~$a$.
 
        \item The group~$H$ is~$\Theta$-divergent, and~$\Lambda_{\Theta}(H) = \partial_i \mathcal{C} = \Lambda_\O^{\operatorname{orb}}(H)$.

        \item For any~$a \in \borc$, there exists~$b \in \O^*$ such that~$\id \in \pos^{(\Theta, \oppinv(\Theta))}(a,b)$. If~$\Theta$ is self-opposite, then this is equivalent to saying that~$\hyp_a \cap \ \O = \emptyset$.
    \end{enumerate}

\end{lem}

The proof of Lemma~\ref{lem_p_conical} is similar to that of Lemma~\ref{lem_p_conic_shilov}, we give it for convenience. We need again Fact~\ref{fact_caratheodory_metric}, but also another fact on the Caratheodory metric: 
\begin{fact}\label{fact_caratheodory_metric_1}\cite{zimmer2018proper} Let~$\O \subset \Fl(\g, \Theta)$ be a proper domain, and let~$C_\O$ be the Caratheodory metric on~$\O$ defined in Fact~\ref{fact_caratheodory_metric}.  Let~$(x_k), (y_k) \in \O^\mathbb{N}$ be such that~$ \sup_{k \in \mathbb{N}}C_\O(x_k, y_k) < + \infty$ and~$x_k, y_k \rightarrow x,y \in \overline{\O}$. Then for all~$\xi \in \O^*$ we have~$x \in \hyp_\xi$ if and only if~$y \in \hyp_\xi$.
\end{fact}

\begin{proof} (1) Let~$\compact \subset \mathcal{C}$ be a compact subset such that~$\mathcal{C} = H \cdot \compact$. Let~$x_k \in \mathcal{C}^{\mathbb{N}}$ such that~$x_k \rightarrow a$. For all~$k \in \mathbb{N}$ there exist~$g_k \in \Gamma$ and~$z_k \in \mathcal{C}$ such that~$g_k \cdot z_k = x_k$. Up to extracting we may assume that there exists~$x_0 \in \compact$ such that~$z_k \rightarrow x_0$. By~$\Aut_G(\O)$-invariance of the Caratheodory metrics, in the notation of Fact~\ref{fact_caratheodory_metric}, we have 
\begin{equation}
    C_{\O}(g_k \cdot z_k, g_k \cdot x_0) = C_{\O}(z_k, x_0) \longrightarrow 0.
\end{equation}
Thus by Fact~\ref{fact_caratheodory_metric}, we have~$g_k \cdot x_0 \rightarrow z_\infty \in \Lambda_\O^{\operatorname{orb}}(H)$. 

(2) Let~$y \in  \O$ and let~$a'$ be a limit point of~$(g_k \cdot y)$, then~$a' \in \Lambda_\O^{\operatorname{orb}}(H) \subset \borc$. For all~$k \in \mathbb{N}$ one has~$C_\O(g_k \cdot x, g_k \cdot y) = C_\O(x,y) < + \infty$.
By Fact~\ref{fact_caratheodory_metric_1}, this implies that~$a' \in \Fl_\O^d(a)$. But~$\borc$ is strictly convex, so~$a' = a$.

We have proven that~$g_k \cdot y \rightarrow a$ for all~$y \in \O$. Now let~$\compact' \subset \O$ be a compact subset with nonempty interior. Then Fact~\ref{fact_caratheodory_metric} implies that~$g_k \cdot \compact' \rightarrow \{ a \}$ for the Hausdorff topology. Then by Fact~\ref{fact_KAK_divergent}.(1), the sequence~$(g_k)$ is~$\Theta$-contracting with~$\Theta$-limit~$a$. This proves (2).

(3) Let~$(g_k) \in H^\mathbb{N}$ be a sequence of distinct elements of~$H$ and let~$(\delta_k)$ be a subsequence of~$(g_k)$. Let~$a$ be a limit point of~$(\delta_k \cdot x)$. Then~$a \in \borc$ and there exists a subsequence~$(\delta_k')$ of~$\delta_k$ such that~$\delta_k ' \cdot x \rightarrow a$. Thus by Point (2), the sequence~$(\delta_k')$ is~$\Theta$-contracting. By Fact~\ref{fact_KAK_divergent}.(2), the sequence~$(g_k)$ is thus~$\Theta$-divergent. This is true for every infinite sequence of~$H$. Thus~$H$ is~$\Theta$-divergent. One then has~$\Lambda_{\Theta} (H) \subset \Lambda_\O^{\operatorname{orb}}(H) \subset \borc$, and  the converse inclusion follows from Points (1) and (2).

(4) Let~$a \in \borc$. By Points (1) and (2), there exists a~$\Theta$-contracting sequence~$(g_k) \in H^\mathbb{N}$ with limit~$a$. Thus there exists~$b \in \Fl(\g, \Theta)^\opp$ such that~$g_k \cdot y \rightarrow a$ for all~$y \notin \hyp_b$. Let~$y \in \O^*$ be such that~$w_0 \in \pos^{(\oppinv(\Theta), \oppinv(\Theta))}(b,y)$; such an element y exists because the set 
$$\big{\{}y \in  \Fl(\g, \Theta)^\opp \mid w_0 \in \pos^{(\oppinv(\Theta), \oppinv(\Theta))}(b,y) \big{\}}$$ 
is dense in~$\Fl(\g, \Theta)^\opp$, and~$\O^*$ is open. Then, by Lemma~\ref{lem_positions_weyl} there exists~$a' \in \Fl(\g, \Theta) \smallsetminus \hyp_{b}$ such that~$\id \in \pos^{(\oppinv(\Theta), \Theta)}(y, a')$. Hence we have~$g_k \cdot a' \rightarrow a$. On the other hand, up to extracting we may assume that~$(g_k \cdot y)$ converges to some~$c \in \O^*$. For all~$k \in \mathbb{N}$ we have~$\id \in \pos^{(\oppinv(\Theta), \Theta)}(g_k \cdot y, g_k \cdot a')$, so taking the limit, one has~$\pos^{(\oppinv(\Theta), \Theta)}(c, a)$ (see \cite[Lem.\ 3.15]{kapovich2017dynamics}).~$\qed$
\end{proof}

In the notation of Lemma~\ref{lem_p_conical}, Point (3) implies that the ideal boundary of~$\mathcal{C}$ contains no more than the information about the dynamics of the elements of~$H$.

\bibliographystyle{alpha}
\bibliography{bibliography}

\vspace{1cm}
\small\noindent \textsc{Institut des Hautes Etudes Scientifiques,
35 rte de Chartres, 91440 Bures-sur-Yvette, France.} \emph{email address:} \texttt{galiay@ihes.fr}

\end{document}